%% file: stoch9.tex
\documentclass[10pt,a4paper]{article}
\pdfoutput=1

\usepackage{hyperref}
\providecommand{\doi}[1]{Eprint \href{http://dx.doi.org/#1}{doi:#1}}
      
\usepackage{amsfonts,amsmath,amsthm,amstext, amssymb}
\usepackage{mathrsfs}
\usepackage{vmargin}
\usepackage{enumerate}
\usepackage{graphicx}
\usepackage{color}

\newcommand{\Idt}{\mathcal{I}} 

\newcommand{\T}{\tau} 
\newcommand{\meanv}{\rho} 

\newcommand{\balpha}{\sigma}
\newcommand{\bbeta}{\delta}
\newcommand{\salpha}{\bar \sigma}
\newcommand{\sbeta}{\bar \delta}

\newcommand{\Xk}{\xi}    
\newcommand{\Ak}{\alpha} 
\newcommand{\Bk}{\beta} 
\newcommand{\Ik}{\zeta} 
\newcommand{\sx}{i}   
\newcommand{\sy}{j}   

\newcommand{\staMC}{\pi} 
\newcommand{\valx}{v} 
\newcommand{\valy}{w} 
\newcommand{\valz}{z} 
\newcommand{\sH}{\mathscr{H}}

\newcommand{\new}[1]{{\em #1}\index{#1}}
\newcommand{\comp}{\circ}
\newcommand{\R}{\mathbb{R}}
\newcommand{\inte}[1]{\overset{\;_\circ}{#1}}

\newcommand{\Q}[1]{#1^{\omega}}  

\newcommand{\fvi}[3]{#1(#2;#3)}     
\newcommand{\fvia}[4]{#1(#2;#3,#4)}     

\newcommand{\fr}[1]{\hat{#1}}     
\newcommand{\frvia}[4]{\hat{#1}(#2;#3,#4)}     

\def\tinf#1{\acute{#1}}
\def\tinfens#1{\acute{#1}}
\newcommand{\tfx}[2]{\tinf{#1}_{#2}}  
\newcommand{\tfxvia}[5]{\tinf{#1}_{#2} (#3;#4,#5)}  
\newcommand{\tA}[2]{\tinfens{\A}_{#1,#2}}
\newcommand{\tAm}[3]{\tinfens{\A}_{#1,#2,#3}}
\newcommand{\tB}[3]{\tinfens{\B}_{#1,#2,#3}}
\newcommand{\tBg}[2]{\tinfens{\B}_{#1,#2}}

\def\bg{\bar g}

\newcommand{\sigmaiter}[1]{\balpha_{#1}}

\def\sigmaold{\balpha_{k}}
\def\sigmanew{\balpha_{k+1}}
\def\sigmainit{\balpha_{0}}
\def\deltaold{\bbeta_{k}}
\def\deltanew{\bbeta_{k+1}}
\def\deltainit{\bbeta_{0}}

\def\finit{f^{(\balpha_0)}}
\def\fold{f^{(\balpha_k)}}

\def\fnew{f^{(\balpha_{k+1})}}
\def\bfnew{\overline{f^{(\balpha_{k+1})}}}

\def\vold{v^{(k)}}
\def\vnew{v^{(k+1)}}
\def\Cold{C^{(k)}}
\def\vdemi{v'}
\def\wold{w^{(k)}}
\def\wnew{w^{(k+1)}}
\def\wdemi{w'}

\def\etaold{\eta^{(k)}}
\def\etanew{\eta^{(k+1)}}
\def\tetav{t\eta+v}
\def\tpetav{(t+1)\eta+v}

\newcommand{\norm}[1]{\left\Vert {#1} \right\Vert}        
\newcommand{\vex}{\operatorname{co}}
\newcommand{\sQ}{\mathcal{Q}}
\newcommand{\sF}{\mathcal{F}}
\newcommand{\sT}{\mathcal{T}}
\newcommand{\sB}{\widetilde{B}}
\newcommand{\sG}{\mathcal{G}}
\newcommand{\sP}{\mathcal{P}}
\newcommand{\set}[2]{\{#1\mid\,#2\}}
\newcommand{\gf}{\sG^{\rm f}}
\newcommand{\gc}{\sG^{\rm c}}
\newcommand{\grf}{G} 

\newcommand{\Bl}{\mathcal{B}}

\newcommand{\sE}{\mathbb{E}}       

\newcommand{\X}{{[n]}}
\newcommand{\RX}{\R^n} 
\newcommand{\A}{A}
\newcommand{\B}{B}
\newcommand{\Ag}{A}
\newcommand{\Bg}{B}
\newcommand{\Am}{{\A}_{\mathrm M}}
\newcommand{\Bm}{{\B}_{\mathrm M}}

\newcommand{\proc}[1]{\left( {#1} \right)_{k\ge 0}}              

\newcommand{\MIN}{\text{\sc min}}
\newcommand{\MAX}{\text{\sc max}}

\newtheorem{corollary}{Corollary}
\newtheorem{thm}{Theorem}
\newtheorem{lem}[thm]{Lemma}
\theoremstyle{definition}
\newtheorem{algo}{Algorithm}

\begin{document}

\title{Policy iteration algorithm for zero-sum multichain stochastic games with mean payoff and perfect information}

\author{Marianne Akian, Jean Cochet-Terrasson, Sylvie Detournay, St\'ephane Gaubert}








\maketitle

\begin{abstract}
We consider zero-sum stochastic games with
finite state and action spaces, perfect information,
mean payoff criteria, without any irreducibility assumption on
the Markov chains associated to strategies (multichain games).
The value of such a game can be characterized by a system of nonlinear 
equations, involving the mean payoff vector and an auxiliary vector
(relative value or bias).
We develop here a policy iteration algorithm
for zero-sum stochastic games with mean payoff,
following an idea of two of the authors
(Cochet-Terrasson and Gaubert, C.\ R.\ Math.\ Acad.\ Sci.\ Paris, 2006).
The algorithm relies on a notion of 
nonlinear spectral projection (Akian and Gaubert, Nonlinear Analysis TMA, 2003),
which is analogous to the notion of reduction of super-harmonic functions
in linear potential theory. To avoid cycling, at each degenerate iteration
(in which the mean payoff vector is not improved), the new relative value is obtained by reducing the earlier one.
We show that the sequence of values and relative values
satisfies a lexicographical monotonicity property, which implies that the
algorithm does terminate.
We illustrate the algorithm by a mean-payoff version of Richman games
(stochastic tug-of-war or discrete infinity Laplacian type equation),
in which degenerate iterations are frequent.
We report numerical experiments on large scale instances, 
arising from the latter games, as well as from
monotone discretizations
of a mean-payoff pursuit-evasion deterministic differential game.
\\ 

\noindent {\em \textup{2010} Mathematics Subject Classification}:
91A20; 31C45; 47H09; 91A15; 91A43; 90C40 
\end{abstract}

\section{Introduction}
\paragraph{The mean-payoff problem for zero-sum two player multichain games}
We consider a zero-sum stochastic game with
finite state space $[n]:=\{1,\ldots,n\}$,
finite action spaces $\A$ and $\B$ for the
first and second player respectively, and perfect information.
In the case of the finite horizon problem, in which
the payoff of the game induced by a pair
of strategies of the two players is defined as the expectation of the
sum in finite horizon of the successive rewards
(the payments of the first player to the second player),
Shapley showed (see~\cite{Shapley}) that the value $v^\T_\sx$ of
the game with horizon  $T$ and initial state $\sx\in\X$
 satisfies the \new{dynamic programming equation} 
$v^{T+1}=f(v^T)$, with a dynamic programming
operator $f:\RX\to\RX$ defined as~:
\[ 
[f(v)]_\sx\, = \, \min_{a \in \A } \left(\, \max_{b \in \B} \, \left( \sum_{\sy \in \X}
P_{\sx \sy}^{ab}\, v_\sy \, + \, r_\sx^{a b} \right)\right), \qquad
 \forall \sx \in \X, v\in\RX\enspace.
\] 
Here, $r_\sx^{ab}$ and $P_{\sx\sy}^{ab}$ represent respectively
the reward in state $\sx\in\X$ and the transition probability from state
$\sx$ to state $\sy\in\X$, when the actions of the first and second players are
respectively equal to $a\in \A$ and $b\in \B$.

The above dynamic programming operator $f$ is \new{order-preserving}, 
meaning  that $\valx\leq \valy\implies f(\valx)\leq f(\valy)$ where
$\leq$ denotes the partial ordering of $\RX$, and 
\new{additively homogeneous}, meaning that it commutes
with the addition of a constant vector.
These two conditions imply that $f$ is nonexpansive in the sup-norm 
(see for instance~\cite{crandall}, see also~\cite{sgjg04} for
more background on this class of nonlinear maps).
Moreover, $f$ is \new{polyhedral}, meaning that
there is a covering of $\RX$ by finitely many
polyhedra such that the restriction of $f$ to any of these
polyhedra is affine. 
Kohlberg~\cite{kohlberg} showed that if $f$ is a polyhedral 
self-map of $\R^n$ that is nonexpansive in
some norm, then, there exist two
vectors $\eta$ and $v$ in $\R^n$ such that 
$f(\tetav)= \tpetav$,
for all $t\in\R$ large enough. 
A map of the form $t\mapsto \tetav$ is called
a \new{half-line}, and $\eta$ is its \new{slope}.
It is \new{invariant} if it satisfies the latter property.
Moreover this property is equivalent to the following system of
nonlinear equations for the couple $(\eta,v)$~:
\begin{equation}\label{system}
\left\{ \begin{array}{r l}
\eta & = \fr{f}(\eta)  \enspace,\\
\eta+v  &= \tfx{f}{\eta}(v) \enspace,
\end{array} \right. 
\end{equation}
where the maps $\fr{f}$ (the recession function) and $\tfx{f}{\eta}$ are constructed from $f$  (see Section~\ref{discrete}).

When $f$ has an invariant half-line with slope $\eta$,
the growth rate of its orbits (also called the cycle time)
$\chi(f):=\lim_{k\to\infty} f^k(\valx) /k$ exists and 
is equal to $\eta$. 
Here, $f^k$ denotes the $k$-th iterate of $f$,
and $\valx$ is an arbitrary vector of $\R^n$. 
This shows in particular that the value of the finite horizon
game satisfies $\lim_{T\to\infty} v^T_\sx/T=\eta_\sx$ for any final
reward. Moreover, $\eta_\sx$ gives the value of the game with 
initial state $\sx$, and {\em mean payoff}, that is such that 
the payoff of the game induced by a pair
of strategies of the two players is the 
Cesaro limit of the expectation of the successive rewards.
Then a vector $v$ such that $t\mapsto \tetav$ is an invariant half-line
is called a \new{relative value} of the game, or \new{bias}.
It is not unique, even up to an additive constant.

In this paper, we give an algorithm to find an invariant 
half-line, or equivalently a solution of~\eqref{system}, for general
multichain games. This allows us in particular to determine
the mean payoff, as well as optimal strategies for both players.
By {\em multichain}, we mean that there
is no irreducibility assumption
on the Markov chains associated to the strategies of the two players,
which may have in particular several invariant measures.

\paragraph{Classes of games solvable by earlier policy iteration algorithms}
Policy iteration is a general method initially introduced by 
Howard~\cite{Howard60} in the case of one player problems
(Markov decision processes).
The idea is to compute a sequence of strategies as well as certain
valuations, which serve as optimality certificates, and to
use the current valuation to improve the strategy.
The algorithm bears some resemblance with the Newton method, as the
strategy determines a sub or super-gradient of the dynamic programming
operator. The key of the analysis of policy iteration algorithms is generally
to show that the sequence of valuations which are computed satisfies a monotonicity 
property, from which it can be inferred that the same strategy is never
selected twice.
In the discounted one player case, the valuation which is
maintained by the algorithm is nothing but the value vector of the current
policy. For one-player games
with mean-payoff, in the unichain case (in which 
every stochastic matrix associated to a strategy has only one final
class), the valuation consists of the mean payoff of the current
strategy, as well as of a relative value.
In both cases, the monotonicity property
is natural (it relies on the discrete maximum principle, 
or properties of monotonicity and contraction of the dynamic programming
operator, or on the uniqueness of the invariant measure associated 
to a strategy). However, even for one player games, the extension to the 
multichain case is more difficult. It was initially proposed by
Howard~\cite{Howard60}.
The convergence of his method was established by Denardo and Fox~\cite{Denardo-Fox68}.

The idea of extending Howard algorithm to the two player case appeared
independently in the work of Hoffman and Karp~\cite{hoffman} for 
a subclass of mean-payoff games with imperfect information,
and in the work of Denardo~\cite{Denardo} for discounted games.
Both algorithms consist
of nested iterations; the internal iterations are a simplified 
version of the one player Howard algorithm.
The algorithm for discounted games appeared also,
as an adaptation of the Hoffman-Karp algorithm, 
in the work of Rao, Chandrasekaran, and Nair~\cite[Algorithm~1]{rao} 
and of Puri~\cite{Puri} (deterministic games with perfect information).
More recently,
Raghavan and Syed~\cite{raghavan} developed a related
algorithm in which strategy improvements involve
only one state at each iteration.

The Hoffman-Karp algorithm requires the game to satisfy a strong irreducibility
assumption (each stochastic matrix arising from a choice of strategies 
of the two players must be irreducible).
Without an assumption of this kind, {\em degenerate iterations}, at 
which the mean payoff vector is not improved, may
occur, and so the algorithm may cycle (we shall indeed see such an 
example in Section~\ref{exemple}).
This pathology appears in particular for the important subclass of
deterministic mean payoff games, for which the irreducibility assumption
is essentially never satisfied.

A natural idea to solve mean-payoff games, appearing for instance in the work of Puri~\cite{Puri}, 
is to apply the policy iteration algorithm of Denardo~\cite{Denardo}
or Rao, Chandrasekaran, and Nair~\cite[Algorithm 1]{rao} for 
discounted games, choosing a given discount factor $\alpha$ sufficiently
close to one, which allows one to determine the so-called
{\em Blackwell optimal policies}. For deterministic games, 
when the rewards are integers with modulus less or equal to
$W$ and the number of states is equal to $n$, 
Zwick and Paterson~\cite{ZwickPaterson96} showed
that taking $1-\alpha=1/(4n^3W)$
is sufficient to determine the mean payoff by a rounding argument.
However this requires high precision arithmetics.
In the case of stochastic games, the situation is even worse,
since examples are known in which the value of $1-\alpha$ to be used for
rounding has a denominator exponential in the number  of states.
In particular, an approach of this kind is impracticable if one
works in floating point (bounded precision) arithmetics.
Hence, it is desirable to have a policy iteration algorithm for multichain
stochastic games relying only on the computation of mean payoffs and relative values as in 
the algorithm of Howard~\cite{Howard60}  and Denardo and 
Fox~\cite{Denardo-Fox68}.

The first policy iteration algorithm
not relying on vanishing discount,
for general (multichain) deterministic 
mean payoff games, 
was apparently introduced by
Cochet-Terrasson, Gaubert and Gunawardena~\cite{gg,cras}.
The former reference
concerns the special case in which the mean payoff
is the same for all states at each iteration, whereas
the second one covers the general case, see also~\cite{cochet}.
Details of implementation, as well as experimental results
were given in~\cite{dhingra2}.
The idea of the algorithm of~\cite{gg} is to handle degenerate 
iterations by a tropical (max-plus) spectral projector.
The latter is a tropically linear retraction of the whole space onto the fixed
point set of the dynamic programming operator associated to a given strategy
of the first player.
When the mean payoff or the current strategy is not improved, the new relative value is obtained
by applying a spectral projector to the earlier relative value.
The proof of termination of the algorithm~\cite{gg} relies on a key
ingredient from tropical spectral theory, that a fixed point of a tropically
linear map is uniquely defined by its restriction to the critical nodes (the
nodes appearing infinitely often in a strategy which gives the optimal
mean payoff).
Then, it was shown in~\cite{gg} that at each degenerate iteration,
the relative value decreases, and that the set of critical nodes also
decreases, from which the termination of the algorithm can be deduced.

A related class of games consists of parity games, which can be encoded
as special deterministic games with mean payoff.
A policy improvement algorithm for parity games was introduced by
 V\"oge and Jurdzi\'nski~\cite{jurdzinski}. This algorithm differs
from the one of~\cite{gg,cras} in that instead of the relative value,
the algorithm maintains a set of relevant reachable vertices.
Other policy algorithm for parity games or deterministic mean payoff games
were introduced later on by Bjorklund, 
Sandberg and Vorobyov~\cite{BSV,bjorklund},
and by Jurdzi\'nski, Paterson, and Zwick~\cite{JPZ}.
An experimental comparison of algorithms for deterministic games 
was recently made by Chaloupka~\cite{chaloupkathese,chaloupka}, who also
gave an optimized version of the algorithm of~\cite{gg,cras,dhingra2}.

In~\cite{chancelier2}, Bielecki, Chancelier, Pliska, and Sulem
used a policy iteration algorithm to solve
a semi-Markov mean-payoff game problem with infinite
action spaces obtained 
from the discretization of a quasi-variational inequality,
based on the approach of~\cite{gg,cras}. Their algorithm
proceeds in a Hoffman and Karp fashion.

\paragraph{Policy iteration algorithm for stochastic multichain zero-sum games with mean payoff}
Inspired by the policy iteration algorithm of~\cite{gg,cras}
for deterministic games, 
Cochet-Terrasson and Gaubert proposed in~\cite{CochetGaub}
a policy iteration algorithm for general stochastic games
(see also~\cite{cochet} for a preliminary version).
The relative values are now constructed using
the nonlinear analogues of tropical spectral projectors.
These nonlinear projectors where introduced by Akian and Gaubert
in~\cite{spectral2}. They can be thought of as a nonlinear
analogues of the operation of reduction of a super-harmonic 
function, arising in potential theory.
However, no implementation details were given in the short note~\cite{CochetGaub}, in which the algorithm was stated abstractly,
in terms of invariant half-lines.



We develop here fully the idea of~\cite{CochetGaub},
and describe a policy iteration
algorithm for multichain stochastic games with mean-payoff 
(see Section~\ref{subsection-practical}). We explain
how nonlinear systems of the form~\eqref{system}
are solved at each iteration. We show in particular how
non-linear spectral projections can be computed,
by solving an auxiliary (one player) optimal stopping problem. 
This relies on the determination of the so called
{\em critical graph}, the nodes of which ({\em critical nodes})
are visited infinitely often (almost surely) by an optimal
strategy of a one player mean payoff stochastic game.
An algorithm to compute the critical
graph, based on results on~\cite{spectral2}, is given
in Section~\ref{degenerateCase}. 

We give the proof of the convergence theorem (which 
was only stated in~\cite{CochetGaub}). 
In particular, we show that the sequence $(\etaold,\vold,\Cold)$ 
consisting of the mean payoff vector, relative value vector,
and set of critical nodes, constructed 
by the algorithm satisfies a kind of lexicographical monotonicity property 
so that it converges in finite time (see Section~\ref{convergence}).
The proof of convergence exploits some results of spectral theory of convex order-preserving 
additively homogeneous maps,  by Akian and Gaubert~\cite{spectral2}.
Hence, the situation is somehow analogous to the deterministic case~\cite{gg},
the technical results of tropical (linear) spectral theory used in~\cite{gg}
being now replaced by their non-linear analogues~\cite{spectral2}.
Note also that the convergence proof of the algorithm of~\cite{gg,cras}
can be recovered as a special case of the present proof.

The convergence proof leads to a coarse exponential bound on the execution time
of the algorithm: the number of iterations of the first player is
bounded by its number of strategies, and the number of elementary
iterations (resolutions of linear systems) is bounded by
the product of the number of strategies of both players.

We also show that the specialization of this algorithm
to a one-player game gives an algorithm which is
similar to the multichain policy algorithm of Howard and Denardo and Fox,
see Section~\ref{sec-detailed-sub1}. 

Then, we discuss an example 
(see Section~\ref{exemple})
involving a variant of Richman games~\cite{loeb} (also called stochastic tug-of-war~\cite{peres}, related with discretizations
of the infinity Laplacian~\cite{oberman}), showing that degenerate iterations do occur
and that cyclic may occur with naive policy iteration rules.
Hence, the handling of degenerate iterations,
that we do here by nonlinear spectral projectors,
cannot be dispensed with.

The present algorithm has been implemented in the C library {\sc PIGAMES} by Detournay, see~\cite{detournayThese} for more information. We finally report numerical experiments (see Section~\ref{numerical})
carried out using this library, 
both on random instances of Richman type games with various numbers of states and on a class of discrete games arising from
the monotone discretization of a pursuit-evasion differential game.
These examples indicate that degenerate iterations are frequent,
so that their treatment cannot be dispensed with.
They also show that the algorithm scales well, allowing
one to solve structured instances with $10^6$ nodes and $10^7$ actions
in a few hours of CPU time on a single core processor (the bottleneck
being the resolution of linear systems).

We note that our experimental are consistent
with earlier experimental tests carried out for
simpler algorithms (dealing with one player or deterministic problems)
of which the present one is an extension.
These tests indicate that policy iteration algorithms are 
fast on typical instances (although instances with an exponential
number of iterations have been recently constructed, as
discussed in the next subsection).
Indeed, in the case of one-player deterministic games (maximal circuit mean
problem), Dasdan, Irani and Guptka~\cite{dasdan98} concluded that the instrumentation of Howard's policy iteration algorithm by Cochet-Terrasson et al.~\cite{ccggq98}, in which each iteration is carried out in linear time, was the fastest
algorithm on their test suite. Dasdan latter on developed further
optimizations of this method~\cite{dasdan}.
More recent experiments by Georgiadis,
Goldberg, Tarjan, and Werneck~\cite{georgiadis} have indicated that the class
of cycle based algorithms (to which~\cite{ccggq98,dasdan} belongs) is among the best performers, close second to the tree based method of Young, Tarjan,
and Orlin~\cite{young}. In the deterministic two player case,
Chaloupka~\cite{chaloupka} 
compared several algorithms and observed that the one of~\cite{gg,cras,dhingra2}, with the optimization
that he introduced (see also~\cite{chaloupkathese}), is 
experimentally the best performer.

\paragraph{Alternative algorithms and complexity issues}
Gurvich, Karzanov and Khachiyan~\cite{gurvich} were
the first to develop a combinatorial algorithm ({\em pumping algorithm}) to solve zero-sum deterministic games
with mean payoff. An alternative approach
was developed by Zwick and Paterson~\cite{ZwickPaterson96}, who showed that
such a game can be solved
by considering the finite horizon game for a sufficiently
large horizon, and applying a rounding argument.
Both algorithms are pseudo-polynomial.
Other algorithms, also pseudo-polynomial, based
on max-plus (tropical) cyclic projections, 
with a value iteration flavor,
have been developed
by Butkovi\v{c} and Cuninghame-Green~\cite{CGB-03}, Gaubert and Sergeev~\cite{gauser}, and Akian, Gaubert, Niti\c{c}a and Singer~\cite{AGNS10}.
Deterministic mean payoff games have been recently proved
to be equivalent to decision problems for tropical polyhedra
(the tropical analogue of linear programming)~\cite{AGGut10}.
More generally, the results there show that {\em stochastic}
games problems with mean payoff can be cast as tropical convex
(non-polyhedral) programming problems. 

The pumping algorithm of~\cite{gurvich} was recently extended
to the case of stochastic games with perfect information
by Boros, Elbassioni, Gurvich, and Makino~\cite{gurvich3}.
They showed that their algorithm is pseudo-polynomial
when the number of states of the game at which
a random transition occurs remains fixed. 
No pseudo-polynomial seems currently known without
the latter restriction.
Their algorithm applies to more general games than
the ones covered by the irreducibility assumption of Hoffman and Karp in~\cite{hoffman}, but it does not apply to all multichain games.

The question of the complexity of deterministic mean payoff games was raised
in~\cite{gurvich}, and it has remained open since that time.
Note in this respect that such games are known
to have a good characterization in the sense of Edmonds,
i.e., to be in NP$\cap$coNP. Indeed, the strategies of one player
can be used as concise certificates, as observed by Condon~\cite{condon2},
Paterson and Zwick~\cite{ZwickPaterson96}. Such games even belong
to the class UP$\cap$coUP as shown
by Jurdzi\'nski~\cite{Jur98}. 
We refer the reader to the discussion in~\cite{BSV,JPZ}
for more information. The arguments of Condon~\cite{condon2}
also imply that zero-sum stochastic games with
perfect information (and finite state and action spaces) 
belong to NP$\cap$coNP. An important subclass
of deterministic games with mean payoff consists
of parity games. These can be reduced to mean payoff
deterministic games (Puri~\cite{Puri}), which in turn can be reduced 
to discounted deterministic games. The latter ones can be reduced to 
simple stochastic games (Zwick and Paterson~\cite{ZwickPaterson96}).
In~\cite{AndMilt09}, Andersson and Miltersen generalized this result
showing that stochastic mean payoff games with perfect information,
stochastic parity games and simple stochastic games are polynomial
time equivalent. 
In particular, the decision problem corresponding to a game of any of these
classes lies in the complexity class of NP$\cap$coNP.  

Friedmann has recently constructed an example~\cite{Friedmann} showing that the
V\"oge-Jurdji\'nsky strategy improvement
algorithm for parity games~\cite{jurdzinski}
may require an exponential number of iterations.
This also yields an exponential lower bound~\cite{Friedmannlong}
for the Hoffman-Karp strategy improvement rule
for discounted deterministic games~\cite{Puri}. 
The result of Friedmann has also been extended to total
reward and undiscounted MDP by Fearnley~\cite{fearnley2,fearnley} and
to simple stochastic games and weighted discounted stochastic games
by Andersson~\cite{andersson2009}. 

Moreover, for Markov decision process with a {\em fixed} discount factor,
some upper bound on the number of policy iterations was
given in~\cite{HolzbaurMeister85}. Recently, Ye
gave a 
the first strongly polynomial bound~\cite{Ye2005,ye2010simplex}. The latter bound has been improved and
generalized to zero-sum two~player stochastic games with perfect
information factor by Hansen, Miltersen and Zwick
in~\cite{hansen2011strategy}, 
again for a fixed discount factor,
giving the first strongly polynomial
bound for these games. Note that a polynomial
bound for mean payoff games does not follow from these results
(to address the mean payoff case, we need to consider
the situation in which the discount factor tends to $1$).


Complexity results of a different nature have been established
with motivations from numerical analysis (discretizations
of PDE), exploiting in particular the relation between policy iteration
and the Newton method. The policy iteration
algorithm for one-player discounted games
with an infinite number of actions has been proved
to have a superlinear convergence around the 
solution under suitable assumptions
(see in particular the works of Puterman and 
Brumelle~\cite{PutBrum}, Akian~\cite{aki90b}, and Bokanowski, Maroso, and 
Zidani~\cite{zidaniBokanowski09}).
Chancelier,  Messaoud, and Sulem~\cite{chancelier} also considered,
in view of their application to quasi-variational inequalities, 
partially undiscounted infinite horizon problems for which they proved 
the contraction of the policy iteration algorithm.

The plan of the paper is the following: Section~\ref{discrete} is recalling
some background on stochastic zero-sum two player games,
Section~\ref{reduced} explain the construction of the nonlinear projection,
Section~\ref{section-PIgames} gives the algorithm, 
its practical version and its proof,
Section~\ref{sec-detailed} gives the ingredients of the algorithm, 
Section~\ref{exemple}
shows an example with possible cycling of iterations when not using the notion
of spectral projector, and Section~\ref{numerical} is for the numerical experiments.

\section{Two player zero-sum stochastic games with discrete time and mean payoff}
\label{discrete}

The class of two player zero-sum stochastic games was first introduced
by Shapley in the early fifties, see~\cite{Shapley}. We recall in this
section basic definitions on these games in the case of finite state
space and discrete time (for more details see \cite{Shapley,FilarVrieze,sorin}).

We consider the finite state space $\X := \{1, \dots, n\}$.  A
stochastic process $\proc{\Xk_k}$ on $\X$ gives the state of the game
at each point time $k$, called stage. At each of these stages, two 
players, called ``\MIN'' and  ``\MAX'' (the minimizer and the maximizer)
 have the possibility to influence the course of the game. 

The \new{stochastic game} $\Gamma(\sx_0)$ starting from $\sx_0 \in \X$ is played
in stages as follows. The initial state  $\Xk_0$ is equal to $\sx_0$ and known by
the players. Player \MIN\ plays first, and chooses an action
$\Ak_0$ in a set of possible actions $\A_{\Xk_0}$. Then the second player,
\MAX, chooses an action $\Bk_0$ in a set of possible actions
$\B_{\Xk_0}$. The actions of both players and the current state
determine the payment $r_{\Xk_0}^{\Ak_0\Bk_0}$ made by \MIN\ to \MAX\ and the
probability distribution $\sy\mapsto P_{\Xk_0\sy}^{\Ak_0\Bk_0}$ of the new state
$\Xk_1$. Then the game continues in the same way with state $\Xk_1$ and so
on.

At a stage $k$, each player chooses an action knowing the \new{history}
defined by $\Ik_k = (\Xk_0,\allowbreak \Ak_0,\allowbreak \Bk_0, \allowbreak \cdots,\allowbreak \Xk_{k-1}, \allowbreak\Ak_{k-1},\allowbreak \Bk_{k-1},
\Xk_k)$ for \MIN\ and $(\Ik_k, \Ak_k)$ for \MAX.  
We call a \new{strategy} or \new{policy} for a player, a rule which tells him the
action to choose in any situation. There are several classes of
strategies. 
Assume $\A_\sx \subset \Ag$ and $\B_\sx \subset \Bg$ for some sets $\Ag$ and $\Bg$.
A \new{behavior or randomized strategy} for \MIN\ (resp. \MAX) is
a sequence $\salpha := (\balpha_0, \balpha_1, \cdots)$
(resp. $\sbeta := (\bbeta_0, \bbeta_1, \cdots)$)  
where $\balpha_k$ (resp. $\bbeta_k$) is a map which to a history
$h_k = (\sx_0, a_0, b_0,
\dots, \sx_{k-1}, a_{k-1}, b_{k-1}, \sx_k)$ with $\sx_\ell \in \X$, $a_\ell \in
\A_{\sx_\ell}$, $b_\ell\in \B_{\sx_\ell}$ for $0\le \ell\le k$ (resp.\ $(h_k, a_k)$) at stage $k$ associates a probability distribution on a probability space over $\Ag$
(resp.\ $\Bg$) which support is included in the possible actions space 
 $\A_{\sx_k}$ (resp. $\B_{\sx_k}$).  
A \new{Markovian strategy} is a strategy which only depends on
the information of the current stage $k$: $\balpha_k$
(resp. $\bbeta_k$) depends only on $\sx_k$ (resp. $(\sx_k, a_k$)),  
then $\balpha_k(h_k)$ (resp. $\balpha_k(h_k, a_k)$) will be denoted
$\balpha_k(\sx_k)$ (resp. $\bbeta_k (\sx_k,a_k)$).  
It is said \new{stationary} if it is independent of $k$, then $\balpha_k$ is
also denoted by $\balpha$ and $\bbeta_k$ by $\bbeta$. 
A strategy of any type is said \new{pure} if for any stage $k$, the values
of $\balpha_k$ (resp. $\bbeta_k$) are Dirac probability measures at
certain actions in $\A_{\sx_k}$  (resp. $\B_{\sx_k}$) then we 
denote also by $\balpha_k$ (resp. $\bbeta_k$) the map which to the
history assigns the only possible action in $\A_{\sx_k}$ (resp. $\B_{\sx_k}$).

In particular, if $\salpha$ is a pure Markovian stationary strategy,
also called \new{feedback} strategy,
then $\salpha = \proc{\balpha_k}$ with $\balpha_k = \balpha$ for all
$k$ and $\balpha$ is a map $\X \rightarrow \Ag$ such that $\balpha (\sx)
\in \A_{\sx}$ for all $\sx\in\X$. In this case, we also speak about pure
Markovian stationary or feedback strategy for $\balpha$ and 
we denote by $\Am$ the
set of such maps. We adopt a similar convention for player \MAX~: $\Bm
:= \set{\bbeta : \X\times\Ag \rightarrow \Bg}{\bbeta(\sx,a) \in \B_{\sx}
  \, \forall \sx \in \X, \, a\in\A_{\sx}}$.

A strategy $\salpha=\proc{\balpha_k}$ (resp. $\sbeta=\proc{\bbeta_k}$) together with an initial state determines stochastic processes $\proc{\Ak_k}$ for the actions of \MIN, $\proc{\Bk_k}$ for the actions of \MAX\ and $\proc{\Xk_k}$ for the states of the game such that 
\begin{subequations}\label{probG1}
\begin{eqnarray}
P(\Xk_{k+1} = \sy \, | \,  \Ik_k = h_k, \Ak_k = a, \Bk_k = b) &\,=\, &
P_{\sx \sy}^{ab} \\ 
P(\Ak_{k} \in A' \, | \,  \Ik_k = h_k) &=& \balpha_k(h_k)(A') \\ 
P(\Bk_{k} \in B' \, | \,  \Ik_k = h_k, \Ak_k = a) &=& \bbeta_k(h_k,a)(B') \enspace,
\end{eqnarray}
\end{subequations}
where $\Ik_k := (\Xk_0, \Ak_0, \Bk_0, \dots, \Xk_{k-1},
\Ak_{k-1}, \Bk_{k-1} \Xk_k)$ is the history process,
$h_k$ is a history vector at time $k$:
$h_k=(\sx_0, a_0, b_0, 
\dots, \sx_{k-1}, a_{k-1}, b_{k-1}, \sx)$ and  $A'$
(resp.\ $B'$) are measurable sets in $\Ag$ (resp.\ $\Bg$). 
For instance, for each pair of feedback strategies ($\balpha$, $\bbeta$) of the two players, that is such that for $k\ge0$~:
$\balpha_k = \balpha$ with $\balpha \in \Am$ and $\bbeta_k = \bbeta$
with $\bbeta \in \Bm$, the state process 
$\proc{\Xk_k}$ is a Markov chain on $\X$ with transition probability 
\[
P(\Xk_{k+1}=\sy \, | \,  \Xk_k = \sx)\,=\, P^{\balpha(\sx)\bbeta(\sx,\balpha(\sx))}_{\sx \sy}
 \quad \text{ for } \sx, \sy \in \X\enspace,
\]
and  $\Ak_k = \balpha(\Xk_k)$ and $\Bk_k = \bbeta(\Xk_k,\Ak_k)$.

When the strategies $\salpha$ for \MIN\ and $\sbeta$ for \MAX\ are fixed,
the \new{payoff in finite horizon} $\T$ of the game $\Gamma(\sx, \salpha, \sbeta)$ starting from $\sx$ is 
\[
J^\T(\sx, \salpha, \sbeta)\,=\, \sE^{\salpha \sbeta}_{\sx} \left[ \, \sum_{k = 0}^{\T-1} r_{\Xk_k}^{\Ak_k \Bk_k} \,\right], 
\]
where $\sE^{\salpha, \sbeta}_{\sx}$ denotes the expectation for the
probability law determined by \eqref{probG1}. 
The \new{mean payoff} of the game $\Gamma(\sx, \salpha, \sbeta)$ starting
from $\sx$ is  
\[
J(\sx, \salpha, \sbeta)\,=\, \limsup_{\T \rightarrow \infty }\,
\frac{1}{\T} \, J^\T(\sx, \salpha, \sbeta).
\]
When the action spaces  $\A_\sx$ and $\B_\sx$ are finite sets
for all $\sx\in\X$,
the finite horizon game and the mean payoff game have a \new{value} which is
given respectively by:
\begin{equation} \label{valueG-MP-finite}
v^\T_\sx \,=\, \inf_{\salpha} \sup_{\sbeta} \, J^\T(\sx, \salpha, \sbeta),
\end{equation}
and
\begin{equation} \label{valueG-MP}
\meanv_\sx \,=\, \inf_{\salpha} \sup_{\sbeta} \,J(\sx, \salpha, \sbeta),
\end{equation}
for all starting state $\sx \in \X$,
where the infimum is taken among all strategies $\salpha$ for \MIN\
and the supremum is taken over all strategies $\sbeta$ for \MAX\
(see \cite{Shapley} for finite horizon games, and \cite{liggettlippman}
for mean payoff games).

Indeed, the value $v^\T$ of the finite horizon game satisfies the 
\new{dynamic programming equation}~\cite{Shapley}:
\begin{equation} \label{eq1}
v^{\T+1}_\sx\, = \, \min_{a \in \A_\sx } \left(\, \max_{b \in \B_\sx} \, \left( \sum_{\sy \in \X}
P_{\sx \sy}^{ab}\, v^{\T}_\sy \, + \, r_\sx^{a b} \right)\right), \qquad
 \forall \sx \in \X,
\end{equation}
with initial condition $v^0_\sx=0,\; \sx \in \X$.
Moreover, optimal strategies are obtained for both players by taking pure Markovian strategies $\salpha$ for \MIN\ and $\sbeta$ for \MAX\ such that, for all 
$k=0,\ldots,\T-1$, and $\sx$ in $\X$, $\balpha_k(\sx)$ attains the minimum in \eqref{eq1} with $\T$ replaced by $\T-k-1$, and that, for  all 
$k=0,\ldots,\T-1$,  $\sx$ in $\X$ and $a$ in $\A_{\sx}$,
$\bbeta_k(\sx,a)$ attains the maximum in the expression of $\fvia{F}{v^{\T-k-1}}{\sx}{a}$ defined as follows:
\begin{equation} \label{ILPI}
\fvia{F}{v}{\sx}{a} \, = \, \max_{b \in \B_\sx} \, \left( \sum_{\sy \in \X}
P_{\sx \sy}^{ab}\, \valx_\sy \, + \, r_\sx^{a b} \right).
\end{equation}

We denote by $f$ the \new{dynamic programming or Shapley operator}
from $\RX$ (that is here equivalent to $\R^\X$) to itself given by:
\begin{equation} \label{eq1opdef}
[f(v)]_{\sx}:=\fvi{F}{v}{\sx} \, := \, \min_{a \in \A_\sx } \, \fvia{F}{v}{\sx}{a}
, \qquad  \forall \sx \in \X,\; v\in \RX.
\end{equation}
Then, the dynamic programming equation of the finite horizon game writes:
\begin{equation} \label{eq1op}
v^{\T+1}\, = \, f(v^\T).
\end{equation}

The operator $f$ is \new{order-preserving}, 
i.e. $\valx\leq \valy\implies f(\valx)\leq f(\valy)$ where
$\leq$ denotes the partial ordering of $\RX$
($\valx\leq \valy$ if $\valx_{\sx}\leq \valy_{\sx}$ for all $i\in [n]$), and 
\new{additively homogeneous}, i.e. it commutes
with the addition of a constant vector, which means that 
$f(\lambda+\valx)=\lambda+f(\valx)$ for all $\lambda\in\R$ and
$\valx\in\RX$, where
$\lambda+\valx=(\lambda+\valx_{\sx})_{\sx\in \X}$.
This implies that $f$ is nonexpansive in the sup-norm 
(see for instance~\cite{crandall}).
Note that it was
observed independently by Kolokoltsov~\cite{kolokoltsov}, 
by Gunawardena and Sparrow (see~\cite{mxpnxp0})
and by Rubinov and Singer~\cite{singer00} that,
conversely, if $f:\RX\to\RX$ is order-preserving and 
additively homogeneous, then $f$ can be put in the form 
{\rm (\ref{ILPI},\ref{eq1opdef})}, with possibly infinite sets $A_i$ and $B_i$.

When the action spaces  $\A_\sx$ and $\B_\sx$ are finite sets
for all $\sx\in\X$, the map $f$ is also \new{polyhedral}, meaning that
there is a covering of $\RX$ by finitely many
polyhedra such that the restriction of $f$ to any of these
polyhedra is affine. 
Kohlberg~\cite{kohlberg} showed that if $f$ is a polyhedral 
self-map of $\RX$ that is nonexpansive in
some norm, then, there exist two
vectors $v$ and $\eta$ in $\RX$ such that 
$f(\tetav)= \tpetav$,
for all $t\in\R$ large enough. 
A map $\omega:t\in [t_0,\infty)\mapsto \tetav\in\RX$,
with $t_0\in\R$, and $\eta,v\in\RX$, is called a \new{half-line}
with slope $\eta$.
A \new{germ of half-line at infinity} is an equivalence class for the 
equivalence relation on half-lines $\omega\sim\omega'$ if
$\omega(t)=\omega'(t)$ for $t\in\R$ large enough.
A germ can be identified with the couple $(\eta,v)$ of vectors of $\RX$.
Hence, in the sequel, we shall use the expression ``half-line''
either for a map $\omega:t\in [t_0,\infty)\mapsto \tetav\in\RX$,
for its germ, or for the  couple $(\eta,v)$.
We shall say that it is \new{invariant} by $f$ if it
satisfies the latter property, that is $f(\tetav)= \tpetav$,
for all $t\in\R$ large enough.
The interest of an invariant half-line is that
its slope determines the growth
rate of the orbits of $f$,
$\chi(f):=\lim_{k\to\infty} f^k(\valy) /k$.
Here, $f^k$ denotes the $k$-th iterate of $f$,
and $\valy$ is an arbitrary vector of $\RX$. 
When it exists, the growth rate $\chi(f)$ is called 
the \new{cycle time} of $f$.
Indeed, if $f(\tetav)= \tpetav$ for $t\geq t_0$, then
$f^k(t_0\eta+v)=(t_0+k)\eta+v$ for $k\geq 0$, hence
$\lim_{k\to\infty} f^k(t_0\eta+v) /k=\eta$,
and by the nonexpansiveness of $f$, 
$\lim_{k\to\infty} f^k(\valy) /k=\eta$ for all $\valy \in\RX$,
that is $\chi(f)$ does exist and is equal to $\eta$.
For the game problem this shows that the value of the finite
horizon game has a linear growth with respect to time:
\[ 
\lim_{\T \rightarrow \infty }\, \frac{1}{\T} \,
v^\T_\sx=[\chi(f)]_\sx=\eta_\sx \enspace, 
\]
where $f$ is the Shapley operator defined in~{\rm (\ref{ILPI},\ref{eq1opdef})}.
Moreover, the value $\meanv$ of the mean payoff game defined
in~\eqref{valueG-MP}  coincides with the slope of an invariant 
half-line of $f$, and thus with the former limit:
\[
\meanv_\sx= \eta_\sx=[\chi(f)]_\sx.
\]

Finally, when the action spaces are finite, one can easily see that
the Shapley operator $f$  in~{\rm (\ref{ILPI},\ref{eq1opdef})} satisfies
for all  $\eta,v\in \RX$,
\begin{equation}  \label{Shapley-inv}
f(\tetav)=t \fr{f}(\eta)+\tfx{f}{\eta}(v)\quad\text{ for $t$ large,}
\end{equation}
where $\fr{f}$ is the \new{recession function} of $f$ (see~\cite{sgjg04}):
\begin{align}
[\fr{f}(\eta)]_{\sx}:=& \lim_{t\to\infty} \frac{[f(t\eta)]_{\sx}}{t}=
\, \min_{a \in \A_\sx } \, \max_{b \in \B_\sx} \, \left( \sum_{\sy \in \X}
P_{\sx \sy}^{ab}\, \eta_\sy \right), \quad \sx\in \X\enspace, \label{def-frec}
\end{align}
and $\tfx{f}{\eta}$ is what we shall call
the \new{tangent of $f$ at infinity} around the slope $\eta$:
\begin{subequations} \label{eta-sets}
\begin{align}
[\tfx{f}{\eta}(v)]_{\sx}:=&  \lim_{t\to\infty} [f(\tetav)-t \fr{f}(\eta)]_{\sx}
= \, \min_{a \in \tA{\sx}{\eta} } \, \max_{b \in \tB{\sx}{a}{\eta} } \, \left( \sum_{\sy \in \X}
P_{\sx \sy}^{ab}\, \valx_\sy \, + \, r_\sx^{a b} \right) \enspace, \label{def-ftg}
\end{align}
with
\begin{align}
\tA{\sx}{\eta} &:= \underset{a \in \A_\sx}{\operatorname{argmin}} 
\;\left\{ \max_{b \in   \B_\sx} \, \left( \sum_{\sy \in \X}
P_{\sx \sy}^{ab}\, \eta_\sy \right)  \right\}  \label{eta-setsA} \\ 
\tB{\sx}{a}{\eta} &:= \underset{b \in \B_\sx}{\operatorname{argmax}} \; \left\{
 \sum_{\sy \in \X} P_{\sx \sy}^{ab}\, \eta_\sy \right\} \enspace.
\end{align}
\end{subequations}
Indeed, for an action $a \in \A_{\sx}$ and $\sx\in \X$, we have
from the finiteness of the sets $\B_\sx$~:
\begin{align*} 
  \fvia{F}{\tetav}{\sx}{a} \, & = \, \max_{b \in \B_\sx} \, \left( \sum_{\sy \in \X}
P_{\sx \sy}^{ab}\, (t\eta_\sy+v_\sy) + \, r_\sx^{a b}  \right) \\
&= \, \max_{b \in \B_\sx} \, \left( t \sum_{\sy \in \X} P_{\sx \sy}^{ab}\, \eta_\sy 
+  P_{\sx \sy}^{ab}\, v_\sy + \, r_\sx^{a b}  \right) \\
& = \, \max_{b \in \B_\sx} \, \left( t \sum_{\sy \in \X}
P_{\sx \sy}^{ab}\, \eta_\sy  \right)  + \max_{b \in \tB{\sx}{a}{\eta}} \, \left( 
\sum_{\sy \in \X}
P_{\sx \sy}^{ab}\,v_\sy +\, r_\sx^{a b}  \right) \qquad \text{for $t$ large}\\
& = \, t \, \frvia{F}{\eta}{\sx}{a}+ \tfxvia{F}{\eta}{v}{\sx}{a}
\end{align*}
where one denotes~:
\begin{align} 
\frvia{F}{\eta}{\sx}{a} \, &:= \, \max_{b \in \B_\sx} \, \left( \sum_{\sy \in \X}
P_{\sx \sy}^{ab}\, \eta_\sy \right) \label{def-frec-a} \\
\tfxvia{F}{\eta}{v}{\sx}{a} \,&:= \, \max_{b \in \tB{\sx}{a}{\eta}} \, \left( \sum_{\sy \in \X}
P_{\sx \sy}^{ab}\, \valx_\sy \, + \, r_\sx^{a b} \right)\enspace .
\label{def-ftg-a}  
\end{align}
Then, using the finiteness of the sets $\A_\sx$, and
\begin{equation*} 
  [f(\tetav)]_\sx=\fvi{F}{\tetav}{\sx} \,  = \min_{a \in \A_\sx } \, \fvia{F}{\tetav}{\sx}{a} \enspace ,
\end{equation*}
one obtains Equation~\eqref{Shapley-inv}.

From~\eqref{Shapley-inv}, we deduce easily that 
$(\eta,v)$ is an invariant half-line of $f$ if, and only if,
it satisfies:
\begin{equation}\label{DPsyst2P1} 
\left\{ \begin{array}{r l}
\eta & = \fr{f}(\eta)  \enspace,\\
\eta+v  &= \tfx{f}{\eta}(v) \enspace.
\end{array} \right. 
\end{equation}
This couple system of equations is what is solved in practice,  when
one looks for the value function $\meanv=\eta$ of the mean payoff game.

\section{Reduced super-harmonic vectors}\label{reduced}
We next present the non-linear analogue of a result of classical
potential theory, on which the policy iteration algorithm  for mean
payoff games relies.
Recall that a self-map  $f$ of $\RX$ is  order-preserving
if $\valx\leq \valy\implies f(\valx)\leq f(\valy)$, where
$\leq$ denotes the partial ordering of $\RX$, and that 
it is additively homogeneous if it commutes
with the addition of a constant vector.
More generally, it is  \new{additively subhomogeneous}, if
$f(\lambda+\valx)\leq \lambda +f(\valx)$ for all $\lambda\geq 0$ and
$\valx\in\RX$.
It is easy to see that an order-preserving  self-map  $f$ of $\RX$
is additively subhomogeneous if, and only if, it is
nonexpansive in the sup-norm.  (See for instance~\cite{sgjg04} 
for more background on order-preserving additively homogeneous maps.)

We shall now recall some definitions
and results of~\cite{spectral2}, 
where the corresponding proofs
can be found, up to an extension from additively homogeneous maps
to subhomogeneous maps as in~\cite[\S 1.4]{spectral2}. 
To show the analogy with potential theory, we 
shall say that a vector $u\in \RX$ is \new{harmonic}
with respect to an order preserving, additively (sub)homogeneous map
$g$ of $\RX$ if it is a fixed point of $g$, i.e.\ if $g(u)=u$, 
and that it is
\new{super-harmonic} if $g(u)\leq u$. 
(\cite{spectral2} deals more generally with additive eigenvectors and 
super-eigenvectors).
We shall denote by $\sH(g)$ and $\sH^+(g)$ the set of
harmonic and super-harmonic vectors respectively.

We say that a  self-map $g$ of $\RX$ is \new{convex} if all its 
coordinates $g_{\sx}:\RX\to\R$ are convex functions.
Then, the \new{subdifferential}
of $g$ at a point $u\in\RX$ is defined as
\[\partial g(u):=\set{M\in \R^{n\times n}}{g(\valx)-g(u) \geq M(\valx-u),\;\forall \valx\in\RX}\enspace .\]
Hence, 
\begin{equation}\label{subdiffequiv}
\partial g(u)=\set{M\in \R^{n\times n}}{M_{i.}\in \partial g_{\sx}(u)}\enspace,
\end{equation}
where $M_{i.}$ denotes the $i$-th row of the matrix $M$.
It can be checked that when $g$ is order-preserving and additively homogeneous
(resp.\ subhomogeneous), 
$\partial g(u)$ consists of stochastic (resp.\ substochastic) matrices,
that is matrices with nonnegative entries and row sums equal to $1$ 
(resp.\ less or equal to $1$), see~\cite[Cor.\ 2.2 and (4)]{spectral2}.
Assume $g$ has a harmonic vector $u$.
We say that a node is \new{critical} if it belongs
to a recurrence class of some matrix $M\in \partial g(u)$,
where a recurrence class of $M$ means a (final) communication class $F$ 
of $M$ such that the $F\times F$ submatrix of $M$ is stochastic
(note that a recurrence class may not exist if $g$ is not additively
homogeneous), see~\cite[\S 2.3 and 1.4]{spectral2}.
One defines also the critical graph $\gc(g)$ of $g$ as the union 
of the graphs of the $F\times F$ submatrices of the matrices
$M\in \partial g(u)$, such that $F$ is a recurrence class of $M$.
The set of critical nodes and the critical graph of $g$  are
independent of the choice
of the harmonic vector $u$~\cite[Prop.\ 2.5]{spectral2}. Indeed,
when $g$ arises from a stochastic control problem with ergodic reward,
a node is critical iff it is recurrent for some stationary optimal strategy.

If $I$ is any subset of $[n]$,  we denote by $r_I$ the
restriction from $\RX$ to $\R^I$, such that $(r_I \valx)_{\sx}:=\valx_{\sx}$,
for all $i\in I$. For all $u\in \RX$, we define $u_I:=r_I u$,
and for all self-maps $g$ of $\RX$, we define $g_I:=r_I\comp g$.
Let $J:=[n]\setminus I$.
We denote by $\imath_I$ the canonical map identifying $\R^I\times \R^J$
to $\RX$, which sends $(\valy,\valz)$ to the vector $u$ such that $u_{\sx}=\valy_{\sx}$
for all $i\in I$ and $u_{\sx}=\valz_{\sx}$ for all $i\in J$.
Then, the transpose $r_I^*$ of $r_I$ is the
map from $\R^I$ to $\RX$ such that $r_I^*(\valy)=\imath_I(\valy,0)$.
Finally, for all $I,J\subset [n]$,
and for all $n\times n$ matrices $M$, we denote by 
$M_{IJ}$ the $I\times J$ submatrix of $M$.


\begin{lem} \label{lem-1}
Let $g$ denote a convex, order preserving,
and additively homogeneous self-map of $\RX$.
Assume that $u\in\RX$ is harmonic with respect to $g$.
Denote by $C$ the set of critical nodes of $g$ and by $N=[n]\setminus C$
its complement in $[n]$.
Then, the map $h:\R^N\to\R^N$ with
$h(\valy):= (r_N\comp g\comp \imath_N) (\valy,u_C)$
has a unique fixed point.
\end{lem}
\begin{proof}
Since the map $g$ is order preserving and additively homogeneous, 
it is nonexpansive in the sup-norm, and so, the map
$h$ is also order preserving and nonexpansive in the sup-norm, hence it is
additively subhomogeneous.
Since $u$ is harmonic with respect to $g$, that is a fixed point of $g$,
$u_N$ is a fixed point of the map $h$.
A classical result of convex analysis (Theorem~23.9 of~\cite{ROCK})
shows in particular that if $G$ is a finite valued convex function
defined on $\mathbb{R}^d$, if $A$ is a linear map $\mathbb{R}^p\to\mathbb{R}^d$,
and if $H(\valx):=G(A\valx)$, then, $\partial H(\valx)= A^*\partial{G}(A\valx)$.
Applying this result to every convex map $G_{\sx}$ defined on
$\RX$ such that $G_{\sx}(\valy):=g_{\sx}(\valy+\imath_N(0,u_C))$,
with $\sx\in N$, and to the linear map $A=r_N^*$, 
we deduce that
$\partial h_{\sx} (u_N)$ is the projection on $\R^N$
of the subdifferential  of $G_{\sx}$ at the point $r_N^*(u_N)$,
or equivalently of the subdifferential of
$g_{\sx}$ at the point $r_N^*(u_N)+\imath_N(0,u_C)=\imath(u_N,u_C)=u$.
Using~\eqref{subdiffequiv}, this implies that
$\partial h(u_N)=\set{M_{NN}}{M\in \partial g(u)}$.
Since $g$ is order
preserving and additively homogeneous, the elements of $\partial g(u)$ are
stochastic matrices, and by the above equality, or since $h$ is
order preserving and additively subhomogeneous, the elements of
$\partial h(u_N)$ are substochastic matrices.
Recall that the set of critical nodes of $h$ 
is defined as the set of nodes that belong to a final class
$F$ of some matrix $P\in \partial h(u_N)$ satisfying 
that $P_{FF}$ is stochastic. Denote by $F$ such a class.
We have $F\subset N$. Moreover, 
since $\partial h(u_N)=\set{M_{NN}}{M\in \partial g(u)}$, 
we can find a matrix $Q\in \partial g(u)$
the $N\times N$ submatrix of which, $Q_{NN}$, coincides
with $P$. Since $F\subset N$, $Q_{FF}$ coincides with $P_{FF}$.
Hence $Q_{FF}$ is a stochastic matrix, which implies 
that $F$ is a recurrent class of $Q$. 
 This shows that the nodes of $F$ are critical
nodes of $g$, which contradicts the fact that the set
of critical nodes is $C$ since $F\subset N=[n]\setminus C$.
 Therefore the set of critical nodes of $h$ is empty.
It follows from Corollary~1.3 of~\cite{spectral2} that $h$ has a unique
fixed point.
\end{proof}
We shall need the following result of~\cite{spectral2}.
\begin{lem}[{\cite[(7) and Lemma~3.3]{spectral2}}]\label{lem-withmarianne}
Let $g$ be a convex order-preserving
additively homogeneous self-map of $\RX$, with
at least one harmonic vector. Denote by $C$ the set of critical
nodes. If $u$ is super-harmonic with respect to
 $g$, then $g(u)=u$ on $C$, and
 $g^\omega(u):=\lim_{k\to\infty} g^k(u)$ exists, is harmonic
with respect to $g$ and coincides with $u$ on $C$.
Moreover, the map $\Q{g}:\sH^+(g)\to \sH(g)$ is order-preserving, additively
homogeneous, convex, and is a projector.
\end{lem}
The following result gives other characterizations of  $\Q{g}(u)$ that
allows one to compute it efficiently.
\begin{thm}\label{th-1}
Let $g$ denote a convex, order preserving, and additively homogeneous
self-map of $\RX$.
Assume that $g$ admits at least one harmonic vector. Let $C$ denote the
set of critical nodes of $g$, and let $N$ denote
its complement in $[n]$, $N=[n]\setminus C$.
For a super-harmonic vector $u$, 
the following conditions define uniquely the same
vector $v$:
\begin{enumerate}
\item[{\rm (i)}] $v=\Q{g}(u):=\lim_{k\to\infty} g^k(u)$;
\item[{\rm (ii)}] $v$ is harmonic and coincides with $u$ on $C$;
\item[{\rm (iii)}] $v$ coincides with $u$ on $C$ and its restriction
to $N$ is a fixed point of the map $h:\valy\mapsto (r_N\comp g\comp \imath_N)(\valy,u_C)$;
\item[{\rm (iv)}] $v$ is the smallest super-harmonic vector that
dominates $u$ on $C$.
\end{enumerate}
\end{thm}
\begin{proof}
\noindent (i)$\Rightarrow$(ii): This follows from 
Lemma~\ref{lem-withmarianne}.

\noindent (ii)$\Rightarrow$(iii):
Assume that the vector $v$ is harmonic and coincides with $u$ on $C$
and let $h$ be defined as in Point (iii).
Then, $\valx_N=h(\valx_N)$, showing that $v_N$ is a fixed point of $h$.

\noindent (iii)$\Rightarrow$(i):
Let $v$ and $h$ be as in Point (iii), hence $\valx_C=u_C$ and
 $\valx_N$ is a fixed point of $h$.
By Lemma~\ref{lem-withmarianne},  $w:=\Q{g}(u)$ is harmonic with respect to 
$g$ and $w_C=u_C$. 
Applying Lemma~\ref{lem-1} to $g$ and $w$ (instead of $u$), 
and using  $w_C=u_C$, we  get that
the fixed point of $h$ is unique, and thus equal to $w_N$.
This shows that $\valx_N=w_N$, and since $\valx_C=u_C=w_C$, we get that
$\valx=w=\Q{g}(u)$, that is Point (i).

\noindent (ii)$\Rightarrow$(iv):
Let $v$ be as in Point (ii).
Since $v$ is harmonic and coincides with $u$ on $C$, it is super-harmonic and
dominates $u$ on $C$.
By ((ii)$\Rightarrow$(iii)),  $v_N$ is a fixed point of $h$,
with $h$ as in Point (iii).
Assume now that $w$ is super-harmonic and dominates $u$ on $C$, that is
$w_C\geq u_C$. Then,
$w\geq g(w)$, and since $g$ is order preserving,
$w_N\geq g_N(w_N,w_C)\geq g_N(w_N,u_C)=h(w_N)$. Since $h$ is order-preserving,
we deduce from $w_N\geq h(w_N)$ that 
$w_N\geq h^1(w_N)\geq h^2(w_N)\geq \cdots$.
Since $h$ is nonexpansive and admits a fixed point,
every orbit of $h$ is bounded. 
Hence, $h^k(w_N)$ has a limit as $k$ tends to infinity,
and this limit is a fixed point of $h$.
Applying Lemma~\ref{lem-1} to $g$ and $v$ (instead of $u$),
and using  $v_C=u_C$, we get that the fixed point of $h$ is unique and
equal to $v_N$. It follows that $w_N\geq v_N$.
Since $v$ coincides with $u$ on $C$ and $w_C\geq u_C$,
we deduce that $w\geq v$.
This shows that  $v$ is the smallest super-harmonic vector that
dominates $u$ on $C$.

\noindent (iv)$\Rightarrow$(ii):
Let $v$ be a minimal super-harmonic vector that
dominates $u$ on $C$ (or the smallest one if it exists).
Since $v$ is a super-harmonic vector, that is $g(v)\leq v$,
and $g$ is order-preserving, we get that
$g(g(v))\leq g(v)$, which shows that $g(v)$ is also super-harmonic.
Moreover, by Lemma~\ref{lem-withmarianne},
$g(v)$ coincides with $v$ on $C$, hence it dominates $u$ on $C$.
Since $g(v)\leq v$, the  minimality of $v$ implies $g(v)=v$, 
which shows that $v$ is harmonic.
Since $u$ and $v$ are super-harmonic vectors and $g$ is order-preserving,
we get that the infimum $v\wedge u$ of $v$ and $u$ 
is also a super-harmonic vector. Since  $v$ 
dominates $u$ on $C$, we get that $v\wedge u$ equals $u$ on $C$.
Hence by the minimality of $v$, and $v\wedge u\leq v$,  we obtain that
$v=v\wedge u$, hence $v\leq u$. This implies that $v$ equals $u$ on $C$,
hence $v$ satisfies (ii).
\end{proof}

Let $\Q{g}$ be defined as in Theorem~\ref{th-1}.
When $g(\valx)=M\valx$ is a linear operator,
and $M$ is a stochastic matrix, $\Q{g}(u)$ coincides
with the \new{reduced}
super-harmonic vector of $u$ with respect to the set $C$.
When $g$ is a max-plus linear operator, the operator $\Q{g}$ 
coincides with the {\em spectral projector} which has been defined in the
max-plus literature, see~\cite{gg}.
For this reason, we call $\Q{g}$ the (nonlinear)
\new{spectral projector} of $g$.

We now define a spectral projector acting on half-lines.
We assume that $g$ is a polyhedral, convex, order preserving, 
and additively homogeneous self-map of $\RX$. 
This implies in particular that for all $i\in\X$,
the domain of the Legendre-Fenchel transform $g_i^*$ of the coordinate $g_i$ 
of $g$ is included in the set of stochastic vectors, 
and that $g_i$ is the Legendre-Fenchel transform of $g_i^*$,
hence can be put in the same form as in~\eqref{ILPI}:
\begin{equation} \label{convpoly}
g_{\sx}(v) \, = \, \max_{b \in \B_\sx} \, \left( \sum_{\sy \in \X}
P_{\sx \sy}^{b}\, \valx_\sy \, + \, r_\sx^{ b} \right)\enspace,
\end{equation}
where, for all $\sx\in\X$, $P_{\sx .}^b\in\RX$ is a stochastic vector,
$r_\sx^b\in\R$, and $B_i$ is
the domain of $g_i^*$, see~\cite[Prop.\ 2.1 and Cor.\ 2.2]{spectral2}.
Since the map $g_i$ is polyhedral, the domain of $g_i^*$
is also a polyhedral convex set, see~\cite[Th.\ 19.2]{ROCK},
and since it is included in the set of stochastic vectors, it is
compact, hence it is the convex envelope of the finite set 
of its extremals.
Then, in~\eqref{convpoly}, $B_i$ can be replaced by this finite set.

Since $g$ is polyhedral, order preserving, and additively homogeneous, 
we get by Kohlberg theorem~\cite{kohlberg} recalled in Section~\ref{discrete},
that  $g$ has an invariant half-line
 $(\eta,v)$, $\eta$ is necessarily equal to $\chi(g)$, and
by~\eqref{DPsyst2P1}, $v$ and $\eta$ satisfy  $\eta = \fr{g}(\eta)$ and
$\eta+v  = \tfx{g}{\eta}(v)$, where $ \fr{g}$ and $\tfx{g}{\eta}$ are defined
in~\eqref{def-frec} and~\eqref{def-ftg} respectively.
When $g$ is given by~\eqref{convpoly}, these maps can be rewritten as:
\begin{align}
[\fr{g}(\eta)]_{\sx} = \,\max_{b \in \B_\sx} \, \left( \sum_{\sy \in \X}
P_{\sx \sy}^{b}\, \eta_\sy \right), \quad \sx\in \X\enspace, \label{def-frec-conv}
\end{align}
and
\begin{subequations} \label{eta-sets-conv}
\begin{align}
[\tfx{g}{\eta}(v)]_{\sx}&= \, \max_{b \in \tBg{\sx}{\eta} } \, \left( \sum_{\sy \in \X}
P_{\sx \sy}^{b}\, \valx_\sy \, + \, r_\sx^{b} \right) \enspace, \label{def-ftg-conv}
\\
\tBg{\sx}{\eta} &:= \underset{b \in \B_\sx}{\operatorname{argmax}} \; \left\{
 \sum_{\sy \in \X} P_{\sx \sy}^{b}\, \eta_\sy \right\} \enspace.
\end{align}
\end{subequations}

Let us fix an invariant half-line  $(\eta,v)$ of $g$.
Denote $\bar g(\valy):=\tfx{g}{\eta}(\valy)-\eta$,
then $v$ is harmonic with respect to $\bar g$: $\bar g(v)=v$.
We define the set of \new{critical nodes}
of $g$, $C(g)$, to be the set of critical nodes of $\bar g$.
A half-line $w: t\mapsto \tetav$ is \new{super-invariant}
 if $g\comp w(t)\leq w(t+1)$, for $t$
large enough. From~\eqref{Shapley-inv}, this property 
is equivalent to the conditions $\eta \geq \fr{g}(\eta)$ with 
 $v_\sx \geq \bar{g}_\sx(v)$ when $\eta_\sx =\fr{g}_\sx(\eta)$. 
In particular when the equality $\eta =\fr{g}(\eta)$ holds, 
it is equivalent to $v  \geq \bar{g}(v)$.

\begin{corollary}\label{cor-1}
Assume that $g$ is a polyhedral, convex, order preserving, 
and additively homogeneous self-map of $\RX$.
Assume that $w: t\mapsto \tetav$
is a super-invariant half-line of $g$ with $\eta=\chi(g)$. Then, there exists
a unique invariant half-line of $g$ which coincides with $w$
on the set of critical nodes of $g$. 
It is given by $t\mapsto t\eta+\Q{\bar g} (v)$, where
$\bar g:\valy\mapsto \tfx{g}{\eta}(\valy)-\eta$.
\end{corollary}
\begin{proof}
As said above, an invariant half-line of $g$ must be of the form 
$t\mapsto t\eta+\valz$, where $\eta=\chi(g)$ and $\valz\in \RX$ is
a fixed point of $\bar g$.
If  $w: t\mapsto \tetav$
is a super-invariant half-line of $g$ with $\eta=\chi(g)$, then
$\eta=\fr{g}(\eta)$, and by~\eqref{Shapley-inv}, we get
 $v  \geq \bar{g}(v)$.
From this, we deduce that $t\mapsto t\eta+\valz$ is an invariant half-line
of $g$ which coincides with $w$ on $C$, if and only if $\valz$ is 
harmonic with respect to $\bar g$ and coincides with $v$ on $C$.
By Theorem~\ref{th-1}, $\Q{\bar g}(v)$ is such a harmonic vector, and
it is the unique one. The corollary follows.
\end{proof}
For any super-invariant half-line $w$ of $g$ with $\eta=\chi(g)$, 
we define $\Q{g}(w)$ to be the half-line $t\mapsto t\eta+\Q{\bar g}(v)$.

\section{Policy iteration algorithm for stochastic mean payoff games}
\label{section-PIgames}

The following policy iteration scheme was introduced by
Cochet-Terrasson and Gaubert in~\cite{CochetGaub}. We first give,
in Algorithm~\ref{algo-main}, an
abstract formulation of the algorithm similar to the one given 
in~\cite{CochetGaub},
which is convenient to establish 
its convergence. A detailed practical
algorithm will follow and the proof of the convergence of the
algorithm will be given in the last subsection.

\subsection{The theoretical algorithm}

In order to present the algorithm, we assume that every
coordinate of $f:\RX\to\RX$ is given by:
\begin{align}
f_{\sx}(\valx) = \min_{a\in A_{\sx}} f^a_{\sx}(\valx) \enspace,\label{e-1}
\end{align}
where $A_{\sx}$ is a finite set, and $f^a_{\sx}$ is a polyhedral order
preserving, additively homogeneous, and convex
map from $\RX$ to $\R$.
These conditions all together are indeed equivalent to the property that $f$ is 
of the form~{\rm (\ref{ILPI},\ref{eq1opdef})},
since as already observed any polyhedral order preserving,
additively homogeneous and convex map $g$ from $\RX$ to $\R$
can be put in the form~\eqref{convpoly}, with $B_i$ a finite set.
For all feedback strategies $\balpha\in\Am=
\{\balpha: [n]\to \Ag,\; i \mapsto   \balpha(i)\in A_{\sx}\}$, we denote by 
  $f^{(\balpha)}$ the self-map  of $\RX$ the $i$-th coordinate of
  which is given by $f^{(\balpha)}_{\sx}=f^{\balpha(i)}_{\sx}$. 

\begin{algo}[Policy iteration for multichain mean payoff two player games~\cite{CochetGaub}]
\label{algo-main}
\

\emph{Input}: A map $f$ the coordinates of which are of the form~\eqref{e-1}.

\emph{Output}: 
An invariant half-line $w:t\mapsto \tetav$ of $f$ and an optimal policy
$\balpha\in\Am$.
\begin{enumerate}
\item\label{step-init} {\em Initialization}: Set $k=0$.
Select an arbitrary strategy $\sigmainit\in\Am$.
Compute an invariant half-line of $\finit$, 
$w^{(0)}: t\mapsto  t\eta^{(0)}+v^{(0)}$. 
\item\label{step-2} If $f \comp \wold(t)=\wold(t+1)$ holds
for $t$ large enough, the algorithm stops and returns
$\wold$ and $\sigmaold$.
\item\label{step-improve} Otherwise, {\em improve the strategy} $\sigmaold$
for $\wold$, by selecting a strategy $\sigmanew$ such that 
$f\comp \wold(t)=\fnew\comp \wold(t)$,
for $t$ large enough. The choice of $\sigmanew$ must be conservative,
meaning that, for all $\sx\in\X$, $\sigmanew(\sx)=\sigmaold(\sx)$ if 
$f_{\sx}\comp \wold(t)=\fold_{\sx}\comp \wold(t)$, 
for $t$ large enough.
\item\label{step-val} Compute
an arbitrary invariant half-line $\wdemi(t):t\mapsto t\etanew+\vdemi$ of $\fnew$.
If $\etanew \neq \etaold$ then set $\vnew = \vdemi$, i.e. $\wnew=\wdemi$,
and go to step~\ref{step-increment}.
Otherwise ($\etanew=\etaold$), we say that the iteration
is \new{degenerate}.
\item\label{step-degenerate}
Compute the invariant half-line $\wnew= \Q{(\fnew)}(\wold)$ of $\fnew$,
and define $\vnew$ and $\etanew$ by $\wnew(t)=t\etanew+\vnew$.
\item\label{step-increment} Increment $k$ by one and go to step~\ref{step-2}.
\end{enumerate}
\end{algo}

Let us give some details about the well posedness of this algorithm. 
First, the existence of the invariant half-lines in Steps~\ref{step-init} and
\ref{step-val} follows from Kohlberg theorem~\cite{kohlberg}
applied to the polyhedral order
preserving additively homogeneous maps $\fold$ with $k\geq 1$.
Second, due to the finiteness of the action sets $A_{\sx}$ and
the fact that the maps $ f^a_{\sx}$ are polyhedral, the maps $f$ and
 $ f^a_{\sx}$ can be rewritten in the form~\eqref{Shapley-inv}.
Hence, the test of Step~\ref{step-2} and the asymptotic optimization problem of
Step~\ref{step-improve} can be rewritten as an equality test for (germs of)
half-lines and the pointwise minimization of a finite set of half-lines,
which are transformed into systems of equations and lexicographical 
optimization problems, using the representation of half-lines as
couples $(\eta,v)$ instead of maps $w:t\mapsto t\eta+v$, see the
following section for details.

Finally, 
at each iteration $k$ of Algorithm~\ref{algo-main}, 
$\wold:t\mapsto t\etaold+\vold$ is a super-invariant half-line of $\fnew$.
Indeed, by construction of $\sigmanew$, and 
since $\wold$ is an invariant half-line of $\fold$,  we get
\begin{equation}\label{wksuper}
\fnew(\wold(t)) = f(\wold(t))\leq \fold (\wold(t))=\wold(t+1) \enspace,
\end{equation}
for $t$ large enough. 
Moreover,  since $\wold$ is an invariant half-line of $\fold$, we have 
$\chi(\fold)= \etaold$. 
Hence, in Step~\ref{step-degenerate}, $\wold$ is a
super-invariant half-line of $\fnew$ with slope $\etaold$ equal to
$\etanew = \chi(\fnew)$.
By Corollary~\ref{cor-1}, there exists a
unique invariant half-line of $\fnew$ which coincides with $\wold$ on
the set of critical nodes of $\fnew$ and it is given by $\wnew
= \Q{(\fnew)}(\wold): t\mapsto t \etanew+\vnew$
with $\vnew = \Q{\left(\bfnew\right)}(\vold)$.
Practical computations are detailed in the following sections.

\subsection{The practical algorithm}
\label{subsection-practical}

All the steps of Algorithm~\ref{algo-main} involve  equality tests
or pointwise minimizations of half-lines.
However, it would not be robust to do these tests on half-lines 
just by choosing an arbitrary large number $t$ in the equations
and inequations to be solved.
We shall rather use the equivalence between the representation of a half-line
as a map $w:t\mapsto t\eta+\valx$ with $t$ large and that 
as a couple $(\eta,v)$. This allows one to transform
all the tests into systems of equations or
optimizations of finite sets of half-lines for the pointwise lexicographic
order (which is linear, for each coordinate).
This means that we are solving the system of
equations~\eqref{DPsyst2P1}.
Then, using the notations of Section~\ref{discrete},
the corresponding practical algorithm of the formal  
Algorithm~\ref{algo-main} is given below in Algorithm~\ref{algo-detailed}.

\begin{algo}[Policy iteration for multichain mean payoff two player games] \label{algo-detailed}
\ 

\emph{Input}: A map $f$ the coordinates of which are of the form~\eqref{e-1}
and the notations~{\rm (\ref{eq1opdef},\ref{ILPI})} 
and~{\rm (\ref{eta-sets}--\ref{def-ftg-a})}.

\emph{Output}: An invariant half-line $(\eta,v)$ of $f$ and an optimal policy
$\balpha\in\Am$.

\begin{enumerate}
\item \label{step-init-detail}  {\em Initialization}:
Set $k=0$.
Select an arbitrary strategy $\sigmainit\in\Am$.
Compute the couple ($\eta^{(0)}$, $v^{(0)}$) solution of 
   \begin{equation} \label{DP-syst-step-init} \left\{ \begin{array}{r l}
	\eta_{\sx}^{(0)} & = \,	\frvia{F}{\eta^{(0)}}{\sx}{\sigmainit(\sx)}  \\
	\eta_{\sx}^{(0)} + \valx_{\sx}^{(0)} & = \,\tfxvia{F}{\eta^{(0)}}{v^{(0)}}{\sx}{\sigmainit(\sx)}
     \end{array} \right. \quad \text{for all } i \in [n]\enspace.
   \end{equation}
\item \label{step-stop-detail} If $\etaold$ and $\vold$ satisfy
  System~\eqref{DPsyst2P1}, or equivalently if $\sigmanew=\sigmaold$ 
is solution of~\eqref{sys-improve} below, then 
the algorithm stops and returns $(\etaold, \vold)$ and  $\sigmaold$.

\item \label{step-improve-detail} Otherwise, improve the policy $\sigmaold\in\Am$ for $(\etaold, \vold)$
  in a conservative way, that is choose $\sigmanew\in\Am$  such that
\begin{equation}\label{sys-improve}
\left\{ \begin{array}{l}
\displaystyle 
\sigmanew(i) \ \in \ \underset{ a \in \tA{\sx}{\etaold}}{\operatorname{argmin}} \left\{ \tfxvia{F}{\etaold}{\vold}{\sx}{a} \right\} \\
\sigmanew(i) = \sigmaold(i)\text{ if }\sigmaold(i)\text{ is optimal,}  
\end{array}\right. \quad \text{for all } i\in\X\enspace.
\end{equation}

\item \label{step-val-detail} Compute a couple ($\etanew$, $\vdemi$) for policy $\sigmanew$
   solution of 
   \begin{equation}\label{DP-syst-step-val} \left\{ \begin{array}{r l} 
	\etanew_{\sx} & = \, \frvia{F}{\etanew}{\sx}{\sigmanew(\sx)} \,    \\
       	\etanew_{\sx} + \vdemi_{\sx}  & =  \,  \tfxvia{F}{\etanew}{\vdemi}{\sx}{\sigmanew(\sx)} 
     \end{array} \right.  \quad \text{for all } i\in\X\enspace.
   \end{equation}
If $\etanew \ne \etaold$ then set $\vnew =
   \vdemi$ and go to
   step~\ref{step-increment-detail}. Otherwise, the iteration is
   degenerate. 
 \item \hspace{-.8em} i) \label{step-critical-graph} Let 
$g:=\fnew$ ($g_\sx= F(\cdot;\sx,\sigmanew(\sx))$).
Compute $C(g)$ the set of critical nodes of the
   map $\bar g$ defined by~: $\bar g= \tfx{g}{\etanew}(\cdot)- \etanew$, 
or equivalently:
  \[
  \bar g_\sx (v) \, = \, \tfxvia{F}{\etanew}{v}{\sx}{\sigmanew(\sx)} 
- \etanew_\sx \quad \text{for all } \sx\in\X\enspace, 
  \]
  for which $\vdemi$ is a harmonic vector.
\addtocounter{enumi}{-1}
\item \hspace{-.8em} ii) Compute $\vnew = \Q {\bar g} (\vold)$, that is the solution of: 
\begin{equation}\label{degenerate-syst} \left\{ \begin{array}{l c l l}
   \vnew_{\sx} &=& \tfxvia{F}{\etanew}{\vnew}{\sx}{\sigmanew(\sx)} - 
\etanew_\sx  & \sx \in \X\setminus C(g) \\
   \vnew_{\sx} &=& \vold_{\sx} & \sx \in C(g) \enspace.
\end{array} \right. 
\end{equation} 
 \item \label{step-increment-detail} Increment $k$ by one and go to
  Step~\ref{step-stop-detail}.
\end{enumerate}
\end{algo}

It remains to precise how the steps are performed.
Step~\ref{step-improve-detail} is just composed of lexicographic optimization
problems in finite sets.
The systems~\eqref{DP-syst-step-init} and~\eqref{DP-syst-step-val} are the
dynamic programming equations  of a one player multichain mean payoff
game, they can be computed by applying the policy iteration algorithm
for multichain Markov decision processes with mean payoff
introduced by Howard~\cite{Howard60} and Denardo and Fox~\cite{Denardo-Fox68}.
Note that one can also choose to solve Systems~\eqref{DP-syst-step-init}
and~\eqref{DP-syst-step-val} by applying Algorithm~\ref{algo-detailed}
to the maps $h=f^{(\sigma_0)}$ and $h=\fnew$ respectively,
while replacing minimizations by maximizations, but in that case
the algorithm is almost equivalent to that of Howard~\cite{Howard60} 
and Denardo and Fox~\cite{Denardo-Fox68}, see Section~\ref{sec-detailed} below.
In Step~\ref{step-critical-graph}, the set of critical nodes of $g$,
that is that of $\bar g$, can be computed using a variant of the algorithm
proposed in~\cite[\S~6.3]{spectral2} described in Section~\ref{degenerateCase}.
Finally, System~\eqref{degenerate-syst} is the dynamic programming equation
of an optimal control problem with infinite horizon stopped when
reaching the set $C(g)$ which can be solved using the
original policy iteration algorithm of Howard~\cite{Howard60}. 
We shall recall all these algorithms in Section~\ref{sec-detailed}.

\subsection{Convergence of the algorithm}
\label{convergence}

In this subsection, we show in Theorem~\ref{thm-algo-terminates} that
Algorithm~\ref{algo-main}, or equivalently Algorithm~\ref{algo-detailed}
 terminates after a finite number of steps.
This result is proved using Theorem~\ref{th-1}.
Let first show some intermediate results.

The following lemma is known, see for instance Sorin~\cite{sorindeds}.
\begin{lem}[{See~\cite{sorindeds}}]
\label{lem-sorin}
Let $g$ denote an order preserving self-map of $\RX$, that
is nonexpansive in the sup-norm, and has a cycle time $\chi(g)$.
If $w:t\mapsto \tetav$ is a super-invariant half-line
of $g$, then, $\chi(g)\leq \eta$.
\end{lem}
\begin{proof}
We reproduce the argument, for completeness:
if $w:t\mapsto \tetav$ is a super-invariant half-line
of $g$, that is $g(w(t))\leq w(t+1)$ for $t\geq t_0$ for some $t_0\geq 0$,
then, $g^k(w(t))\leq w(t+k)$,
for all $k\geq 0$, and $t\geq t_0$,
and so $\chi(g)\leq \lim_{k\to \infty} w(t_0+k)/k=\eta$,
which shows Lemma~\ref{lem-sorin}.
\end{proof}

Since, by~\eqref{wksuper},
$\wold$ is a super-invariant half-line of $\fnew$, with slope
$\eta^{(k)}=\chi(\fold)$, it follows from Lemma~\ref{lem-sorin} that~:
\begin{lem}\label{lem-dec}
The sequence of strategies defined in Algorithm~\ref{algo-main} is such that
\[
\chi(\fnew)\leq \chi(\fold) \enspace . \qedhere
\]
\end{lem}

We now examine degenerate iterations.
\begin{lem}\label{lem-cdec}
Let $(\sigmaold)_{k\geq 1}$ be the sequence of strategies defined
in Algorithm~\ref{algo-main},
and assume that $\chi(\fnew)= \chi(\fold)$.
Then, the following statements hold.
\begin{enumerate}
\item\label{lem-cdec2} The half-line $\wnew$ agrees with $\wold$ on the set
of critical nodes of $\fnew$.
\item \label{lem-cdec1}
Every critical node of $\fnew$ is a critical node of $\fold$.
\item\label{lem-cdec3} $\wnew\leq \wold$.
\end{enumerate}
\end{lem}

\begin{proof} Let us use the notations: $g:=\fnew$ (as in Algorithm~\ref{algo-detailed}) and $h=\fold$.
By construction and assumption, we have 
$\etaold= \chi(h)=\chi(g)=\etanew$, that we shall also
denote by $\eta$.

\noindent {\em Point~\ref{lem-cdec2}:}
Since, by~\eqref{wksuper}, $\wold$ is a super-invariant half-line of $g$, 
with slope $\etaold=\chi(g)$, and since $\wnew$ is defined as 
$\Q{g}(\wold)$,
the result follows from Corollary~\ref{cor-1}.

\noindent {\em Point~\ref{lem-cdec1}:}
Again, since $\wold$ is a super-invariant half-line of $g$, 
with slope $\etaold=\chi(g)$, we deduce 
 from the definition of $\bar g$ and~\eqref{Shapley-inv}, that
\begin{equation}\label{vksub}
\bar g(\vold) \leq \vold \enspace .
\end{equation}
Then by Lemma~\ref{lem-withmarianne}, $\bar g(\vold)$ agrees with $\vold$
on $C(\bar g)=C(g)$, the set of critical nodes of $g$, and so, the equality 
$g(\wold(t))=\wold(t+1)$ holds on $C(g)$ for $t$ large.
Since $\wold$ is an invariant half line of $\fold$,
we get that 
$f_i(\wold(t))=\fnew_i( \wold(t))=\wold(t+1)=\fold_i(\wold(t))$ for $t$ 
large enough and $i\in C(g)$. Hence, the conservative
selection rule ensures that $\sigmanew(i)=\sigmaold(i)$ for all $i\in C(g)$.
This implies that $g_i=h_i$ for all $i\in C(g)$, and since $\chi(g)=\chi(h)$,
we get from the definitions of $\bar g$ and $\bar h$ that
\begin{equation}\label{strategy-csq}
\bar g_i=\bar h_i \qquad \text{for all} \ i\in C(g) \enspace.
\end{equation}
Observe that $\vnew$ is a fixed-point
of $\bar{g}$, and that $\bar g$ is a polyhedral additively homogeneous
order preserving convex selfmap of $\RX$.
Hence the critical nodes of $\bar g$ are the indices that belong to
a final class of a matrix $M\in \partial \bar g(\vnew)$
(since the elements of $\bar g(\vnew)$ are stochastic matrices, 
all their final classes are recurrent).
Let $F$ be  such a final class.
From~\eqref{subdiffequiv}, 
the line $M_{i\cdot} \in \partial \bar{g}_i(\vnew)$ for $i\in F$,
that is $\bar g_i(v)-\bar g_i(\vnew)\geq M_{i\cdot} (v-\vnew)$ for all
$v\in\RX$.
Since $\vnew$ is a fixed point of $\bar g$, $\vold$ a fixed point of
$\bar h$, and $\vnew$ agrees with $\vold$ on $C(g)$ 
(from Point~\ref{lem-cdec2}),
we get that $\bar g_i(\vnew)=\vnew_i=\vold_i=\bar h_i(\vold)$ for all
$i\in C(g)$. From~\eqref{strategy-csq}, we deduce 
that $\bar g_i(v)-\bar g_i(\vnew)=\bar h_i(v)-\bar h_i(\vold)$
for all $i\in C(g)$ and $v\in\RX$.
Now, since $F$ is a final class
of $M$, hence $F\subset C(g)$, and
$M_{ij}=0$ for $i\in F$ and $j\not\in C(g)$,
we get that $M_{i\cdot} \vnew=M_{i\cdot} \vold$ for $i\in F$.
This implies that  $\bar h_i(v)-\bar h_i(\vold)\geq M_{i\cdot} (v-\vold)$
for all $v\in\RX$ and $i\in F$,
which shows that $M_{i\cdot} \in \partial \bar{h}_i(\vold)$ for $i\in F$.
Let $N:=[n]\setminus F$ and define the matrix
$Q$ such that $Q_{\sx\cdot}=M_{\sx\cdot}$ if $\sx\in F$, and $Q_{\sx\cdot}$
be any element of $\partial \bar{h}_{\sx}(\vold)$ if $\sx\in N$, 
then $Q\in \partial{\bar{h}(\vold)}$.
Hence, the $F\times F$ submatrix of $M$
is also a $F\times F$ submatrix of $Q$, and so
$F$ is a final class of $Q$. Since $\vold$ is a fixed
point of $\bar h$, this implies that $F$ is included in the set of 
critical nodes of $\bar h$, which is also by definition 
the set of critical nodes of $h$.
This shows that all critical nodes of $g$ are also
critical nodes of $h$, and shows Point~\ref{lem-cdec1}.

\noindent {\em Point~\ref{lem-cdec3}:} 
From~\eqref{vksub}, we get that $\bar g(\vold) \leq \vold$, hence 
the sequence $\bar g^k(\vold)$ is nonincreasing and 
$\Q{\bar g}(\vold) \leq \vold$.
Since $\etaold=\etanew$, we get that $\wnew=\Q{g}(\wold) =
t\etaold + \bar g(\vold)\leq \wold$.
\end{proof}

Finally, we prove that the algorithm terminates. 

\begin{thm}\label{thm-algo-terminates}
A strategy cannot be selected twice in Algorithm~\ref{algo-main}, and so,
the algorithm terminates after a finite number of iterations.
\end{thm}
\begin{proof}
Assume by contradiction that the same strategy is selected twice
in Algorithm~\ref{algo-main}, that is $\sigmaiter s=\sigmaiter m$ for some 
iterations $1\leq s<m$ of the algorithm before it stops.
Then, $\chi(f^{(\sigmaiter s)})=\chi(f^{(\sigmaiter m)})$ and since
by Lemma~\ref{lem-dec}, $\chi(f^{(\sigmaiter s)}) \geq \chi(f^{(\sigmaiter {s+1})})\geq
 \cdots\geq \chi(f^{(\sigmaiter m)})$, we get the equality
 $\chi(f^{(\sigmaiter s)})= \chi(f^{(\sigmaiter {s+1})})=  \cdots=\chi(f^{(\sigmaiter m)})$.
Hence, by Lemma~\ref{lem-cdec}, Part~\ref{lem-cdec1}, 
we have that $ C(f^{(\sigmaiter m)})\subset C(f^{(\sigmaiter {m-1})})\subset\cdots
\subset C(f^{(\sigmaiter s)})$ and since $\sigmaiter s=\sigmaiter m$,
we get the equality $ C(f^{(\sigmaiter m)})=C(f^{(\sigmaiter {m-1})})=\cdots=
C(f^{(\sigmaiter s)})$.
So by Lemma~\ref{lem-cdec}, Part~\ref{lem-cdec2},  
$w^{(s)}$ and $w^{(m)}$ are both invariant half-lines of $f^{(\sigmaiter s)}$
with slope $\chi(f^{(\sigmaiter s)})$, 
that agree on $C(f^{(\sigmaiter s)})$. Hence by Corollary~\ref{cor-1},
 $w^{(s)}=w^{(m)}$.
Since by Lemma~\ref{lem-cdec}, Part~\ref{lem-cdec3},
we have $w^{(s)}\geq w^{(s+1)}\geq \cdots \geq w^{(m)}$,
it follows that $w^{(s)}=\cdots =w^{(m)}$. In particular, $w^{(s)}=w^{(s+1)}$.
Hence, $w^{(s)}(t+1)=w^{(s+1)}(t+1)=f^{(\sigmaiter {s+1})}\comp w^{({s+1})}(t)=
f^{(\sigmaiter {s+1})}\comp w^{(s)}(t)=f\comp w^{(s)}(t)$ for $t$ large enough.
It follows that $w^{(s)}$
is an invariant half-line of $f$, and so, the algorithm stops at step $s$,
which contradicts the existence of iteration $m$, and so the same
strategy cannot be selected twice in Algorithm~\ref{algo-main}.

Since the sets $A_i$ are finite, the number of strategies 
(the elements of $\Am$) is also finite, and since a strategy cannot be 
selected twice, Algorithm~\ref{algo-main} stops after a finite number 
of iterations, that is bounded by the number of strategies.
\end{proof}

\section{Ingredients of Algorithm~\ref{algo-main} or~\ref{algo-detailed}:
one player games algorithms}
\label{sec-detailed}

As said in Section~\ref{subsection-practical}, each basic step of the policy
iteration algorithm for multichain mean payoff zero-sum two player games
(Algorithm~\ref{algo-main} or~\ref{algo-detailed})
concerns the solution of one player games, also called
stochastic control problems or Markov decision processes,
with finite state and action spaces: 
a mean payoff problem for Systems~\eqref{DP-syst-step-init}
and~\eqref{DP-syst-step-val}, an infinite horizon problem stopped at
the boundary for System~\eqref{degenerate-syst}, and the set of
critical nodes of the corresponding dynamic programming operator in
Step~\ref{step-critical-graph}. 
We recall here the policy iteration algorithm for solving stochastic control
problems, with either infinite horizon or mean payoff, and the
algorithm proposed in~\cite[\S~6.3]{spectral2} for computing a
critical graph, and explain how all these algorithms are applied in
Algorithm~\ref{algo-main} or~\ref{algo-detailed}. By doing so,
we shall also see that the classical Howard / Denardo-Fox algorithm
can be thought of as a special case of these algorithms, 
in which the second player has no choices of actions.

In all the section, we consider the following dynamic
programming or Shapley operator of
a one player game with finite state and action spaces:
$g$ is a map from $\RX$ to itself, given by~:
\begin{equation} \label{eq1opdefG}
[g(v)]_{\sx}:= \, \max_{b \in \B_\sx } \, \fvia{G}{v}{\sx}{b}
 \qquad  \forall \sx \in \X,\; v\in \RX \enspace, 
\end{equation}
where
\begin{equation} \label{ILPIG}
\fvia{G}{v}{\sx}{b} \, = \,  \sum_{\sy \in \X} P_{\sx \sy}^{b}\,
\valx_\sy \, + \, r_\sx^{b}  \enspace,
\end{equation}
the vectors $P_{\sx \cdot}^b$ are substochastic vectors,
for all
$\sx\in\X$ and $b\in \B_\sx$, and $\B_{\sx}$ are finite sets,
for all $\sx \in \X$. 
Equivalently, $g$ is a convex additively subhomogeneous order preserving 
polyhedral selfmap of $\RX$.

Since player \MIN\ does not exist, 
the set of feedback strategies for player \MAX, $\Bm$, is
given by  $\Bm := \set{\bbeta : \X \rightarrow \Bg}{\bbeta(\sx) \in \B_{\sx}
  \, \forall \sx \in \X}$, where $\Bg$ contains all the sets $\B_{\sx}$.
For each $\bbeta\in\Bm$, we denote by
$g^{(\bbeta)}$ the self-map  of $\RX$ given by:
\[ g^{(\bbeta)}_{\sx}(v):= \fvia{G}{v}{\sx}{\bbeta(\sx)}
 \qquad  \forall \sx \in \X,\; v\in \RX \enspace. \]
We also denote by $r^{(\bbeta)}$ the vector of $\RX$  such that
$r^{(\bbeta)}_\sx= r_\sx^{\bbeta(\sx)} $ and $P^{(\bbeta)}$ the $n\times n$ 
matrix such that $P^{(\bbeta)}_{\sx\sy}=P^{(\bbeta(\sx))}_{\sx\sy}$,
then $g^{(\bbeta)}:v\mapsto P^{(\bbeta)} v+r^{(\bbeta)}$.

\subsection{Policy iterations for one player games with discounted payoff} 
\label{howardalg}

System~\eqref{degenerate-syst} consists in 
finding the solution $v$ of the equation
$v=\bar g(v)$ with $v=u$ on $C(\bar g)$ 
where $u\in\RX$ is super-harmonic with respect to $\bar g$, $\bar g(u)\leq u$,
and $g$ is as in~\eqref{eq1opdefG} with~\eqref{ILPIG}.
The solution $v$ is thus the value of a one player game 
with infinite horizon stopped
when reaching the set $C(\bar g)$ whose transition probabilities
are given by the $P_{\sx \sy}^{b}$, instantaneous reward is given by the
$r_\sx^{b}$ and final reward is given by $u_\sx$, when the game is in state $\sx\in C(\bar g)$.
This value function can be obtained using the
classical policy iteration algorithm of Howard~\cite{Howard60} for a
one player game. 
From Theorem~\ref{th-1}, $v$ is solution of the above equation, if and
only if $v_C=u_C$ and $v_N$ is a fixed point of the
convex polyhedral additively subhomogeneous order preserving selfmap $h$ of
$\R^N$, with $C=C(\bar g)$, $N=\X\setminus C$, and 
$h$ defined as in Theorem~\ref{th-1}, Point (iii),
 with $g$ replaced by $\bar g$.
One can also consider the equivalent equation $v=h(v)$ with
$h_i=\bar g_i$ for $i\in N$ and $h_i(v)=u_i$ for $i\in C$ and $v\in\RX$.
In that case, $h$ is a convex polyhedral additively subhomogeneous order 
preserving selfmap of $\RX$.

In these two settings, we need to solve an equation of the 
form $v=g(v)$, where $g$ is of the form~\eqref{eq1opdefG}, 
and $g$ has no critical node: $C(g)=\emptyset$.
From~\cite[Corollary~1.3]{spectral2}, $g$ has
a unique fixed point and all the maps $g^{(\bbeta)}$ 
with $\bbeta\in \Bm$ have a unique fixed point (since their 
critical nodes are necessarily critical nodes of $g$).
The policy iteration algorithm of Howard applied to this equation
is then given by Algorithm~\ref{algo-Howard}. 

\begin{algo}[Policy iteration of Howard~\cite{Howard60} for stochastic control problems]
\label{algo-Howard}
\

\emph{Input}: A map $g$ of the form~\eqref{eq1opdefG} with no critical node.

\emph{Output}: The fixed point of $g$ and an optimal policy 
$\bbeta\in\Bm$.

\begin{enumerate}
\item  {\em Initialization}: Set $k=0$.
Select an arbitrary strategy $\deltainit\in\Bm$.
\item\label{howard-step-2}
 Compute the value of the game $\vold$ with fixed feedback
strategy $\deltaold$, that is the solution of the linear system:
\[
\vold \ = \ g^{(\deltaold)}(\vold) \enspace.
\]
\item\label{stopHoward} If $\vold=g(\vold)$, or equivalently if
 $\deltanew=\deltaold$ is solution of~\eqref{improvehoward} below,
then  the algorithm stops and returns $\vold$ and $\deltaold$.
\item\label{step-improvehoward} 
Otherwise, improve the policy $\deltanew\in\Bm$ for the value $\vold$~: 
\begin{equation}\label{improvehoward}
\deltanew(\sx) \ \in \ \underset{b \in
  \B_{\sx}}{\operatorname{argmax}} \ \fvia{G}{\vold}{\sx}{b} \
  \ \ \forall \sx \in [n]. 
\end{equation}
 \item Increment $k$ by one and go to
  Step~\ref{howard-step-2}.
\end{enumerate}
\end{algo}

It is known~\cite{Howard60} that $\vnew\leq \vold$ and that the algorithm stops 
after a finite number of steps.

\subsection{Policy iteration for multichain one player games}
\label{sec-detailed-sub1}

Consider a one player game with dynamic programming operator $g$ 
given by~\eqref{eq1opdefG} and mean payoff. 
Then, as explained in Section~\ref{discrete} in the more general 
two player case, the mean payoff of the game
is the slope $\eta$ of any invariant half line $(\eta,\valx)$ of $g$,
which is also any solution of the following couple system
(see Equation~\eqref{DPsyst2P1}):
\begin{equation}\label{DPsyst1P} \left\{ \begin{array}{rcl}
  \eta &\,=\,& \fr{g}(\eta)
\\
\eta + \valx  &=&  \tfx{g}{\eta}(\valx) \enspace.
\end{array} \right. 
\end{equation}
where $\fr{g}$ and $\tfx{g}{\eta}$ are defined in~\eqref{def-frec}
and~\eqref{eta-sets} respectively.
In the present one player case, they are reduced to:
\begin{equation}\label{defghat}
[\fr{g}(\eta)]_{\sx} \, := \, \max_{b \in \B_\sx} \, \frvia{G}{\eta}{\sx}{b}
\qquad \text{and} \qquad
[\tfx{g}{\eta}(\valx)]_{\sx} \, := \, \max_{b \in \tBg{\sx}{\eta}} \,
\fvia{G}{v}{\sx}{b} \enspace,
\end{equation}
with
\begin{equation}\label{defGhat}
\frvia{G}{\eta}{\sx}{b} =  \sum_{\sy \in \X} P_{\sx \sy}^{b}\, \eta_\sy 
\qquad \text{and} \qquad
\tBg{\sx}{\eta} := \underset{b \in \B_\sx}{\operatorname{argmax}} \; \left\{
 \sum_{\sy \in \X} P_{\sx \sy}^{b}\, \eta_\sy \right\} \enspace,
\end{equation}
for all $\eta,\valx\in\RX$, $\sx \in [n]$, $b\in \Bg$.
We refer also to~\cite{Denardo-Fox68,puterman} for the existence of solutions 
to System~\eqref{DPsyst1P},
and for the proof that $\eta$ solution of this system is
the mean payoff of the game in this one player context.
The following algorithm for multichain mean payoff Markov decision processes
was introduced by Howard~\cite{Howard60} and proved to converge by
Denardo and Fox~\cite{Denardo-Fox68}:

\begin{algo}[Policy iteration algorithm for multichain mean payoff one player games]
\label{algo-DF}
\

\emph{Input}: A map $g$ of the form~\eqref{eq1opdefG} with~\eqref{ILPIG}, 
and the notations~{\rm (\ref{defghat},\ref{defGhat})}.

\emph{Output}: An invariant half-line $(\eta,v)$ of $g$ and an optimal policy
$\bbeta\in\Bm$.

\begin{enumerate}
\item  {\em Initialization}: Set $k=0$.
Select an arbitrary strategy $\deltainit\in\Bm$.
\item \label{step-val-DF} For each final class $F$ of
  $P^{(\deltaold)}$, denote by $i_F$ the minimal index of the elements
  of $F$, and define $S$ as the set of all these indices $i_F$. Compute
  the couple  $(\etaold,\vold)$ for policy $\deltaold$ 
  solution  of 
\begin{equation} \label{DF-LS}
\left\{ \begin{array}{r  c  l l}
  \etaold_{\sx} &=& \frvia{G}{\etaold}{\sx}{\deltaold(\sx)} & \sx \in [n]  \setminus S\\
  \etaold_{\sx} + \vold_{\sx} &=& \fvia{G}{\vold}{\sx}{\deltaold(\sx)}  & \sx \in [n] \\
  \vold_{\sx} &=& 0 & \sx \in S \enspace. \\
\end{array}
\right.
\end{equation}
\item \label{stop-DF}
If $(\etaold,\vold)$ is solution of~\eqref{DPsyst1P}, or equivalently if
 $\deltanew=\deltaold$ is solution of~\eqref{improveonemulti} below,
then the algorithm stops and returns $(\etaold,\vold)$ and $\deltaold$.
\item \label{step-improve-DF}  Otherwise, improve the policy $\deltanew\in\Bm$ for ($\etaold$, $\vold$) in a conservative way,
 that is choose $\deltanew\in\Bm$  such that~: 
\begin{equation}\label{improveonemulti}
\left\{ \begin{array}{l}
\displaystyle 
\deltanew(\sx) \ \in \ \underset{ b \in 
  \tBg{\sx}{\etaold}}{\operatorname{argmax}} \, \fvia{G}{\vold}{\sx}{b} \\
\deltanew(i) = \deltaold(i)\text{ if }\deltaold(i)\text{ is optimal,}  
\end{array}\right. \quad \text{for all } i\in\X\enspace.
\end{equation}
 \item Increment $k$ by one and go to Step~\ref{step-val-DF}.
\end{enumerate}
\end{algo}

The justifications and details of Algorithm~\ref{algo-DF} can be found
in~\cite{Denardo-Fox68,puterman} and are recalled in Appendix.
Solving System~\eqref{DF-LS} turns out to be a critical step.
This can be optimized by exploiting the structure of the system,
we discuss this issue in Appendix.
As explained in Section~\ref{subsection-practical}, another way to solve 
a multichain mean payoff Markov decision process may be to use 
Algorithm~\ref{algo-main} or~\ref{algo-detailed} in the particular
case of a one-player game, with maximizations instead of minimizations.
In order to compare it with Algorithm~\ref{algo-DF}, we 
rewrite below Algorithm~\ref{algo-detailed}
in that case, with the above notations.
Note that in the one-player case, 
the map $g$ of Step~\ref{step-critical-graph} of Algorithm~\ref{algo-detailed}
 is affine,  hence
its critical graph reduces to the final graph of its tangent matrix.

\begin{algo}[Specialization of Algorithm~\ref{algo-detailed} to the one player case]
\label{algo-detailed-oneplayer}
\ 

\emph{Input}: A map $g$ of the form~\eqref{eq1opdefG} with~\eqref{ILPIG}, 
and the notations~{\rm (\ref{defghat},\ref{defGhat})}.

\emph{Output}: An invariant half-line $(\eta,v)$ of $g$ and an optimal policy
$\bbeta\in\Bm$.

\begin{enumerate}
\item \label{step-init-detail-oneplayer}  {\em Initialization}:
Set $k=0$. Select an arbitrary strategy $\deltainit\in\Bm$.
Compute the couple ($\eta^{(0)}$, $v^{(0)}$) solution of 
   \begin{equation} \label{DP-syst-step-init-oneplayer} \left\{ \begin{array}{r l}
	\eta_{\sx}^{(0)} & = \,	\frvia{G}{\eta^{(0)}}{\sx}{\deltainit(\sx)}  \\
	\eta_{\sx}^{(0)} + \valx_{\sx}^{(0)} & = \,\fvia{G}{\valx^{(0)}}{\sx}{\deltainit(\sx)}
     \end{array} \right. \quad \text{for all } i \in [n]\enspace.
   \end{equation}
\item \label{step-stop-detail-oneplayer} If $\etaold$ and $\vold$ satisfy
  System~\eqref{DPsyst1P}, or equivalently if $\deltanew=\deltaold$ 
is solution of~\eqref{sys-improve-oneplayer} below, then 
the algorithm stops and returns $(\etaold, \vold)$ and  $\deltaold$.

\item \label{step-improve-detail-oneplayer} Otherwise, improve the policy $\deltaold\in\Bm$ for $(\etaold, \vold)$
  in a conservative way, that is choose $\deltanew\in\Bm$ such that
\begin{equation}\label{sys-improve-oneplayer}
\left\{ \begin{array}{l}
\displaystyle 
\deltanew(\sx) \ \in \ \underset{ b \in 
  \tBg{\sx}{\etaold}}{\operatorname{argmax}} \, \fvia{G}{\vold}{\sx}{b} \\
\deltanew(i) = \deltaold(i)\text{ if }\deltaold(i)\text{ is optimal,}  
\end{array}\right. \quad \text{for all } i\in\X\enspace.
\end{equation}

\item \label{step-val-detail-oneplayer} Compute a couple ($\etanew$, $\vdemi$) for policy $\deltanew$ solution of 
   \begin{equation}\label{DP-syst-step-val-oneplayer} \left\{ \begin{array}{r l} 
	\etanew_{\sx} & = \, \frvia{G}{\etanew}{\sx}{\deltanew(\sx)} \,    \\
       	\etanew_{\sx} + \vdemi_{\sx}  & =  \,  \fvia{G}{\vdemi}{\sx}{\deltanew(\sx)} 
     \end{array} \right.  \quad \text{for all } i\in\X\enspace.
   \end{equation}
If $\etanew \ne \etaold$ then set $\vnew =
   \vdemi$ and go to
   step~\ref{step-increment-detail-oneplayer}. Otherwise, the iteration is
   degenerate. 
 \item \hspace{-.8em} i) \label{step-critical-graph-oneplayer} 
Compute $C$ the set of final nodes of the matrix $P^{(\delta_{k+1})}$.
\addtocounter{enumi}{-1}
\item \hspace{-.8em} ii) Compute the solution $\vnew$ of: 
\begin{equation}\label{degenerate-syst-oneplayer} \left\{ \begin{array}{l c l l}
   \vnew_{\sx} &=& \fvia{G}{\vnew}{\sx}{\deltanew(\sx)} - 
\etanew_\sx  & \sx \in \X\setminus C \\
   \vnew_{\sx} &=& \vold_{\sx} & \sx \in C \enspace.
\end{array} \right. 
\end{equation} 
 \item \label{step-increment-detail-oneplayer} Increment $k$ by one and go to
  Step~\ref{step-stop-detail-oneplayer}.
\end{enumerate}
\end{algo}

Systems~\eqref{DP-syst-step-init-oneplayer} and~\eqref{DP-syst-step-val-oneplayer}  are of the form:
\begin{equation} \label{LSDF}
\left\{ \begin{array}{r  c l}
  \eta &=& P \, \eta \\
  \eta + \valx &=& P \, \valx + r\enspace,
\end{array}
\right.
\end{equation}
where $r=r^{(\delta)}\in \RX$ and $P=P^{(\delta)}$ is a stochastic matrix, with
$\delta=\deltainit$ or $\deltanew$.
It can be shown that the solution $\eta$ of such a system is unique, 
that one can eliminate for each final class $F$ of $P$ one of the equations
$\eta_i=(P \eta)_i$  with index $i\in F$,  and
that $\valx$ is defined up to an element of the kernel of $I - P$,
the dimension of which is equal to the number of final classes of $P$.
When this number is strictly greater than one, and 
$\vnew$ is chosen to be any solution $\vdemi$ of
\eqref{DP-syst-step-val-oneplayer} in Algorithm~\ref{algo-detailed-oneplayer},
the algorithm may cycles, see Section~\ref{exemple} for an example
in the two player case.
One way to handle this~\cite{Denardo-Fox68,puterman}, is
either to fix to zero the value of $\mu_F \valx$ for each 
invariant measure $\mu_F$ of $P$ with support in a final class $F$ of $P$,
or to fix to zero the components of $\valx$ with indices in some set $S$
containing exactly one node of each final class of $P$.
In these two cases, the solution $\valx$ of~\eqref{LSDF} become unique.
Moreover, if in Algorithm~\ref{algo-detailed-oneplayer},
\eqref{DP-syst-step-val-oneplayer} is combined 
with either the conditions  $\mu_F \vdemi=0$
or the conditions $\vdemi_S=0$
with $S$ chosen in a conservative way, that is such that the same
index is chosen in $F$ for iterations $k$ and $k+1$,
if $F$ is a final class of $P^{(\deltanew)}$ 
which is also a final class of $P^{(\deltaold)}$,
then $\vdemi=\vold$ on the set of final nodes of $P^{(\deltanew)}$
when $\etanew=\etaold$, which implies that $\vdemi=\vnew$, hence
Step~\ref{step-critical-graph-oneplayer} of
Algorithm~\ref{algo-detailed-oneplayer} becomes useless.
This shows that Algorithm~\ref{algo-DF} is equivalent to 
Algorithm~\ref{algo-detailed-oneplayer}, where
\eqref{DP-syst-step-val-oneplayer} is combined 
with the conditions $\vdemi_S=0$, where $S$ is the set of
minimal indices of each final class of $P^{(\deltanew)}$.
In other words, Algorithm~\ref{algo-DF} is a particular realization of
Algorithm~\ref{algo-detailed-oneplayer}, where one chooses one special 
solution $\vdemi=\vnew$ of~\eqref{DP-syst-step-val-oneplayer} at each 
iteration of the algorithm, even when $\etanew\neq \etaold$.
Denardo and Fox proved~\cite{Denardo-Fox68,puterman} 
that the sequence of couples
$(\etaold,\vold)_{k\geq 1}$ of Algorithm~\ref{algo-DF} is non 
decreasing in a lexicographical order, meaning that $\etanew \geq
\etaold$, with $\vnew \geq \vold$ when $\etanew = \etaold$, and that
Algorithm~\ref{algo-DF} stops after a finite number of iterations
(when the sets of actions are finite).
Indeed, the convergence of Algorithm~\ref{algo-main},
proved in Section~\ref{convergence}, shows that this also holds 
for the little more general Algorithm~\ref{algo-detailed-oneplayer}.

\subsection{Critical graph}
\label{degenerateCase}

When a degenerate iteration ($\etanew=\etaold$) occurs in
Step~\ref{step-val} of Algorithm~\ref{algo-main}, one has to compute
the critical nodes of $g := \fnew$, that is that of $\bg$.
This can be done by applying the techniques of~\cite[\S~6.3]{spectral2} ,
leading to Algorithm~\ref{algo-critique} below.
More precisely, one applies first the followings steps to the map $\bg$ and
its harmonic vector $\vdemi$, then apply Algorithm~\ref{algo-critique}.

Consider an additively homogeneous map $g$ whose coordinates are defined as in
\eqref{eq1opdefG}  with~\eqref{ILPIG}, and $u$ a harmonic vector of $g$.  
For any set $\sP$ of stochastic matrices, we define $\gf(\sP)$ as the
union of the graphs of the matrices $M_{FF}$, where $M\in \sP$ and $F$
is a final class of $M$. 
Define 
\begin{equation}\label{applycritical}
\sB_{\sx} = \set{b \in
  \B_{\sx}}{ \fvia{G}{u}{\sx}{b} = u}\quad \text{and}\quad
\sP_{\sx}= \set{P_{\sx\cdot}^{b}}{b \in \sB_{\sx}}\enspace.\end{equation}
Then, the critical graph of $g$ is given by 
\begin{equation}\label{critical-final}
 \gc (g) = \gf (\partial g(u)), \quad\text{where}\quad
\partial g(u) = \vex(\sP_1) \times \cdots  \times \vex(\sP_n)\enspace,
\end{equation}
and $\vex(\cdot)$ denotes the convex hull of a set.
The following algorithm computes the graph in~\eqref{critical-final} for a
general family $\{\sP_{\sx}\}_{\sx\in \X}$,
where $\sP_{\sx}\subset \RX$ is a nonempty finite set of stochastic vectors.
Note that any such family $\{\sP_{\sx}\}_{\sx\in \X}$ corresponds
to the map $g:\RX\to\RX$ such that 
\begin{equation}\label{Ptog}
 [g(v)]_i= \max_{p\in \sP_i} p v\quad \text{for all}\; i\in \X\enspace,
\end{equation}
which has $u=0$ as a harmonic vector, and is of the above form.
Hence the algorithm below corresponds also to the 
computation of the critical graph of this map $g$.

Before writing the algorithm, we recall some definitions of graph theory
(see for instance~\cite{cormen}). We define a graph
$G:=(V,E)$ as a finite set of vertices (or nodes) $V$
and a set of edges (or arcs) $E := \set{(i,j)}{i,j \in V}$. A 
path of length $l\geq 0$ is a sequence $(i_0, \dots, i_l)$ such that $i_k
\in V$ for $k\in \{0,\dots,l\}$ and $(i_{k}, i_{k+1}) \in E$ for $k <l$. 
A strongly connected component of G is the restriction $G|_{V'}$ 
of $G$ to some subset of nodes $V'\subseteq V$, 
that is the graph $(V',E')$ with
$E':= \set{(i,j)\in E}{i,j \in V'}$, where $V'$ is such that there
exists a path from each node $i \in V'$ to every node $j\in V'$. A
strongly connected component $G'$ is called trivial if it consists in 
exactly one node and no arcs. We define a final class 
of $G=(V,E)$ as a non trivial strongly connected component $G'=(V',E')$ of
$G$ such that there exists no arc $(i,j) \in E$ with $i\in V'$
and $j\in V \setminus V'$. 
Note that the strongly connected components of a graph can be find
using Tarjan algorithm, see~\cite{cormen}.  

\begin{algo}[Algorithm to compute the critical graph, compare with~\protect{\cite[\S~6.3]{spectral2}}]
\label{algo-critique}
\

\emph{Input}: $(\sP_1, \cdots, \sP_n)$ where $\sP_{\sx}\subset \RX$
is a finite set of stochastic vectors for $i \in \X$.

\emph{Output}: A graph depending on $\sP_1, \cdots, \sP_n$,  equal
to $\gf(\vex(\sP_1) \times \cdots  \times\vex(\sP_n))$ if all the $\sP_i$ are nonempty; and its set of nodes.

\begin{enumerate}
\item Set $F^{(0)} = \emptyset$, $I^{(0)} = [n]$, 
$\grf^{(0)} =  \emptyset$, $\sQ^{(0)}_{\sx} =
  \sP_{\sx}$ for $\sx\in \X$, and $k=0$.
\item \label{algo-graph-critique-stop} 
If all the sets $\{\sQ^{(k)}_{\sx}\}_{\sx\in I^{(k)}}$
  are empty, then the algorithm stops and returns $\grf^{(k)}$ and  $F^{(k)}$.
\item \label{algo-graph-critique-step2}  Otherwise, build the graph
  $G = (I^{(k)}, E)$ with set of nodes $I^{(k)}$, and set of arcs 
$E = \set{(i,j)\in I^{(k)}\times  I^{(k)}}{p_j \ne 0\;\text{for some } 
p \in \sQ^{(k)}_i}$. Set $F$ as the union of final classes of~$G$. 
\item Put $I^{(k+1)}=I^{(k)}\setminus F$ and $F^{(k+1)} = F^{(k)} \cup F$.
\item Set $\grf^{(k+1)} = \grf^{(k)} \cup G|_{F}$ where
  $G|_{F}$ denotes the restriction of $G$ to ${F}$. 
\item For all $\sx\in I^{(k+1)}$, define the sets
  $\sQ^{(k+1)}_{\sx}\subset \R^{I^{(k+1)}}$ of row vectors obtained by
  restricting to $I^{(k+1)}$ the vectors $p \in \sQ^{(k)}_{\sx}$ such
  that $\sum_{j\in I^{(k+1)}} p_j = 1$. 
\item Increment $k$ by one, and go to Step~\ref{algo-graph-critique-stop}.
\end{enumerate}
\end{algo} 

The convergence (after at most $n$ iterations) of this algorithm 
follows from variants of Lemmas~4.7 and 4.9 of~\cite{spectral2},
applied to the maps $g_k$ constructed by~\eqref{Ptog} from the
families $(\sQ^{(k)}_{\sx})_{\sx\in\X}$.
Indeed, if all the $\sQ^{(k)}_{\sx}$ with  $\sx\in I^{(k)}$
are nonempty, the map $g_{k}$ 
is a map from $\RX$ to itself and 
Lemma~4.7 says that $g_{k}$ has at least one invariant
critical class, which implies that the set $F$ of
 Step~\ref{algo-graph-critique-step2} is nonempty.
Moreover, Lemma~4.9 says that, if all the $\sQ^{(k+1)}_{\sx}$ 
with  $\sx\in I^{(k+1)}$ are nonempty,
the critical graph of $g_{k}$ is equal to
the union of $G|_{F}$ with the critical graph of the map $g_{k+1}$.

In order to generalize these arguments, one need to extend the notion 
of critical graph to the case of a map $g$ from $(\R\cup\{-\infty\})^n$
to itself, of the form~\eqref{Ptog} with general
families $\{\sP_{\sx}\}_{\sx\in \X}$ of (possibly empty) finite sets 
of stochastic vectors (or of the form  \eqref{eq1opdefG}  
with~\eqref{ILPIG}, with a harmonic vector $u\in (\R\cup\{-\infty\})^n$).
For instance, define the critical graph of $g$ as 
the restriction to the set
 of nodes $i\in\X$ such that $\sP_{\sx}$ is nonempty (or $u_\sx\neq -\infty$)
of the critical graph of $g\vee \mathrm{id}$, where $\mathrm{id}$ is 
the identity map and $\vee$ denotes the supremum operation.
Then, the identically $-\infty$ map has no critical class,
any map $g$ which is not identically $-\infty$ has an invariant 
critical class, and the above recurrence formula for critical graphs is true
even if $g_{k+1}$ takes $-\infty$ values.
This shows that Algorithm~\ref{algo-critique} computes the
critical graph of the map $g$ associated to the family
$\{\sP_{\sx}\}_{\sx\in \X}$, even if some of the sets of the family
are empty.

Note that since Tarjan algorithm has a linear complexity in the number
of arcs of a graph, the complexity of the above algorithm 
is at most in the order of $nm$, where $m$ is the sum of 
the number of arcs of all the elements of $\sP_i$, $i\in\X$.
This is comparable with the complexity of solving the linear systems 
of the form~\eqref{DF-LS} by LU solvers, hence with the
other steps of Algorithm~\ref{algo-detailed}.

\section{An example with degenerate iterations}\label{exemple}

In this section, we present an example of zero-sum two player
stochastic game for which we encounter a degenerate
iteration when using the policy iteration algorithm for the
mean payoff problem, and showing that Step~\ref{step-degenerate}
 of Algorithm~\ref{algo-main} is essential to obtain 
the convergence of the algorithm.
 
Before doing this, let us note that some degenerate cases may be not 
so problematic.
Indeed, as observed before, the map $\bar g$ of Step~\ref{step-critical-graph}
 of Algorithm~\ref{algo-detailed}
is a polyhedral order preserving additively homogeneous convex map.
By~\cite[Theorem~1.1]{spectral2}, the set of fixed points of $\bar g$ is 
isomorphic to a convex set which dimension is the number of 
strongly connected components of the critical graph
of $\bar g$ and which is invariant by the translations by a constant function.
In particular, 
if the number of strongly connected components of the critical graph
is equal to one, then the set of fixed points of $\bar g$ is exactly equal
to the translations of $v'$ by a constant, hence $v^{(k+1)}-v'$ is a
constant function.
Since all the maps considered in Algorithm~\ref{algo-detailed} are additively
homogeneous, this implies that taking $v'$ instead of $v^{(k+1)}$,
that is applying the same steps as in the nondegenerate case,
does not change the sequence of policies $(\sigma_k)$, and 
the invariant half lines are just translated by a constant after 
this degenerate iteration. Hence, the second part of 
Step~\ref{step-critical-graph} may be avoided in Algorithm~\ref{algo-detailed},
when one encounters only such degenerate iterations.
However, to know that $\bar g$ has only one strongly connected component
in its critical graph, one need to apply the 
the first part of Step~\ref{step-critical-graph}.

We show now an example for which degenerate iterations occur with
two strongly connected components of the critical graph of $\bar g$.
We shall call these iterations \new{strongly degenerate}.

We consider a directed graph, with a set of nodes (or edges) $\X$ 
and a set of arcs $E \subset \X \times \X$, in which each arc 
$(i,j)$ is equipped with a weight $r_{ij}\in \R$, and consider the
map $f$ from $\R^n$ to itself, defined by:
\begin{equation}\label{richman-game}
f_{\sx}(\valx)= \frac 1 2 \left(\max_{j:\;(i,j)\in E} (r_{ij}+\valx_j) \;+ \;\min_{j:\;(i,j)\in E} (r_{ij} +\valx_j)\right)\enspace.
\end{equation}
When the value of $\valx$ is fixed at some ``boundary'' points, 
and the weights $r_{ij}$ are independent of $j$, the map $f$ 
arises as the dynamic programming operator 
of the ``tug of war'' game~\cite{peres}, which can viewed also as
a discretization of the infinity Laplacian operator.
Moreover the case where all the weights $r_{ij}$ are equal to zero
corresponds to a class of auction games, called Richman game~\cite{loeb}.
Therefore, the above map $f$ appears as the dynamic programming
operator of a variant of these games with additive reward and
mean payoff.

We apply the policy iteration algorithm to such a game, with a graph
of $5$ nodes and complete set of arcs $E = [5] \times [5]$.
Hence, the action spaces $A_i$ and $B_i$ in every state $i\in\X$
can be identified with the set $[5]$. 
The weight of each arc $(i,j) \in E$ is defined as the entry
$r_{ij}$ of the following matrix~:
\begin{equation} \nonumber
 r= \left( \begin{array}{c c c c c}
 1& -1& 0 &0 &0 \\
 1& -1& 0& 0& 0 \\
 0& 0 &1 &-1& 0 \\
 0& 0 &1 &-1& 0 \\
 0& -1& 0 &-1& 1 
 \end{array}
\right) \enspace,
\end{equation}
the adjacency graph of which is represented in Figure~\ref{exempleG1}. 
\begin{figure}[htbp]
  \center
 \includegraphics[scale=0.6]{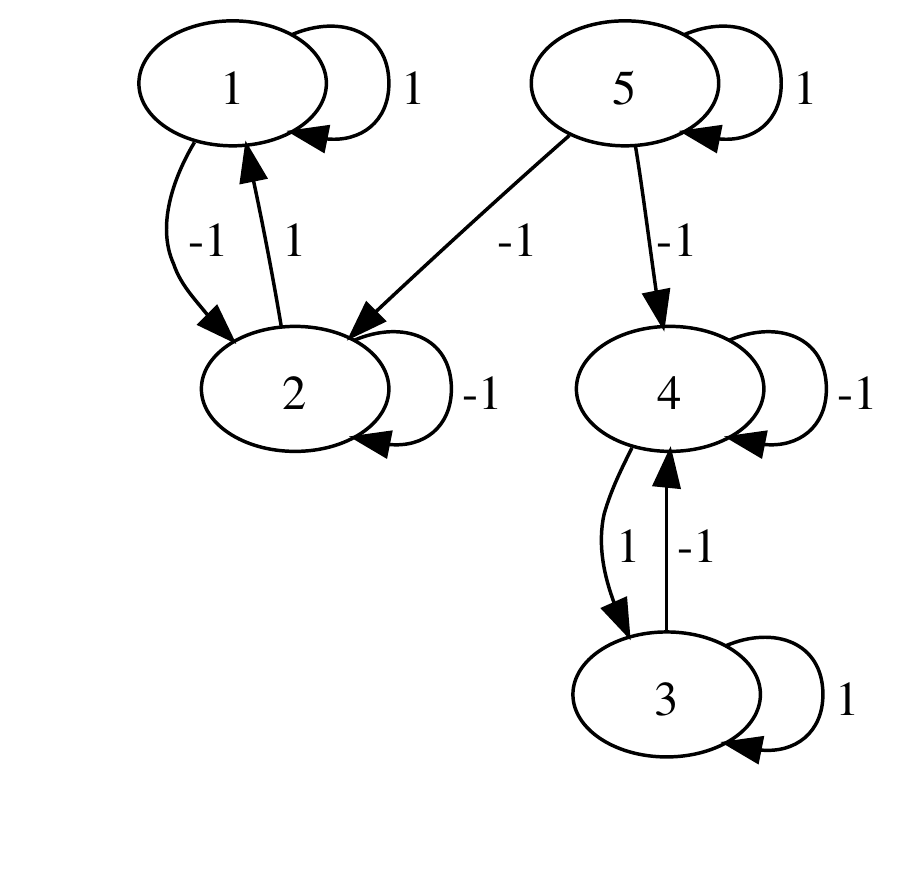}
\caption{Adjacency graph of $r$.}
\label{exempleG1}
\end{figure}

Let us fix the initial strategy $\sigmaiter 0$ for the
first player, such that 
$\sigmaiter 0(1)=2$,
$\sigmaiter 0(2)=2$,
$\sigmaiter 0(3)=4$,
$\sigmaiter 0(4)=4$,
$\sigmaiter 0(5)=2$. 
Then, the corresponding dynamic programming operator $f^{(\sigmaiter 0)}$ is given by
\begin{align*}
f_1^{(\sigmaiter 0)}(\valx) = f_2^{(\sigmaiter 0)}(\valx) & = \frac{1}{2} (-1+\valx_2+ \max(1+\valx_1,-1+\valx_2,\valx_3,\valx_4,\valx_5))\\
f_3^{(\sigmaiter 0)}(\valx) = f_4^{(\sigmaiter 0)}(\valx) & = \frac{1}{2} (-1+\valx_4+ \max(\valx_1,\valx_2,1+\valx_3,-1+\valx_4,\valx_5)) \\
f_5^{(\sigmaiter 0)}(\valx) & = \frac{1}{2} (-1+\valx_2+ \max(\valx_1,-1+\valx_2,\valx_3,-1+\valx_4,1+\valx_5))\enspace.
\end{align*}

In Step~\ref{step-init} of Algorithm~\ref{algo-main}, we compute
an invariant half-line of $f^{(\sigmaiter 0)}$ and obtain for instance
$w^{(0)}(t)=(\eta^{(0)},\valx^{(0)})$, with $v^{(0)}=(0,0,-0.5,-0.5,0)^T$
and $\eta^{(0)}=(0,0,0,0,0)^T$. Since $f(w^{(0)}(t))
<f^{(\sigmaiter 0)}(w^{(0)}(t))$, we need to improve the policy 
(Step~\ref{step-improve}) and get the unique solution (even without the
conservative policy):
$\sigmaiter 1(1)=2$,
$\sigmaiter 1(2)=2$,
$\sigmaiter 1(3)=4$,
$\sigmaiter 1(4)=4$,
$\sigmaiter 1(5)=4$. 
The corresponding operator is then 
given by~:
\begin{align*}
f_i^{(\sigmaiter 1)}&=f_i^{(\sigmaiter 0)} \qquad \ 1\le i \le 4,  \\
f_5^{(\sigmaiter 1)}(\valx)&= \frac{1}{2} (-1+\valx_4
+ \max(\valx_1,-1+\valx_2,\valx_3,-1+\valx_4,1+\valx_5)) \enspace.
\end{align*}
We compute then (in Step~\ref{step-val})  an invariant half-line
$(\eta^{(1)}, \vdemi )$ of $f^{(\sigmaiter 1)}$, and obtain 
$\eta^{(1)}=(0,0,0,0,0)^T$ and for
instance $\vdemi = (0,0,0.5,0.5,0.5)^T$.
Since $\eta^{(1)} = \eta^{(0)}$, the iteration is degenerate.

Hence the algorithm enters in Step~\ref{step-degenerate}.
Set $g := f^{(\sigmaiter 1)}$. We have to compute the
critical graph of $\bar g$, which is here equal to $g$, for instance by
applying Algorithm~\ref{algo-critique} to 
the sets $\sP_i$ defined in~\eqref{applycritical} with $u=\vdemi$.
They are given by $\sP_1 = \sP_2 = \{(0.5,0.5,0,0,0)\}$, $\sP_3
= \sP_4 = \{(0,0,0.5,0.5,0)\}$, $\sP_5 =
\{(0,0,0,0.5,0.5)\}$, then the critical graph of $g$ is equal to the final 
graph of $\sP_1\times\cdots\times \sP_5$, which is composed of two
strongly connected components with nodes $\{1,2\}$ and $\{3,4\}$.
Then, $v^{(1)}$ is the unique solution of:
\begin{equation}\nonumber \left\{ \begin{array}{l c l l}
   v_{5}^{(1)} &=& f_5^{(\sigmaiter 1)}(\valx)=
\frac 1 2 (- 1.5  + \max(0,-1,-0.5, -1.5, \valx_5^{(1)} + 1)) &  \\
   v_{\sx}^{(1)} &=& v_{\sx}^{(0)} & i \in \{1,2,3,4\} \enspace.
\end{array} \right. 
\end{equation}
We obtain $\valx^{(1)} = (0,0,-0.5,-0.5,-0.5)$ and since $f(w^{(1)}(t)) =
f^{(\sigmaiter 1)}(w^{(1)}(t))$, the algorithm stops.

However, if we do not treat the degenerate case by using
Step~\ref{step-degenerate}, and take for instance $\valx^{(1)} =\vdemi$,
we obtain $f(w^{(1)}(t)) <f^{(\sigmaiter 1)}(w^{(1)}(t))$, hence we
need to improve the strategy, and obtain the unique solution
$\sigmaiter 2= \sigmaiter 0$. This means that the algorithm cycle, showing the
necessity of Step~\ref{step-degenerate} in the policy iterations.

\section{Implementation and numerical results}\label{numerical}


The numerical results presented in this section were obtained with 
a slight modification of the policy iteration algorithm
Algorithm~\ref{algo-detailed}) and of its ingredients of
Section~\ref{sec-detailed}, all implemented in the C library {\sc
  PIGAMES}, see~\cite{detournayThese} for more information.  
All the tests of this section were performed on a single processor:
Intel(R) Xeon(R) W$3540$ - $2.93$GHz  with $8$Go of RAM. 

These slight modifications take into account the fact that 
(linear or nonlinear) equations may not be solved exactly (in exact arithmetics)
because of the errors generated by floating-point computations,
and also of the possible use of iterative methods instead of exact methods.
Let us explain them briefly.
For instance, the stopping criterion in Step~\ref{step-stop-detail} of
Algorithm~\ref{algo-detailed} can be replaced by a condition on the
residual of the mean payoff, $\fr{f} (\etaold) - \etaold$ and the
residual of the relative value, $\tfx{f}{\eta}(\vold) - \etaold - \vold$.
Here, we consider the infinity norm of the residual of the game 
that we define as $0.5 * (\|\fr{f} (\etaold) - \etaold\|_\infty +
\|\tfx{f}{\eta}(\vold) - \etaold - \vold\|_\infty)$, where $\|\cdot\|_\infty$
denotes the sup-norm. Then, we stop the policy iterations when
the infinity norm of the residual of the game is smaller than a given value
$\epsilon_{g} > 0$ or when the strategies cannot be improved. For the
tests of this section, we took $\epsilon_{g} = 10^{-12}$. We use the
same condition for the stopping criterion of the intern policy
iterations, that is for Step~\ref{stopHoward} of Algorithm~\ref{algo-Howard}
and Step~\ref{stop-DF} of Algorithm~\ref{algo-DF}. 
Moreover, the optimization problems in Step~\ref{step-improve-detail}
of Algorithm~\ref{algo-detailed} and Step~\ref{step-improvehoward} 
of Algorithms~\ref{algo-Howard} and~\ref{algo-DF},
are solved up to some precision. This means for instance that 
in Algorithm~\ref{algo-detailed}, one choose $\sigmanew\in\Am$  such that,
for all $i\in\X$,
\begin{equation}\label{sys-improve-mod}
\left\{ \begin{array}{l}
\displaystyle 
\tfxvia{F}{\etaold}{\vold}{\sx}{\sigmanew(i)}\leq\epsilon_{v}+ \underset{ a \in \tAm{\sx}{\etaold}{\epsilon_{\eta}}}{\operatorname{min}} \left\{ \tfxvia{F}{\etaold}{\vold}{\sx}{a} \right\}\; \text{with} \\
\displaystyle  \tAm{\sx}{\eta}{\epsilon}:=
\set{a \in \A_\sx}{ \frvia{F}{\eta}{\sx}{a} \leq \epsilon+
\fr{f}(\eta)]_{\sx}
}\\
\sigmanew(i) = \sigmaold(i)\text{ if }\sigmaold(i)\text{ is optimal,}  
\end{array}\right. 
\end{equation}
for some given $\epsilon_{\eta} $ and $\epsilon_{v} >0$.
Finally, the linear systems in Step~\ref{howard-step-2}
of Algorithms~\ref{algo-Howard} and~\ref{algo-DF} are solved up
to some precision, which may be lower bounded 
when the matrices of the systems are ill-conditioned.
See the appendix for details about the solution of these linear systems.

\subsection{Variations on tug of war and Richman games}
\label{section-numeric-richman}

We now present some numerical experiments on the variant of Richman
games defined in Section~\ref{exemple}, constructed on  random graphs. 
As in the previous section, we consider
directed graphs, with a set of nodes equal to $[n]$ and a set of arcs $E\subset
[n]^2$. The dynamic programming operator is the map $f$ defined 
in~\eqref{richman-game}, where the value $r_{ij}$ is the reward
of the arc $(i,j)\in E$.
In the tests of Figure~\ref{figure-richman-deg} to 
Figure~\ref{figure-richman-time1}, we chose random
sparse graphs with a number of nodes $n$ between $1000$ and
$50000$, and a number of outgoing arcs fixed to ten for each node.
The reward of each arc in $E$ has value one or zero, that is
$r_{ij} = 1$ or $0$. The arcs $(i,j)\in E$ and the associated
rewards $r_{ij}$ are chosen randomly (uniformly and independently). We start the
experiments with a sizer of graph (number of nodes) equal to $n=1000$, 
then we increase  the size by 
$1000$ until reaching $n=10000$, after we increase the size by $10000$
and end with a number of $50000$ nodes. For each size that we
consider, we made a sample of $500$ tests. The results of the
application of the policy iteration (Algorithm~\ref{algo-detailed} 
with the above modifications) on those games are presented
in Figures~\ref{figure-richman-deg} to~\ref{figure-richman-time1} and
are commented below.

Figure~\ref{figure-richman-deg} gives for each size $n$, 
and among the sample of $500$ tests, the number of
tests that encountered at least one strongly degenerate policy iteration for the first 
player. Hence, these games require 
the degenerate case issue presented in this paper, that is
Step~\ref{step-critical-graph} of Algorithm~\ref{algo-main} or~\ref{algo-detailed}. Moreover,
from the data of Figure~\ref{figure-richman-deg}, we
observe that approximately between $10$ and $15$ percent of the tests
have at least one strongly degenerate policy iteration for the first player.
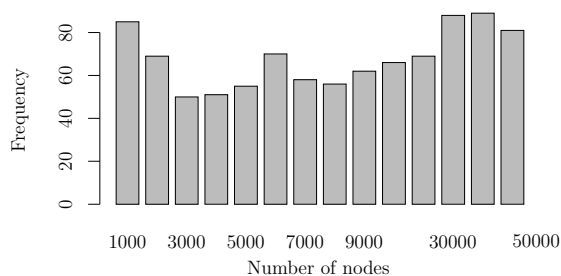
\begin{figure}[htbp]
\centering
\resizebox{0.5\textwidth}{!}{
\input{degdata-tex}
}
\caption{Tests on a variant of Richman games constructed on random graphs.
The histogram shows for each size (number of nodes), the number of tests
  having at least one strongly degenerate policy iteration
  for the first player, among $500$ tests.}
\label{figure-richman-deg}
\end{figure}

In the table below we report the number of strongly degenerate iterations that
occur in the  global sample of tests.
\[ 
\begin{tabular*}{\textwidth}{@{\extracolsep{\fill}}r|ccccc}
Number of strongly degenerate iterations &   $0$&   $ 1 $& $  2 $ & $ 3 $& $  6 $\\
\hline
Number of tests & $6051$& $ 919$ & $ 28 $ & $ 1 $& $  1 $ \\ 
\end{tabular*}
\]
We observe that in general there is no more than one or two strongly degenerate
policy  iterations for our sample of tests. Note that in this section,
a strongly degenerate policy iteration is to be understood as a strongly degenerate
iteration for the first player only, that is for
Algorithm~\ref{algo-detailed}. 

In Figure~\ref{figure-richman-iter1}, we draw  on the left curves that
represent the number of policy iterations for the first player, that
is the number of iterations of Algorithm~\ref{algo-detailed}, as a
function of the size $n$ of the graph.  The dashed lines on top and
bottom are respectively the maximum and minimum value, over the sample
of $500$ tests, and the plain line is the average value, all as a 
function of the size.  We observe that the average number of first player's
policy iterations is almost constant as the size
increases. Using the same model of representation, we show on the right
of Figure~\ref{figure-richman-iter1} respectively
the maximum, average and minimum values for the total number of policy
iterations for the second player, that is the sum of the numbers of
iterations of Algorithm~\ref{algo-DF} when applied by
Algorithm~\ref{algo-detailed}, as a function of the size. We also
observe that these values do not vary a lot with the size.

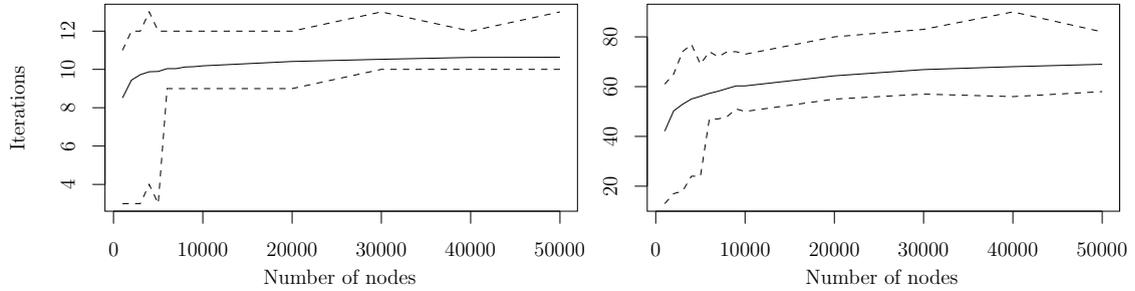
\begin{figure}[htbp]
\centering
\resizebox{\textwidth}{!}{
\input{iter1_mean-tex}
\input{iterTot_mean-tex}
}
\caption{Tests on a variant of Richman games constructed on random graphs.
On the left,  the curves from top to bottom represent respectively the maximum, 
average, minimum number of first player's policy iterations, 
among $500$ tests, as a function of the number of nodes.
On the right, the curves represent the total number of second player's
 policy iterations.}
\label{figure-richman-iter1}
\end{figure}

In Figure~\ref{figure-richman-time1}, we present  on the left the total
cpu time (in seconds) needed by the policy iteration 
to find the solution of the game. 
As for the two previous figures, the curves from
top to bottom show respectively the maximum, average and minimum
values, over the sample of $500$ tests,
 as a function of the size of the graphs. Finally, on the right of 
Figure~\ref{figure-richman-time1}, we give also the
average of the total cpu time (in seconds) needed to solve the game
but we separated the tests with strongly degenerate policy iteration(s),
represented by the dashed line, from the non strongly degenerate ones,
represented by the plain curve. 
We observe that the average cpu time is somewhat greater for the tests
with strongly degenerate iteration(s). This is due to the additional
steps needed for degenerate iterations. Indeed, the cpu time 
of a degenerate iteration should be approximately the double of that of a
nondegenerate iteration, and  since the number of policy iterations is around 10
in the sample of tests, the average of the total cpu time of tests with
(strongly) degenerate iterations should be approximately $10$ percent
greater than that of the other tests.

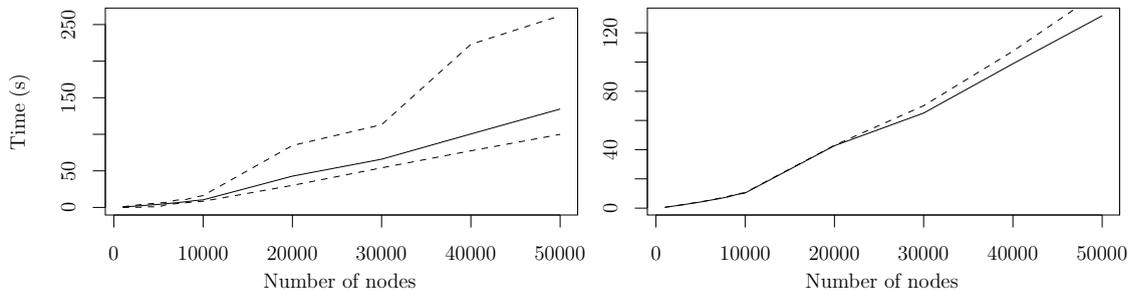
\begin{figure}[htbp]
\centering
\resizebox{\textwidth}{!}{
\input{time_mean-tex}
\input{time1_onlymean-tex}
}
\caption{Tests on a variant of Richman games constructed on random graphs.
On the left, the curves from top to bottom represent 
  respectively the maximum, average and minimum values
of the total cpu time (in seconds) taken by the policy iteration algorithm,
among $500$ tests,  as a function of the number of nodes.
On the right, the dashed line represents the average among the
  tests that encounter at least one strongly degenerate policy iteration for
  the first player, whereas the plain line represents the average  
  among the other tests.}
\label{figure-richman-time1}
\end{figure}

In addition, in Table~\ref{table-richman-large-graphs}, we give 
numerical results for ten tests of the variant of Richman game, 
constructed on random large graphs with a number of nodes 
between $10^5$ and $10^6$. 
We observe that the number of iterations are of the same
order as for the previous sample of tests presented in
Figure~\ref{figure-richman-iter1}.

\begin{table}[htbp]
\centering
\caption{Numerical results on a variant of Richman game constructed on random
  large graphs.} 
\label{table-richman-large-graphs}
\begin{tabular*}{\textwidth}{@{\extracolsep{\fill}} cccccc}
\hline 
 Number of & Iterations of  & Total number  & Strongly degenerate & Infinity norm  & CPU time \\ 
nodes  &  first player  &  of iterations  & iterations & of residual & (s)  \\ 
\hline 
$100000$  & $12$  & $78$  & $1$  & $1.44e-14$  & $3.24e+02$   \\ 
$200000$  & $12$  & $74$  & $0$  & $7.44e-15$  & $7.90e+02$   \\ 
$300000$  & $11$  & $82$  & $0$  & $1.33e-15$  & $9.38e+02$   \\ 
$400000$  & $12$  & $82$  & $1$  & $8.55e-15$  & $1.42e+03$   \\ 
$500000$  & $12$  & $77$  & $1$  & $2.00e-14$  & $2.16e+03$   \\ 
$600000$  & $12$  & $77$  & $0$  & $8.66e-15$  & $2.61e+03$   \\ 
$700000$  & $11$  & $85$  & $0$  & $3.02e-14$  & $2.61e+03$   \\ 
$800000$  & $12$  & $81$  & $1$  & $4.82e-14$  & $6.79e+03$   \\ 
$900000$  & $12$  & $79$  & $1$  & $1.27e-14$  & $4.17e+03$   \\ 
$1000000$ & $12$  & $90$  & $1$  & $3.33e-15$  & $1.96e+04$  \\ 
 \hline 
\end{tabular*} 
\end{table}

\subsection{Pursuit games}

\label{section-numeric-pursuit}

We consider now a pursuit evasion game with two players~: a pursuer
and an evader. The evader wants to maximize the distance between him
and the pursuer and the pursuer has the opposite objective. See for instance
\cite{BaFaSo94,BaFaSo99,LiCruz08} for a complete description of general
pursuit games.
To simplify the model, we consider as state of the
game, the distance between the two players. Then, the state of the
game is given by $x = x_P - x_E$ where $x_P$ is the position of the
pursuer and $x_E$ the position of the evader. We also restrict the
state $x$ to stay in a unit square centered in the $0$-position, that
is $x \in X := [-0.5,0.5] \times [-0.5,0.5]$.  
At each time of the game, the reward for the evader is the
euclidean square norm of the distance between the two players,
i.e. $\norm{x}_2^2$.  
Such a game is a special class of differential game, the dynamic
programming equation of which is an Isaacs partial differential 
equation.
Under our simplifications and assumptions, the Hamiltonian of this
equation is given by~:
\begin{equation}\label{hamiltonian} 
H(x,p)= \max_{a\in \A(x)} (a\cdot p)+ \min_{b\in \B(x)} (b\cdot p )+\|x\|_2^2
\qquad\forall x \in X,  p\in\R^2\enspace ,
\end{equation}
meaning that in the case of a finite horizon problem, the
Isaacs equation would be given, at least formally (but also in
the viscosity sense) by~:
\[ -\frac{\partial v}{\partial t} + H(x, \nabla v(x))=0 \qquad x \in  X
\enspace.\]
Here $A(x)$ and $B(x)$ are the sets of possible directions for 
the evader and the pursuer respectively, when the state is equal to $x\in X$.
On the boundary, we consider that only actions
keeping the state of the game in the domain $X$ are allowed,
hence the above equation has to be satisfied until the boundary.

We shall consider this differential game with a mean-payoff criterion
and the above reward.
This means that the analogous to System~\eqref{DPsyst2P1} is the following
system of Isaacs equations~:
\begin{equation}\label{eq-BI-CM} \left\{ \begin{array}{rl}
\displaystyle 
\max_{a\in \A(x)} (a\cdot\nabla \eta (x))+ \min_{b\in \B(x)}
  (b\cdot\nabla \eta (x)) = 0\enspace, & x \in  X\enspace, 
 \\[1em]
\displaystyle
- \eta (x)  +\max_{a\in \tinf\A_\eta(x)} (a\cdot\nabla \valx (x) )+
\min_{b\in \tinf\B_\eta (x)}
  (b\cdot\nabla \valx (x) )+\|x\|_2^2=0\enspace,  & x \in X\enspace, 
\end{array} \right. 
\end{equation}
where
\begin{align*}
\tinf\A_\eta(x) :=& \underset{a \in \A(x)}{\operatorname{argmax}} \;\left(
a\cdot\nabla \eta (x)  \right)\enspace, \\
\tinf\B_\eta(x) :=& \underset{b \in \B(x)}{\operatorname{argmin}} \; \left(
b\cdot\nabla \eta (x) \right) \enspace.
\end{align*}
In classical pursuit-evasion games, such as in
\cite{BaFaSo99}, the reward is constant
and the value function is defined as the time (or the exponential
of the opposite of the time) for the
pursuer to capture the evader, then the value function is solution of
the stationary Isaacs equation that is~\eqref{eq-BI-CM} with $\eta\equiv 0$,
corresponding to the above Hamiltonian with $1$ instead of $\|x\|_2^2$.
In that case, the value is infinite
when the pursuer's speed is smaller than the evader's speed, and 
it is would be difficult to compute an optimal strategy using
Isaacs equation.
Here by considering a mean-payoff problem, we may solve the problem
even when pursuer's speed is smaller than the evader's speed,
as we shall see below. Note that one may have kept the reward equal to 
$1$, but then the optimal value $\eta$ would have given less
information.

A monotone discretization, for instance a 
finite difference discretization scheme (see~\cite{KuDu92}), 
of System~\eqref{eq-BI-CM} 
yields to System~\eqref{DPsyst2P1} for the dynamic programming
operator $f$ of a discrete time and finite state space game,
which then may be solved using our policy iteration 
Algorithm~\ref{algo-detailed}.

In our tests, the domain $X$ is discretized in each directions with a
constant step size $h$. Then the two players of the discrete game 
are moving on the discretized
nodes of the domain, similarly to the moves in a chess game.
We assume also that the evader cannot move when the euclidean norm of the
relative distance between him 
and the pursuer is less than $0.1$, i.e when $x \in
\Bl((0,0);0.1)$. We shall call the evader, the mouse and his set of
possible actions at each state of the game will given by~: 
\begin{equation*} 
\A(x) := \begin{cases} \{(a_1, a_2)\,|\, a_l \in \{0, 1, -1\}, \ \ l =
1,2\} & x \in \inte{X} \setminus \Bl((0,0);0.1) \\
\{(0, 0)\} & x \in \Bl((0,0);0.1) \enspace,
\end{cases}
\end{equation*}
where $\inte{X}$ denotes the interior of $X$. 
The pursuer, that we shall call the cat, has the following set of
possible actions~: 
\begin{equation*} 
\B(x) := \{(b_1, b_2)\,|\, b_l \in \{0, \bar b, -\bar b\}, \ \ l = 1,2\} \quad x
\in \inte{X} \enspace,
\end{equation*}
where $\bar b$ is a positive real constant and represents the speed of the cat.
Moreover, on the boundary of $X$, the sets $\A(x)$ and $\B(x)$ are 
restricted to avoid actions that bring the state out of $X$.

Numerical results for this game are presented in
Table~\ref{table-CM10001} when $\bar b =0.999$ , $\bar b =1$ and
$\bar b =1.001$ respectively.
 Note that the solution of the
discretization of Equation~\eqref{eq-BI-CM} may differ from the
solution of the continuous equation.
We observe that for $\bar b = 0.999$  and $\bar b =1.001$, we have a
strongly degenerate iteration for the first player on the last iteration. 

The optimal actions for the discretized problem with $\bar b =0.999$
are represented in Figure~\ref{betaCM0999}, at each node of the
grid: the actions of the mouse are on the left, and that of the cat are 
on the right. The optimal
actions are approximately the same for the two other values of $\bar b$.
When $\bar b =0.999$, the speed of the cat is smaller than the
speed of the mouse ($=1$). The numerical results for the discretized
game give an optimal mean-payoff $\eta$ such that
 $\eta (x) = 0.492$ for $x \in X \setminus \Bl((0,0);0.1)$ and
$\eta(x) = 0$ for $x \in \Bl((0,0);0.1)$. This means that the cat
cannot catch the mouse when their starting positions are not too
close and the mouse can keep almost the maximum distance between
them. The relative value is represented  on the left of Figure~\ref{vCM1001}.
When $\bar b =1$, the speeds of the cat and the mouse are equal.
The numerical results for the discretized
game give a relative value $\valx$
approximately equal to zero for every starting point and 
an optimal mean-payoff $\eta(x) \approx \norm{x}^2_2$, meaning
that the cat and mouse keep the same initial distance all along the game.
In the last example, the speed of the cat $\bar b =1.001$ is greater than
that  of the mouse ($=1$).
The numerical results for the discretized
game give an optimal mean-payoff $\eta$ close to zero.
The relative value $\valx$ is given on the right of Figure~\ref{vCM1001}.
In this case, the cat catches the mouse. 

\begin{table} 
\center
 \caption{Numerical results for the mouse and cat example where $\bar
 b$ is the speed of the cat. The second column is the index of the iteration
 on the cat's policies and the third column is the corresponding
total number of iterations on the mouse's policies. The last column
 indicates if the cat's policy iteration is strongly degenerate. Number of discretization nodes: $257 \times 257$.}  
\label{table-CM10001}
 \begin{tabular*} {\textwidth}{@{\extracolsep{\fill}} cccccc}
\\[-0.8em] \hline \\[-0.8em]
$\bar b$ & Cat policy & Number of mouse  & Infinite norm of
  & CPU time  & Strongly degenerate\\
&    iteration index &  policy iterations  & residual & (s) &  iteration \\
 \hline \\[-0.8em]
$0.999$ & $ 1 $ & $ 2 $ & $ 1.25e-06 $  & $ 2.59e+01 $ &  $0$ \\ 
& $ 2 $ & $ 1 $ & $ 9.93e-12 $  & $ 3.95e+01 $ &  $0$ \\ 
& $ 3 $ & $ 1 $ & $ 5.68e-14 $  & $ 7.35e+02 $ & $1$ \\ 
 \hline \\[-0.8em]
$1$ & $ 1 $ & $ 2 $ & $ 1.25e-06 $  & $ 2.60e+01 $ & $0$  \\ 
& $ 2 $ & $ 1 $ & $ 3.39e-21 $  & $ 3.84e+01 $ & $0$  \\ 
 \hline \\[-0.8em]
$1.001$ & $ 1 $ & $ 2 $ & $ 1.25e-06 $  & $ 2.59e+01 $ & $0$  \\ 
& $ 2 $ & $ 1 $ & $ 1.96e-14 $  & $ 6.51e+02 $ &  $1$ \\ 
 \hline 
\end{tabular*} 
 \end{table}

\begin{figure}[htbp] 
\center
\resizebox{\textwidth}{!}{
\includegraphics[scale=0.1]{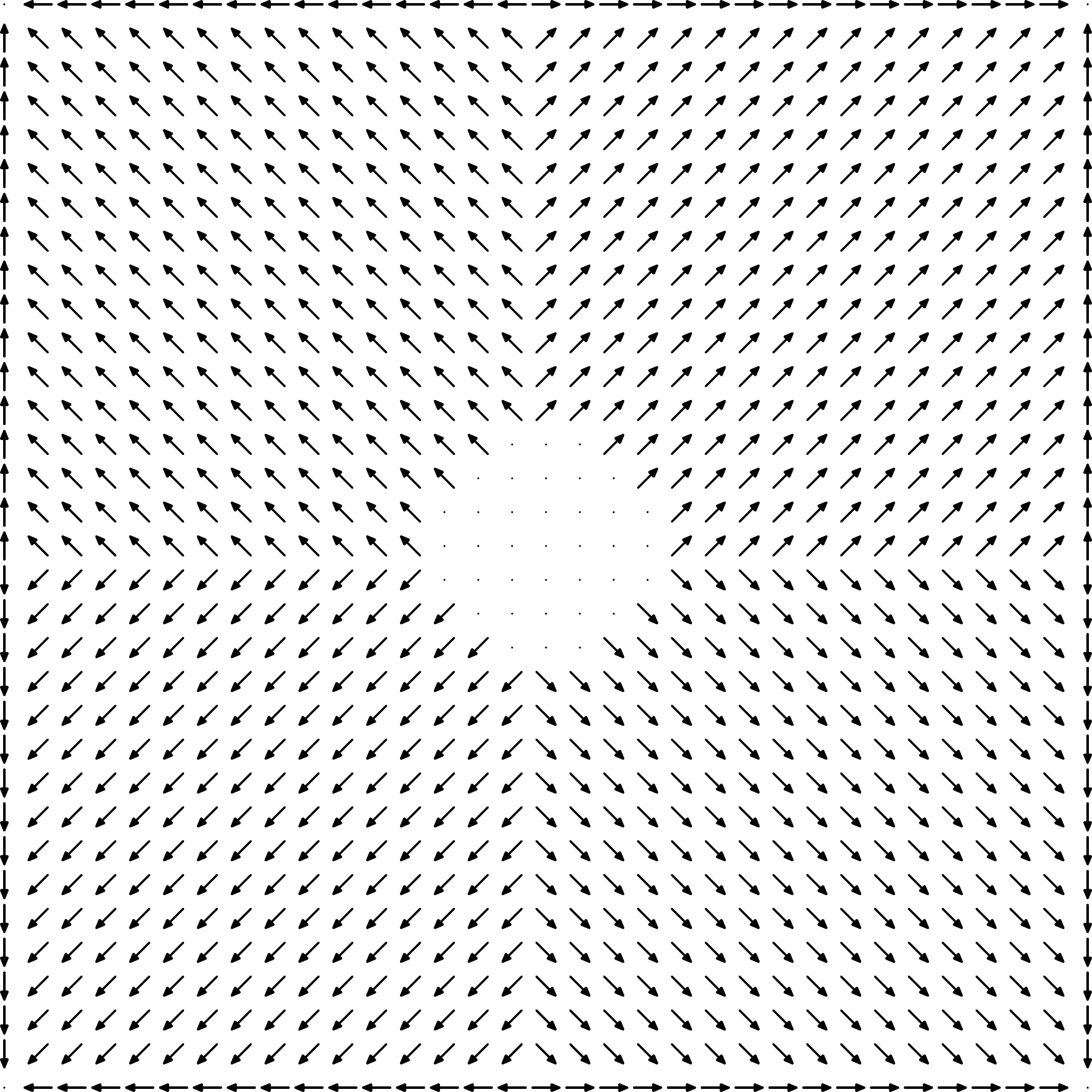}
\ 
\includegraphics[scale=0.1]{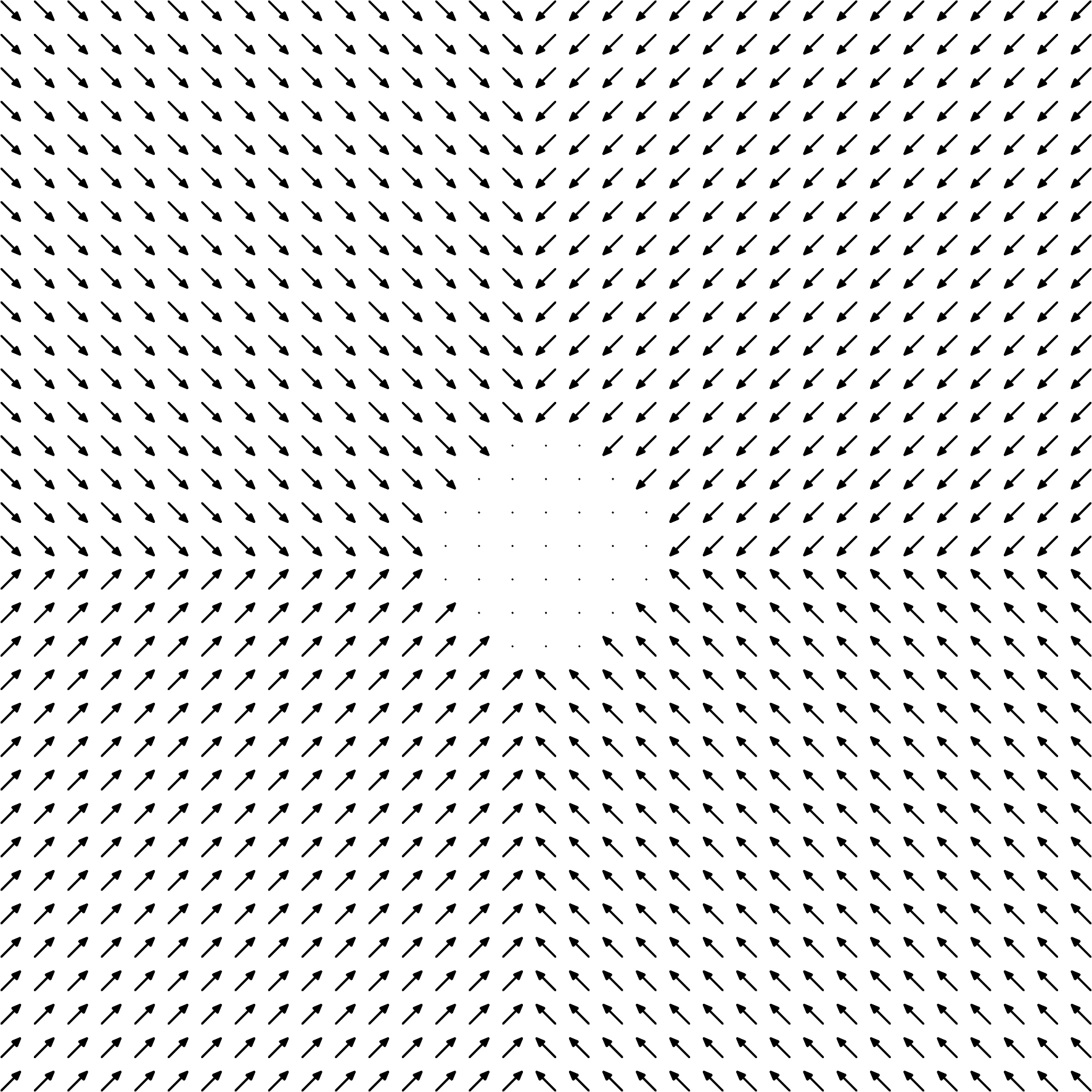}
}
\caption{Optimal actions for the mouse on the left and for the cat on
  the right.}
\label{betaCM0999}
\end{figure}

\begin{figure}[htbp] 
\resizebox{\textwidth}{!}{
\input{CM0999_v-tex}
\input{CM10001_v-tex}
}
\caption{Relative value $\valx$ for the mouse and cat game when the
  speed of the mouse is one, the speed of the cat equals $0.999$ on
  the left and $1.001$ on the right.}  
\label{vCM1001}
\end{figure}
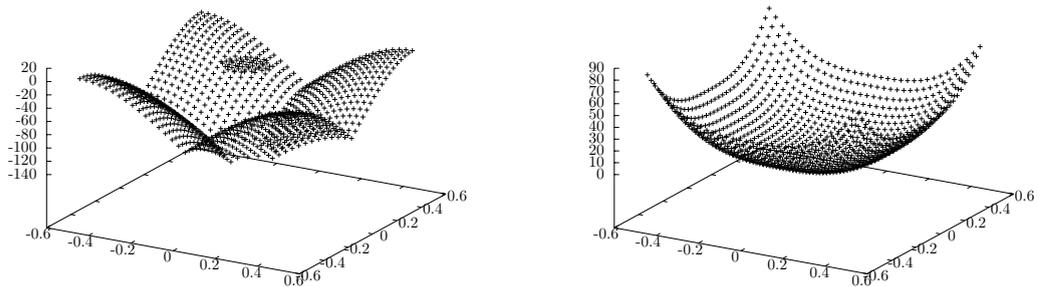

\bibliography{references}
\bibliographystyle{alpha}

\appendix

\section{Details of implementation of Policy Iteration for multichain one player games}

We explain in more details here why System~\eqref{DF-LS} has a unique solution
and is selecting one special solution of 
System~\eqref{DP-syst-step-val-oneplayer}
with $k$ instead of $k+1$, and how it is solved practically
(see also~\cite{Denardo-Fox68,puterman}).

Recall that System~\eqref{DP-syst-step-val-oneplayer} is 
of the form~\eqref{LSDF} rewritten here:
\[\left\{ \begin{array}{r  c l}
  \eta &=& P \, \eta \\
  \eta + \valx &=& P \, \valx + r\enspace,
\end{array}
\right.\]
where $r=r^{(\delta)}\in \RX$ and $P=P^{(\delta)}$ is a stochastic matrix, with
$\delta=\deltanew$.
Moreover, System~\eqref{DF-LS} corresponds to 
\begin{equation}\label{DF-LS-gen}
\left\{ \begin{array}{r  c  l l}
  \eta_{\sx} &=& (P \eta )_{\sx} & \sx \in [n]  \setminus S\enspace,\\
  \eta_{\sx} + v_{\sx} &=& (P v)_{\sx} +r_\sx  & \sx \in \X\enspace, \\
  v_{\sx} &=& 0 & \sx \in S \enspace, \\
\end{array}
\right.
\end{equation}
where $r$ and $P$ are as before but with $\delta=\deltaold$, and 
where $S$ is composed of minimal indices $i_F$ of each final
classes $F$ of $P$.
Then one need to show that System~\ref{DF-LS-gen} has a unique solution
and is selecting one special solution of System~\eqref{LSDF}.

First, for all final classes $F$ of $P$,
$P_{FF}$ is an irreducible Markov matrix, hence the equation
$\eta_F=(P\eta)_F$, which is equivalent to $\eta_F=P_{FF} \eta_F$, is also
equivalent to the condition that $\eta_{\sx} = \eta_j$ for all $i,j\in F$.
Moreover, $P_{FF}$ has a unique stationary (or invariant) probability measure  $\staMC_F$,
that is a row probability vector solution of $\staMC_F = \staMC_F P_{FF}$, and
this vector has strictly positive coordinates.
This implies that one can eliminate, for each final class $F$ of $P$, 
one equation with index $i$ in $F$
in the equation $\eta=P\eta$, without changing the set of solutions.
Hence, any solution of System~\ref{DF-LS-gen} is also solution of~\eqref{LSDF}.

Second, denote by $\sF$ the  union of final classes, and by 
$\sT$ the union of transient classes, that is the complement in $\X$ of $\sF$.
Then, $w$ is in the  kernel of $I-P$ if, and only if, it 
satisfies $w_F=P_{FF}w_F$, for all final classes $F$ of $P$ and
$w_\sT=P_{\sT\sT}w_\sT+P_{\sT\sF} w_\sF$.
As said before the first equations are equivalent to the
conditions $w_{\sx} = w_j$ for all $i,j\in F$.
Since $P_{\sT\sT}$  has transient classes only, it has a 
spectral radius strictly less than one, which implies that given the vectors
$w_F$ for all final classes $F$, the equation
$w_\sT=P_{\sT\sT}w_\sT+P_{\sT\sF} w_\sF$ has a unique solution
$w_\sT$.
Hence, the dimension of the kernel of $I-P$ is equal to the number of 
final classes of $P$,
and any element of this kernel which has one coordinate 
$i\in F$ equal to zero for each final classes $F$ of $P$, 
has all its coordinates equal to zero.
This implies that given $\eta\in\RX$, the solution $\valx$ of 
System~\eqref{DF-LS-gen} is unique if it exists
(the difference between two such solutions satisfies the above conditions).
This also implies that the codimension of the image of $I-P$
is equal to the number of final classes.
Hence, the image of $I-P$ is exactly equal to the set of vectors $\eta\in\RX$
such that $ \staMC_F \eta_F= 0$, for all final classes $F$ of $P$.
A vector $\eta\in\RX$ is such that System~\eqref{DF-LS-gen} has a solution
$\valx\in\RX$ 
if and only if $\eta=P \eta$ and $\eta-r$ is in the image of $I-P$.
These conditions are equivalent to the three conditions
$\eta_{\sx} = \eta_j$ for all $i,j\in F$, for all final classes $F$,
$\eta_\sT=P_{\sT\sT}\eta_\sT+P_{\sT\sF} \eta_\sF$,
and $\staMC_F \eta_F=\staMC_F r_F$, for all final classes $F$.
The first and third conditions together are equivalent to 
$\eta_{\sx} =\staMC_F r_F$, for all $\sx\in F$, and all final classes $F$
of $P$, which gives a unique solution $\eta_\sF$. Since 
the second one has a unique solution $\eta_\sT$, given $\eta_\sF$,
we get that there is a unique vector $\eta\in\RX$ such that
System~\eqref{DF-LS-gen} has a solution $\valx\in\RX$.
In conclusion, System~\eqref{DF-LS-gen} has a unique
solution $(\eta,\valx)$, which finishes the proof what we wanted to show.

One may try to solve System~\eqref{DF-LS-gen} by using usual LU methods,
however when $P$ is not irreducible, such a method is not robust.
We rather use the decomposition of $P$ in classes, and
the previous properties.
In particular, since $P_{\sT\sT}$ has a spectral radius strictly less than $1$,
one can compute $(\eta,\valx)$ solution of System~\eqref{DF-LS-gen}
 by first computing 
$\eta_F$ and $\valx_F$, for all final classes $F$ of $P$, then computing 
successively $\eta_\sT$ and $\valx_\sT$
which are respectively fixed points of contracting affine systems with 
tangent linear operator $P_{\sT\sT}$:
\begin{equation*}
\left\{ \begin{array}{r  c l}
  \eta_\sT &=& P_{\sT\sT} \,\eta_\sT + P_{\sT\sF} \, \eta_\sF  \\
 \valx_\sT &=& P_{\sT\sT} \, \valx_\sT + P_{\sT\sF} \valx_\sF + r_\sT -  \eta_\sT 
\enspace.
\end{array}
\right.
\end{equation*}

There exists two ways to compute $\eta_F$ and $\valx_F$. 
One is to compute the stationary probability $\staMC_F$
to determine $\eta_F$ by $\eta_{\sx} =\staMC_F r_F$, for all $\sx\in F$,
and then solve the following system with unknown $\valx_F\in\R^F$~:
\[
\valx_F = P_{FF} \, \valx_F + r_F -\eta_{F}  \enspace,
\]
by eliminating one equation (since one equation is redundant)
with index $j\in F$,
and adding the condition $\valx_i=0$ for one element $i\in F$.
Another method is to
consider $\eta_F$ as constant, say $\eta_\sx = \bar \eta$ for $\sx\in F$,
and solve the system with unknowns $\bar \eta\in\R$ and $\valx\in\R^F$~: 
\[
\bar \eta +  \valx_\sx = \sum_{\sy \in F} P_{\sx \sy} \, \valx_\sy +
r_\sx, \quad \sx \in F \enspace,
\]
by adding  the condition $\valx_i=0$ for one element $i\in F$.
In our algorithm,  we choose the index $i=j\in F$ to be the minimal index of 
$F$ (for a fixed total ordering of nodes).
This method gives the following algorithm to solve System~\eqref{DF-LS-gen}.

\begin{algo}[Solution of System~\eqref{DF-LS-gen}]
\label{algo-DFstep1} 
Decompose the matrix $P$ into irreducible classes and permute nodes
without changing the order in each class, such
that $P$ takes the following form~:
\begin{equation} \nonumber
P=\left(\begin{array}{ccccc}
P_{11} & P_{12} & \ldots   & \ldots & P_{1m}\\
0 & P_{22} & \ldots &   \ldots & P_{2m}\\
\vdots & \vdots &  \vdots & \vdots & \vdots\\
0 & \ldots &  \ldots & P_{m-1,m-1} & P_{m-1,m}\\
0 & \ldots &  \ldots & 0 & P_{mm}
\end{array}\right)
\end{equation}
where $m$ denotes the number of irreducible classes and $P_{II}$ are
square irreducible submatrices of $P$, for $I = 1 , \dots, m$. Note that, the
class corresponding to a submatrix $P_{II}$ is final if and only if
the submatrices $P_{IJ}$ are all null for all $J\ne I$.

\noindent For each class $I$ from $m$ to $1$, do the following~:
\begin{enumerate} [Step 1.]
\item \label{algo-DFstep1-step1} If $I$ corresponds to a final class,
that is $P_{II}$ is a stochastic matrix,
  do one of the two following sequences of operations~:
  \begin{enumerate}[A.]
    \item \begin{enumerate}[(a)]
      \item Find the stationary probability $\pi_{I}$ of $P_{II}$:
 $\pi_{I} \, P_{II}  =  \pi_{I} \enspace, $
      \item Set $\bar \eta  = \pi_{I} \, r_I$ and $\eta_{\sx} = \bar \eta
      \quad \sx \in {I} \enspace,$
      \item Solve the system with unknown $v_I\in\R^I$~:
	\begin{equation}\label{eq-LS-sing-algo7}
	  \left\{ \begin{array}{r  c l l}
	   \valx_\sx & = &  \sum_{\sy \in I} P_{\sx\sy} \valx_\sy + r_\sx - \bar \eta  & \sx
	    \in I \setminus S \enspace, \\
	    \valx_{\sx} &=& 0  & \sx \in S \cap I \enspace,
	  \end{array}
	  \right.
	\end{equation}
      \end{enumerate}
    \item Solve the system with unknowns $v_I\in\R^I$ and $\bar \eta\in\R$~:
      \begin{equation*}
	\left\{ \begin{array}{r  c l l}
	  \bar \eta + \valx_\sx & = &  \sum_{\sy \in I} P_{\sx\sy} \valx_\sy + r_\sx  & \sx
	  \in I \enspace, \\
	  \valx_{\sx} &=& 0  & \sx \in S \cap I \enspace,
	\end{array}
	\right.
      \end{equation*}
      and set $\eta_{\sx}  = \bar \eta \quad  \sx\in {I} \enspace,$
  \end{enumerate}
\item \label{algo-DFstep1-step2} if $I$ corresponds to a transient class,
that is if $P_{II}$ is a strictly submarkovian matrix.
 do the following steps~:
\begin{enumerate}
\item compute $\eta_I$ solution of the following system~:
\[
\eta_{I}  = P_{II} \eta_{I} + \sum_{J>I} P_{IJ} \eta_{J} \]
\item compute $\valx_I$ solution of the following system~:
\[
 \valx_{I} = P_{II} \valx_{I} + \sum_{J>I} P_{IJ} \valx_{J} + r_{I} - \eta_{I}\enspace .
\]
\end{enumerate}
\end{enumerate}

\end{algo}

In our numerical experiments, the
linear system~\eqref{DF-LS} at each intern policy iteration
is solved by using Algorithm~\ref{algo-DFstep1}. 
For the numerical experiments
of Section~\ref{section-numeric-richman}, on each final class, we used
a SOR iterative solver to find the stationary probability $\pi_I$ and also
to compute the corresponding $v_I$ in method A in
Step~\ref{algo-DFstep1-step1} of Algorithm~\ref{algo-DFstep1}. 
For the transient class, we used the LU solver of the package~\cite{superlu99}.

The {Successive Over-Relaxation (SOR)} method is an iterative scheme
that belongs to the class of splitting methods or relation methods,
see for instance~\cite{plemmons}. It is derived from the
Gauss-Seidel relaxation scheme. Consider a matrix $A\in\R^{n\times n}$
such that $A = D - L - U$ where $D$, $-L$, $-U$ are respectively the
diagonal, lower and upper triangular part of $A$. The SOR smoothing
operator is defined by $S_w = M^{-1} N$ where $M= D \,-\,w
L$ and $N = [(1-w)D\,+\,w U]$ for $0 <w < 2$.

Consider the irreducible stochastic matrix $P_{II}\in\R^{I\times 
I}$ and decompose  $\Idt - P_{II}^T = D - L -U$ where $\Idt$ is the
identity matrix of $\R^{I \times I}$. Starting from an
initial positive approximation $\pi^{(0)}\in\R^I$, a SOR smoothing
step to find the stationary probability of $P_{II}$ is given by~:   
\begin{equation*}
\begin{aligned}
  \pi^{(k)} &\, = \, S_w \,\pi^{(k-1)} \\
  \pi^{(k+1)} &\, = \, \frac{\pi^{(k)}} {\left( \sum_{i\in
  [n]}\pi^{(k)}_i \right)} \enspace\cdot
\end{aligned}
\end{equation*}
The sequence $(\pi^{(2k)})_{k \ge   0}$ converges to the transpose of the
stationary probability of $P_{II}$ when the limit $\lim_{k\rightarrow \infty}
S_w^{(k)}$ exists, see~\cite{plemmons} for more details.
To solve Equation~\eqref{eq-LS-sing-algo7}, decompose $\Idt - P_{II}
= D - L - U$. Then, starting from an initial 
approximation $v^{(0)}\in\R^{I}$, a SOR smoothing step consists in~:     
\begin{align*}
 v^{(k)} &\, = \, (\Idt - {\bf 1} \mu) \, (S_w \,v^{(k-1)}
 \,+ \, M^{-1}(r_I - \eta_I)) 
\end{align*}
where ${\bf 1} = (1 \dots 1)^T \in \R^I$, and $\mu \in \R^I$ is a row vector
such that $\mu_{i}=1$ for $i\in S\cap I$, and $\mu_i =0$ otherwise. 
The sequence $(v^{(k)})_{k\ge 0}$ converges to the  
solution of Equation~\eqref{eq-LS-sing-algo7} 
when the $\lim_{k\rightarrow \infty} S_w^{(k)}$ exists,
see~\cite{plemmons} for more details.

For the numerical tests of Section~\ref{section-numeric-pursuit}, we
used the LU solver of the package~\cite{superlu99} in both cases.

\end{document}

%% file: degdata-tex.tex
\begin{picture}(0,0)%
\includegraphics{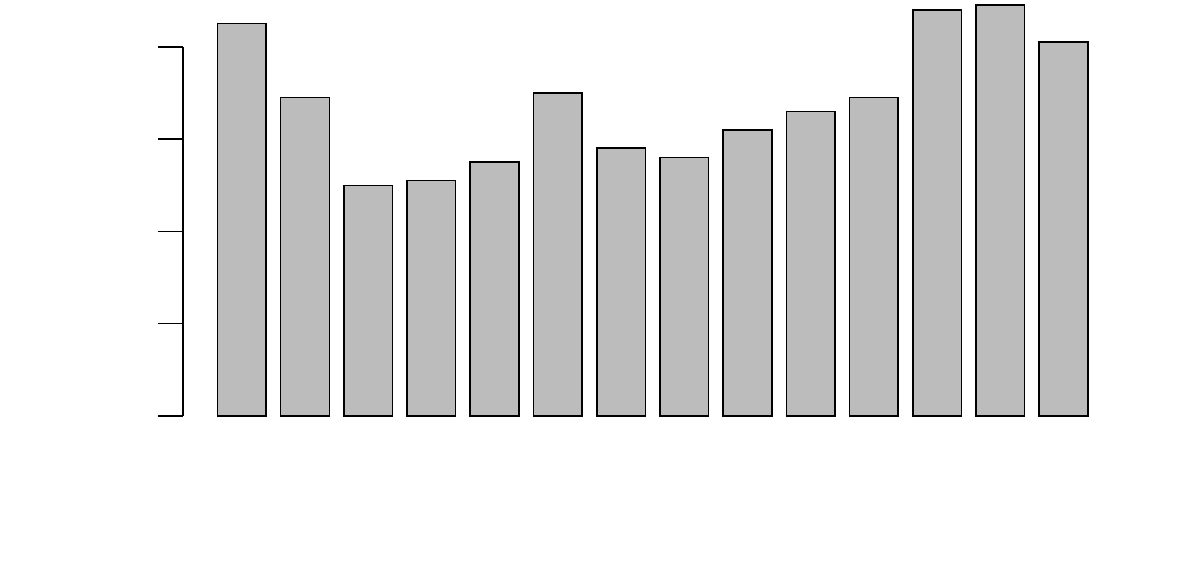}%
\end{picture}%
\setlength{\unitlength}{3947sp}%
\begingroup\makeatletter\ifx\SetFigFont\undefined%
\gdef\SetFigFont#1#2{%
  \fontsize{#1}{#2pt}%
  \selectfont}%
\fi\endgroup%
\begin{picture}(5733,2704)(3708,-5836)
\put(4868,-5568){\makebox(0,0)[b]{\smash{{\SetFigFont{12}{14.4}{\color[rgb]{0,0,0}$1000$}%
}}}}
\put(5475,-5568){\makebox(0,0)[b]{\smash{{\SetFigFont{12}{14.4}{\color[rgb]{0,0,0}$3000$}%
}}}}
\put(6082,-5568){\makebox(0,0)[b]{\smash{{\SetFigFont{12}{14.4}{\color[rgb]{0,0,0}$5000$}%
}}}}
\put(6689,-5568){\makebox(0,0)[b]{\smash{{\SetFigFont{12}{14.4}{\color[rgb]{0,0,0}$7000$}%
}}}}
\put(7296,-5568){\makebox(0,0)[b]{\smash{{\SetFigFont{12}{14.4}{\color[rgb]{0,0,0}$9000$}%
}}}}
\put(8206,-5568){\makebox(0,0)[b]{\smash{{\SetFigFont{12}{14.4}{\color[rgb]{0,0,0}$30000$}%
}}}}
\put(3817,-4140){\rotatebox{90.0}{\makebox(0,0)[b]{\smash{{\SetFigFont{12}{14.4}{\color[rgb]{0,0,0}Frequency}%
}}}}}
\put(4297,-5117){\rotatebox{90.0}{\makebox(0,0)[b]{\smash{{\SetFigFont{12}{14.4}{\color[rgb]{0,0,0}$0$}%
}}}}}
\put(4297,-4673){\rotatebox{90.0}{\makebox(0,0)[b]{\smash{{\SetFigFont{12}{14.4}{\color[rgb]{0,0,0}$20$}%
}}}}}
\put(4297,-4230){\rotatebox{90.0}{\makebox(0,0)[b]{\smash{{\SetFigFont{12}{14.4}{\color[rgb]{0,0,0}$40$}%
}}}}}
\put(4297,-3787){\rotatebox{90.0}{\makebox(0,0)[b]{\smash{{\SetFigFont{12}{14.4}{\color[rgb]{0,0,0}$60$}%
}}}}}
\put(4297,-3344){\rotatebox{90.0}{\makebox(0,0)[b]{\smash{{\SetFigFont{12}{14.4}{\color[rgb]{0,0,0}$80$}%
}}}}}
\put(6826,-5836){\makebox(0,0)[b]{\smash{{\SetFigFont{12}{14.4}{\color[rgb]{0,0,0}Number of nodes}%
}}}}
\put(9057,-5562){\makebox(0,0)[b]{\smash{{\SetFigFont{12}{14.4}{\color[rgb]{0,0,0}$50000$}%
}}}}
\end{picture}%

%% file: iter1_mean-tex.tex
\begin{picture}(0,0)%
\includegraphics{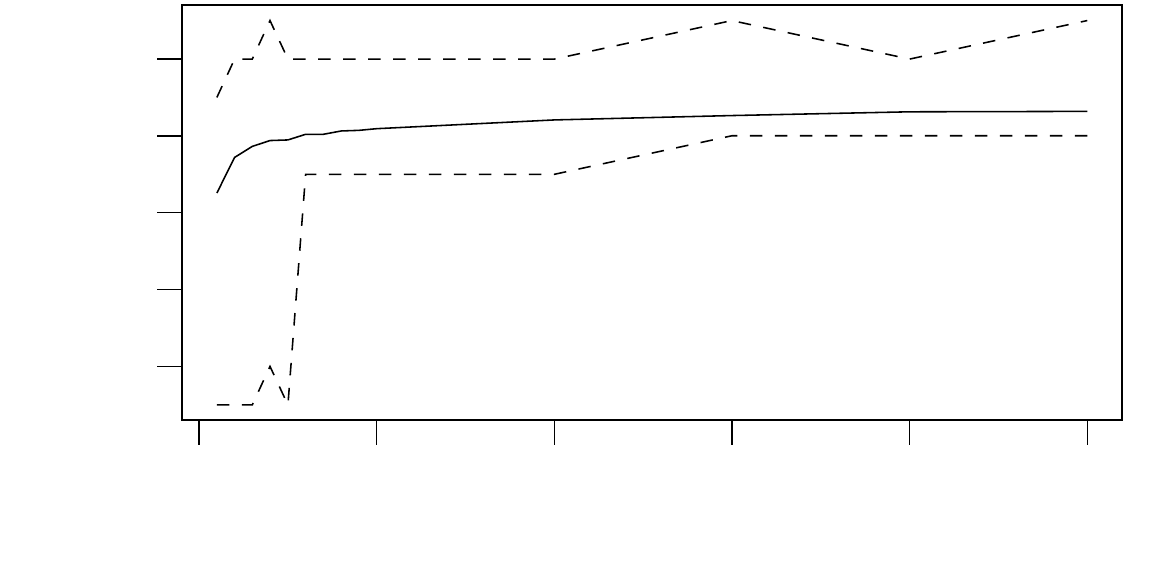}%
\end{picture}%
\setlength{\unitlength}{3947sp}%
\begingroup\makeatletter\ifx\SetFigFont\undefined%
\gdef\SetFigFont#1#2{%
  \fontsize{#1}{#2pt}%
  \selectfont}%
\fi\endgroup%
\begin{picture}(5603,2704)(3711,-5836)
\put(4666,-5568){\makebox(0,0)[b]{\smash{{\SetFigFont{12}{14.4}{\color[rgb]{0,0,0}$0$}%
}}}}
\put(5519,-5568){\makebox(0,0)[b]{\smash{{\SetFigFont{12}{14.4}{\color[rgb]{0,0,0}$10000$}%
}}}}
\put(6372,-5568){\makebox(0,0)[b]{\smash{{\SetFigFont{12}{14.4}{\color[rgb]{0,0,0}$20000$}%
}}}}
\put(7224,-5568){\makebox(0,0)[b]{\smash{{\SetFigFont{12}{14.4}{\color[rgb]{0,0,0}$30000$}%
}}}}
\put(8077,-5568){\makebox(0,0)[b]{\smash{{\SetFigFont{12}{14.4}{\color[rgb]{0,0,0}$40000$}%
}}}}
\put(8930,-5568){\makebox(0,0)[b]{\smash{{\SetFigFont{12}{14.4}{\color[rgb]{0,0,0}$50000$}%
}}}}
\put(4297,-4878){\rotatebox{90.0}{\makebox(0,0)[b]{\smash{{\SetFigFont{12}{14.4}{\color[rgb]{0,0,0}$4$}%
}}}}}
\put(4297,-4509){\rotatebox{90.0}{\makebox(0,0)[b]{\smash{{\SetFigFont{12}{14.4}{\color[rgb]{0,0,0}$6$}%
}}}}}
\put(4297,-4140){\rotatebox{90.0}{\makebox(0,0)[b]{\smash{{\SetFigFont{12}{14.4}{\color[rgb]{0,0,0}$8$}%
}}}}}
\put(4297,-3771){\rotatebox{90.0}{\makebox(0,0)[b]{\smash{{\SetFigFont{12}{14.4}{\color[rgb]{0,0,0}$10$}%
}}}}}
\put(4297,-3403){\rotatebox{90.0}{\makebox(0,0)[b]{\smash{{\SetFigFont{12}{14.4}{\color[rgb]{0,0,0}$12$}%
}}}}}
\put(6826,-5836){\makebox(0,0)[b]{\smash{{\SetFigFont{12}{14.4}{\color[rgb]{0,0,0}Number of nodes}%
}}}}
\put(3817,-4140){\rotatebox{90.0}{\makebox(0,0)[b]{\smash{{\SetFigFont{12}{14.4}{\color[rgb]{0,0,0}Iterations}%
}}}}}
\end{picture}%

%% file: iterTot_mean-tex.tex
\begin{picture}(0,0)%
\includegraphics{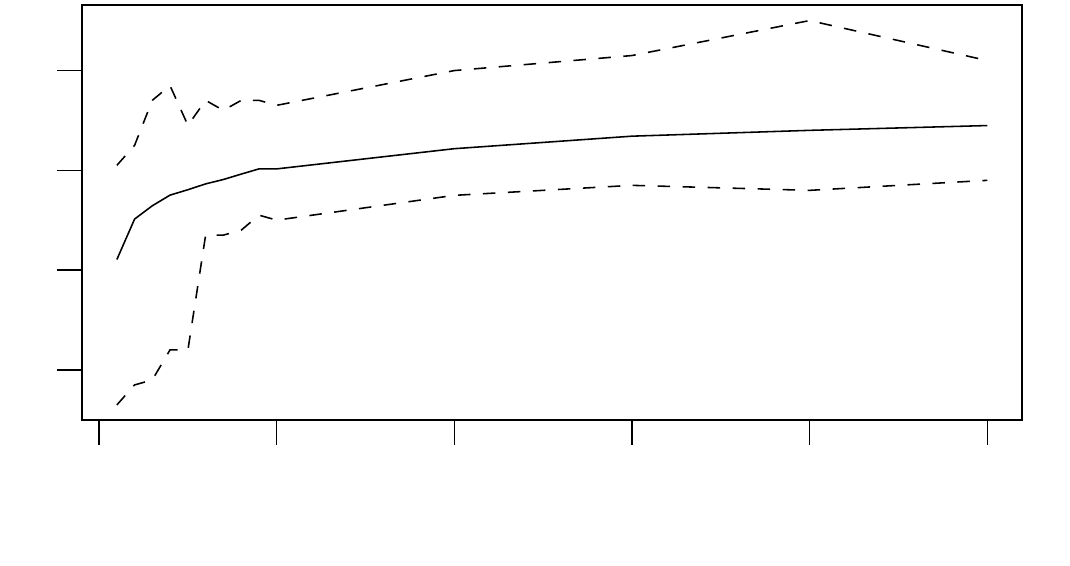}%
\end{picture}%
\setlength{\unitlength}{3947sp}%
\begingroup\makeatletter\ifx\SetFigFont\undefined%
\gdef\SetFigFont#1#2{%
  \fontsize{#1}{#2pt}%
  \selectfont}%
\fi\endgroup%
\begin{picture}(5123,2704)(4191,-5836)
\put(4666,-5568){\makebox(0,0)[b]{\smash{{\SetFigFont{12}{14.4}{\color[rgb]{0,0,0}$0$}%
}}}}
\put(5519,-5568){\makebox(0,0)[b]{\smash{{\SetFigFont{12}{14.4}{\color[rgb]{0,0,0}$10000$}%
}}}}
\put(6372,-5568){\makebox(0,0)[b]{\smash{{\SetFigFont{12}{14.4}{\color[rgb]{0,0,0}$20000$}%
}}}}
\put(7224,-5568){\makebox(0,0)[b]{\smash{{\SetFigFont{12}{14.4}{\color[rgb]{0,0,0}$30000$}%
}}}}
\put(8077,-5568){\makebox(0,0)[b]{\smash{{\SetFigFont{12}{14.4}{\color[rgb]{0,0,0}$40000$}%
}}}}
\put(8930,-5568){\makebox(0,0)[b]{\smash{{\SetFigFont{12}{14.4}{\color[rgb]{0,0,0}$50000$}%
}}}}
\put(4297,-4895){\rotatebox{90.0}{\makebox(0,0)[b]{\smash{{\SetFigFont{12}{14.4}{\color[rgb]{0,0,0}$20$}%
}}}}}
\put(4297,-4416){\rotatebox{90.0}{\makebox(0,0)[b]{\smash{{\SetFigFont{12}{14.4}{\color[rgb]{0,0,0}$40$}%
}}}}}
\put(4297,-3937){\rotatebox{90.0}{\makebox(0,0)[b]{\smash{{\SetFigFont{12}{14.4}{\color[rgb]{0,0,0}$60$}%
}}}}}
\put(4297,-3458){\rotatebox{90.0}{\makebox(0,0)[b]{\smash{{\SetFigFont{12}{14.4}{\color[rgb]{0,0,0}$80$}%
}}}}}
\put(6826,-5836){\makebox(0,0)[b]{\smash{{\SetFigFont{12}{14.4}{\color[rgb]{0,0,0}Number of nodes}%
}}}}
\end{picture}%

%% file: time_mean-tex.tex
\begin{picture}(0,0)%
\includegraphics{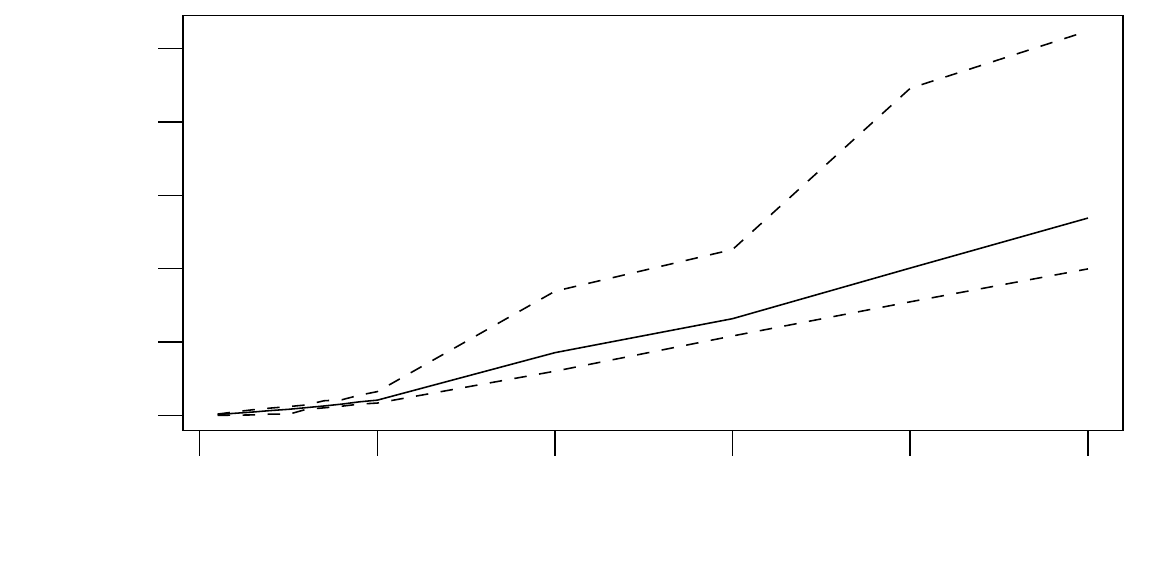}%
\end{picture}%
\setlength{\unitlength}{3947sp}%
\begingroup\makeatletter\ifx\SetFigFont\undefined%
\gdef\SetFigFont#1#2{%
  \fontsize{#1}{#2pt}%
  \selectfont}%
\fi\endgroup%
\begin{picture}(5606,2764)(3708,-5836)
\put(4666,-5568){\makebox(0,0)[b]{\smash{{\SetFigFont{12}{14.4}{\color[rgb]{0,0,0}$0$}%
}}}}
\put(5519,-5568){\makebox(0,0)[b]{\smash{{\SetFigFont{12}{14.4}{\color[rgb]{0,0,0}$10000$}%
}}}}
\put(6372,-5568){\makebox(0,0)[b]{\smash{{\SetFigFont{12}{14.4}{\color[rgb]{0,0,0}$20000$}%
}}}}
\put(7224,-5568){\makebox(0,0)[b]{\smash{{\SetFigFont{12}{14.4}{\color[rgb]{0,0,0}$30000$}%
}}}}
\put(8077,-5568){\makebox(0,0)[b]{\smash{{\SetFigFont{12}{14.4}{\color[rgb]{0,0,0}$40000$}%
}}}}
\put(8930,-5568){\makebox(0,0)[b]{\smash{{\SetFigFont{12}{14.4}{\color[rgb]{0,0,0}$50000$}%
}}}}
\put(4297,-5063){\rotatebox{90.0}{\makebox(0,0)[b]{\smash{{\SetFigFont{12}{14.4}{\color[rgb]{0,0,0}$0$}%
}}}}}
\put(4297,-4711){\rotatebox{90.0}{\makebox(0,0)[b]{\smash{{\SetFigFont{12}{14.4}{\color[rgb]{0,0,0}$50$}%
}}}}}
\put(4297,-4007){\rotatebox{90.0}{\makebox(0,0)[b]{\smash{{\SetFigFont{12}{14.4}{\color[rgb]{0,0,0}$150$}%
}}}}}
\put(4297,-3303){\rotatebox{90.0}{\makebox(0,0)[b]{\smash{{\SetFigFont{12}{14.4}{\color[rgb]{0,0,0}$250$}%
}}}}}
\put(3817,-4140){\rotatebox{90.0}{\makebox(0,0)[b]{\smash{{\SetFigFont{12}{14.4}{\color[rgb]{0,0,0}Time (s)}%
}}}}}
\put(6826,-5836){\makebox(0,0)[b]{\smash{{\SetFigFont{12}{14.4}{\color[rgb]{0,0,0}Number of nodes}%
}}}}
\end{picture}%

%% file: time1_onlymean-tex.tex
\begin{picture}(0,0)%
\includegraphics{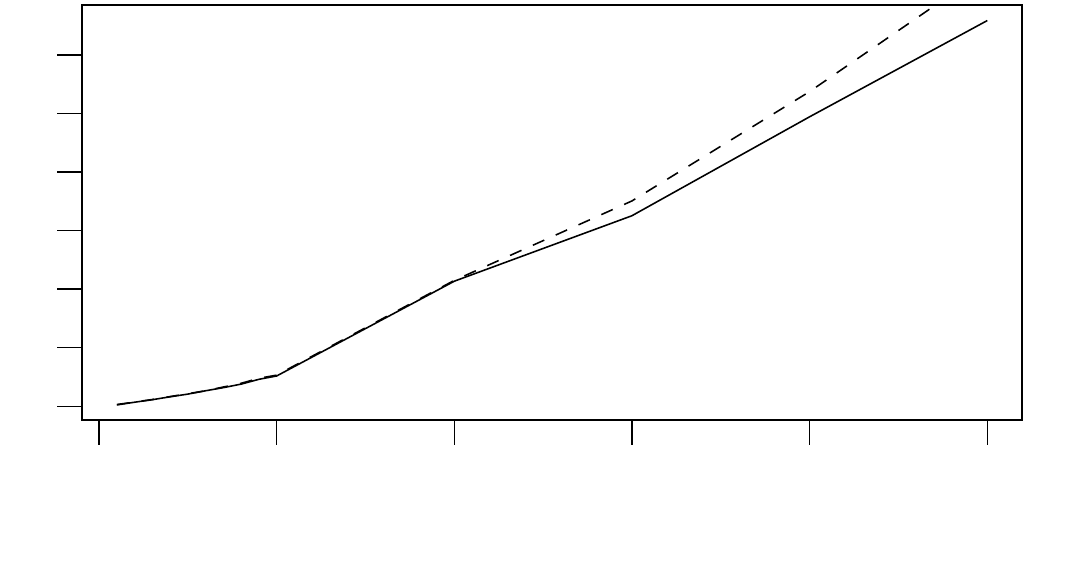}%
\end{picture}%
\setlength{\unitlength}{3947sp}%
\begingroup\makeatletter\ifx\SetFigFont\undefined%
\gdef\SetFigFont#1#2{%
  \fontsize{#1}{#2pt}%
  \selectfont}%
\fi\endgroup%
\begin{picture}(5123,2704)(4191,-5836)
\put(4666,-5568){\makebox(0,0)[b]{\smash{{\SetFigFont{12}{14.4}{\color[rgb]{0,0,0}$0$}%
}}}}
\put(5519,-5568){\makebox(0,0)[b]{\smash{{\SetFigFont{12}{14.4}{\color[rgb]{0,0,0}$10000$}%
}}}}
\put(6372,-5568){\makebox(0,0)[b]{\smash{{\SetFigFont{12}{14.4}{\color[rgb]{0,0,0}$20000$}%
}}}}
\put(7224,-5568){\makebox(0,0)[b]{\smash{{\SetFigFont{12}{14.4}{\color[rgb]{0,0,0}$30000$}%
}}}}
\put(8077,-5568){\makebox(0,0)[b]{\smash{{\SetFigFont{12}{14.4}{\color[rgb]{0,0,0}$40000$}%
}}}}
\put(8930,-5568){\makebox(0,0)[b]{\smash{{\SetFigFont{12}{14.4}{\color[rgb]{0,0,0}$50000$}%
}}}}
\put(4297,-5070){\rotatebox{90.0}{\makebox(0,0)[b]{\smash{{\SetFigFont{12}{14.4}{\color[rgb]{0,0,0}$0$}%
}}}}}
\put(4297,-4507){\rotatebox{90.0}{\makebox(0,0)[b]{\smash{{\SetFigFont{12}{14.4}{\color[rgb]{0,0,0}$40$}%
}}}}}
\put(4297,-3945){\rotatebox{90.0}{\makebox(0,0)[b]{\smash{{\SetFigFont{12}{14.4}{\color[rgb]{0,0,0}$80$}%
}}}}}
\put(4297,-3382){\rotatebox{90.0}{\makebox(0,0)[b]{\smash{{\SetFigFont{12}{14.4}{\color[rgb]{0,0,0}$120$}%
}}}}}
\put(6826,-5836){\makebox(0,0)[b]{\smash{{\SetFigFont{12}{14.4}{\color[rgb]{0,0,0}Number of nodes}%
}}}}
\end{picture}%

%% file: CM0999_v-tex.tex
\begingroup
  \makeatletter
  \providecommand\color[2][]{%
    \GenericError{(gnuplot) \space\space\space\@spaces}{%
      Package color not loaded in conjunction with
      terminal option `colourtext'%
    }{See the gnuplot documentation for explanation.%
    }{Either use 'blacktext' in gnuplot or load the package
      color.sty in LaTeX.}%
    \renewcommand\color[2][]{}%
  }%
  \providecommand\includegraphics[2][]{%
    \GenericError{(gnuplot) \space\space\space\@spaces}{%
      Package graphicx or graphics not loaded%
    }{See the gnuplot documentation for explanation.%
    }{The gnuplot epslatex terminal needs graphicx.sty or graphics.sty.}%
    \renewcommand\includegraphics[2][]{}%
  }%
  \providecommand\rotatebox[2]{#2}%
  \@ifundefined{ifGPcolor}{%
    \newif\ifGPcolor
    \GPcolorfalse
  }{}%
  \@ifundefined{ifGPblacktext}{%
    \newif\ifGPblacktext
    \GPblacktexttrue
  }{}%
  \let\gplgaddtomacro\g@addto@macro
  \gdef\gplbacktext{}%
  \gdef\gplfronttext{}%
  \makeatother
  \ifGPblacktext
    \def\colorrgb#1{}%
    \def\colorgray#1{}%
  \else
    \ifGPcolor
      \def\colorrgb#1{\color[rgb]{#1}}%
      \def\colorgray#1{\color[gray]{#1}}%
      \expandafter\def\csname LTw\endcsname{\color{white}}%
      \expandafter\def\csname LTb\endcsname{\color{black}}%
      \expandafter\def\csname LTa\endcsname{\color{black}}%
      \expandafter\def\csname LT0\endcsname{\color[rgb]{1,0,0}}%
      \expandafter\def\csname LT1\endcsname{\color[rgb]{0,1,0}}%
      \expandafter\def\csname LT2\endcsname{\color[rgb]{0,0,1}}%
      \expandafter\def\csname LT3\endcsname{\color[rgb]{1,0,1}}%
      \expandafter\def\csname LT4\endcsname{\color[rgb]{0,1,1}}%
      \expandafter\def\csname LT5\endcsname{\color[rgb]{1,1,0}}%
      \expandafter\def\csname LT6\endcsname{\color[rgb]{0,0,0}}%
      \expandafter\def\csname LT7\endcsname{\color[rgb]{1,0.3,0}}%
      \expandafter\def\csname LT8\endcsname{\color[rgb]{0.5,0.5,0.5}}%
    \else
      \def\colorrgb#1{\color{black}}%
      \def\colorgray#1{\color[gray]{#1}}%
      \expandafter\def\csname LTw\endcsname{\color{white}}%
      \expandafter\def\csname LTb\endcsname{\color{black}}%
      \expandafter\def\csname LTa\endcsname{\color{black}}%
      \expandafter\def\csname LT0\endcsname{\color{black}}%
      \expandafter\def\csname LT1\endcsname{\color{black}}%
      \expandafter\def\csname LT2\endcsname{\color{black}}%
      \expandafter\def\csname LT3\endcsname{\color{black}}%
      \expandafter\def\csname LT4\endcsname{\color{black}}%
      \expandafter\def\csname LT5\endcsname{\color{black}}%
      \expandafter\def\csname LT6\endcsname{\color{black}}%
      \expandafter\def\csname LT7\endcsname{\color{black}}%
      \expandafter\def\csname LT8\endcsname{\color{black}}%
    \fi
  \fi
  \setlength{\unitlength}{0.0500bp}%
  \begin{picture}(7200.00,5040.00)%
    \gplgaddtomacro\gplbacktext{%
      \csname LTb\endcsname%
      \put(942,1289){\makebox(0,0){\strut{}-0.6}}%
      \put(1482,1190){\makebox(0,0){\strut{}-0.4}}%
      \put(2022,1091){\makebox(0,0){\strut{}-0.2}}%
      \put(2562,992){\makebox(0,0){\strut{} 0}}%
      \put(3101,893){\makebox(0,0){\strut{} 0.2}}%
      \put(3640,794){\makebox(0,0){\strut{} 0.4}}%
      \put(4180,695){\makebox(0,0){\strut{} 0.6}}%
      \put(4398,755){\makebox(0,0){\strut{}-0.6}}%
      \put(4710,927){\makebox(0,0){\strut{}-0.4}}%
      \put(5022,1098){\makebox(0,0){\strut{}-0.2}}%
      \put(5333,1270){\makebox(0,0){\strut{} 0}}%
      \put(5645,1441){\makebox(0,0){\strut{} 0.2}}%
      \put(5957,1613){\makebox(0,0){\strut{} 0.4}}%
      \put(6268,1784){\makebox(0,0){\strut{} 0.6}}%
      \put(920,2070){\makebox(0,0)[r]{\strut{}-140}}%
      \put(920,2242){\makebox(0,0)[r]{\strut{}-120}}%
      \put(920,2413){\makebox(0,0)[r]{\strut{}-100}}%
      \put(920,2585){\makebox(0,0)[r]{\strut{}-80}}%
      \put(920,2755){\makebox(0,0)[r]{\strut{}-60}}%
      \put(920,2926){\makebox(0,0)[r]{\strut{}-40}}%
      \put(920,3098){\makebox(0,0)[r]{\strut{}-20}}%
      \put(920,3269){\makebox(0,0)[r]{\strut{} 0}}%
      \put(920,3441){\makebox(0,0)[r]{\strut{} 20}}%
    }%
    \gplgaddtomacro\gplfronttext{%
    }%
    \gplbacktext
    \put(0,0){\includegraphics{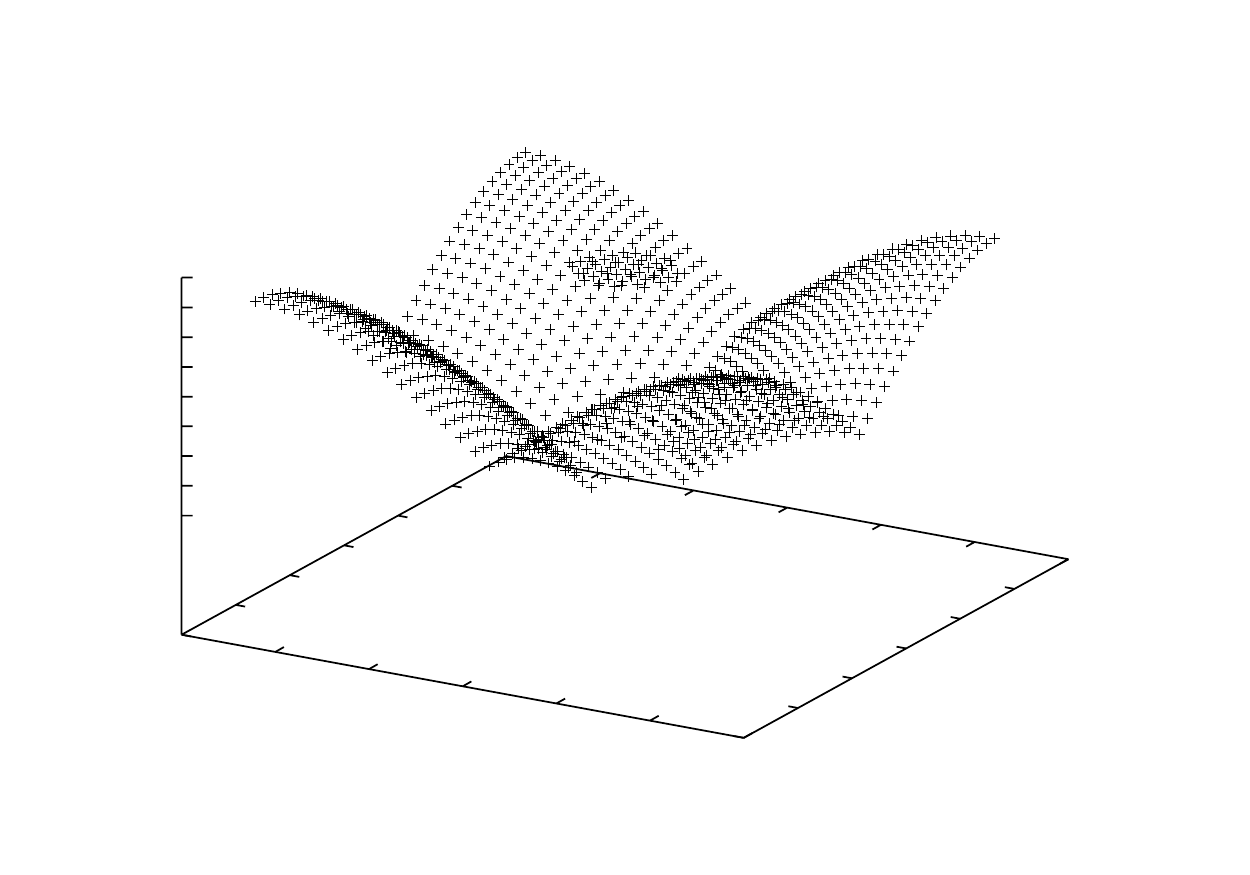}}%
    \gplfronttext
  \end{picture}%
\endgroup

%% file: CM10001_v-tex.tex
\begingroup
  \makeatletter
  \providecommand\color[2][]{%
    \GenericError{(gnuplot) \space\space\space\@spaces}{%
      Package color not loaded in conjunction with
      terminal option `colourtext'%
    }{See the gnuplot documentation for explanation.%
    }{Either use 'blacktext' in gnuplot or load the package
      color.sty in LaTeX.}%
    \renewcommand\color[2][]{}%
  }%
  \providecommand\includegraphics[2][]{%
    \GenericError{(gnuplot) \space\space\space\@spaces}{%
      Package graphicx or graphics not loaded%
    }{See the gnuplot documentation for explanation.%
    }{The gnuplot epslatex terminal needs graphicx.sty or graphics.sty.}%
    \renewcommand\includegraphics[2][]{}%
  }%
  \providecommand\rotatebox[2]{#2}%
  \@ifundefined{ifGPcolor}{%
    \newif\ifGPcolor
    \GPcolorfalse
  }{}%
  \@ifundefined{ifGPblacktext}{%
    \newif\ifGPblacktext
    \GPblacktexttrue
  }{}%
  \let\gplgaddtomacro\g@addto@macro
  \gdef\gplbacktext{}%
  \gdef\gplfronttext{}%
  \makeatother
  \ifGPblacktext
    \def\colorrgb#1{}%
    \def\colorgray#1{}%
  \else
    \ifGPcolor
      \def\colorrgb#1{\color[rgb]{#1}}%
      \def\colorgray#1{\color[gray]{#1}}%
      \expandafter\def\csname LTw\endcsname{\color{white}}%
      \expandafter\def\csname LTb\endcsname{\color{black}}%
      \expandafter\def\csname LTa\endcsname{\color{black}}%
      \expandafter\def\csname LT0\endcsname{\color[rgb]{1,0,0}}%
      \expandafter\def\csname LT1\endcsname{\color[rgb]{0,1,0}}%
      \expandafter\def\csname LT2\endcsname{\color[rgb]{0,0,1}}%
      \expandafter\def\csname LT3\endcsname{\color[rgb]{1,0,1}}%
      \expandafter\def\csname LT4\endcsname{\color[rgb]{0,1,1}}%
      \expandafter\def\csname LT5\endcsname{\color[rgb]{1,1,0}}%
      \expandafter\def\csname LT6\endcsname{\color[rgb]{0,0,0}}%
      \expandafter\def\csname LT7\endcsname{\color[rgb]{1,0.3,0}}%
      \expandafter\def\csname LT8\endcsname{\color[rgb]{0.5,0.5,0.5}}%
    \else
      \def\colorrgb#1{\color{black}}%
      \def\colorgray#1{\color[gray]{#1}}%
      \expandafter\def\csname LTw\endcsname{\color{white}}%
      \expandafter\def\csname LTb\endcsname{\color{black}}%
      \expandafter\def\csname LTa\endcsname{\color{black}}%
      \expandafter\def\csname LT0\endcsname{\color{black}}%
      \expandafter\def\csname LT1\endcsname{\color{black}}%
      \expandafter\def\csname LT2\endcsname{\color{black}}%
      \expandafter\def\csname LT3\endcsname{\color{black}}%
      \expandafter\def\csname LT4\endcsname{\color{black}}%
      \expandafter\def\csname LT5\endcsname{\color{black}}%
      \expandafter\def\csname LT6\endcsname{\color{black}}%
      \expandafter\def\csname LT7\endcsname{\color{black}}%
      \expandafter\def\csname LT8\endcsname{\color{black}}%
    \fi
  \fi
  \setlength{\unitlength}{0.0500bp}%
  \begin{picture}(7200.00,5040.00)%
    \gplgaddtomacro\gplbacktext{%
      \csname LTb\endcsname%
      \put(942,1289){\makebox(0,0){\strut{}-0.6}}%
      \put(1482,1190){\makebox(0,0){\strut{}-0.4}}%
      \put(2022,1091){\makebox(0,0){\strut{}-0.2}}%
      \put(2562,992){\makebox(0,0){\strut{} 0}}%
      \put(3101,893){\makebox(0,0){\strut{} 0.2}}%
      \put(3640,794){\makebox(0,0){\strut{} 0.4}}%
      \put(4180,695){\makebox(0,0){\strut{} 0.6}}%
      \put(4398,755){\makebox(0,0){\strut{}-0.6}}%
      \put(4710,927){\makebox(0,0){\strut{}-0.4}}%
      \put(5022,1098){\makebox(0,0){\strut{}-0.2}}%
      \put(5333,1270){\makebox(0,0){\strut{} 0}}%
      \put(5645,1441){\makebox(0,0){\strut{} 0.2}}%
      \put(5957,1613){\makebox(0,0){\strut{} 0.4}}%
      \put(6268,1784){\makebox(0,0){\strut{} 0.6}}%
      \put(920,2070){\makebox(0,0)[r]{\strut{} 0}}%
      \put(920,2223){\makebox(0,0)[r]{\strut{} 10}}%
      \put(920,2375){\makebox(0,0)[r]{\strut{} 20}}%
      \put(920,2527){\makebox(0,0)[r]{\strut{} 30}}%
      \put(920,2679){\makebox(0,0)[r]{\strut{} 40}}%
      \put(920,2831){\makebox(0,0)[r]{\strut{} 50}}%
      \put(920,2984){\makebox(0,0)[r]{\strut{} 60}}%
      \put(920,3136){\makebox(0,0)[r]{\strut{} 70}}%
      \put(920,3288){\makebox(0,0)[r]{\strut{} 80}}%
      \put(920,3441){\makebox(0,0)[r]{\strut{} 90}}%
    }%
    \gplgaddtomacro\gplfronttext{%
    }%
    \gplbacktext
    \put(0,0){\includegraphics{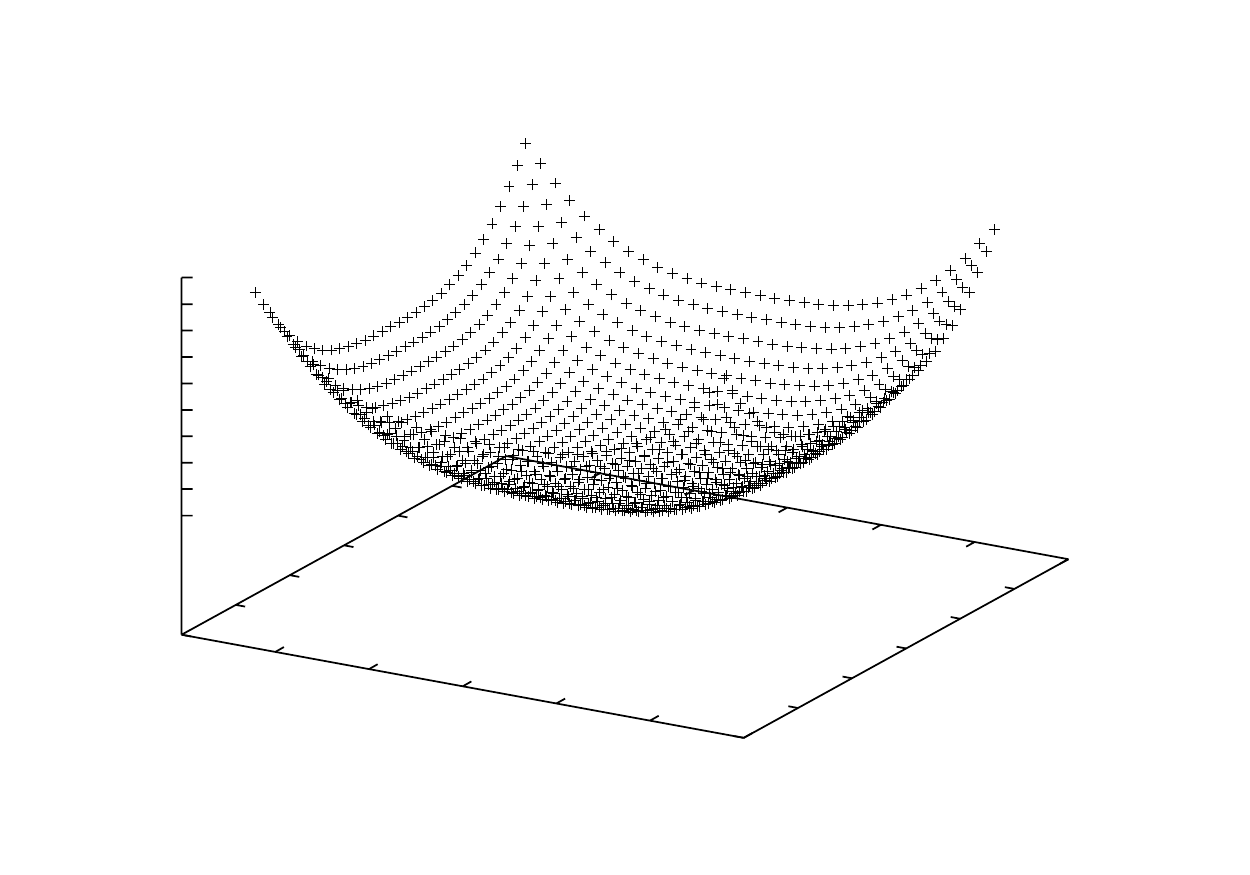}}%
    \gplfronttext
  \end{picture}%
\endgroup